\documentclass[12pt]{amsart}
\usepackage{geometry}                
\usepackage{pdftricks}
\begin{psinputs}
\usepackage{pstricks}
\usepackage{multido}
\end{psinputs}
\usepackage{graphicx}
\usepackage{amssymb}
\usepackage{amsthm}
\usepackage{epstopdf}
\usepackage{enumerate}
\usepackage{pst-all}
\newtheorem{theorem}{Theorem}[section]
\newtheorem{lemma}[theorem]{Lemma} 
\newtheorem{remark}[theorem]{Remark} 
\newtheorem{definition}[theorem]{Definition} 
\newtheorem{corollary}[theorem]{Corollary} 
\newtheorem{claim}[theorem]{Claim} 
\newtheorem{example}[theorem]{Example}
\def\da{\downarrow}
\def\nda{\not\downarrow}
\def\A{\mathcal{A}}
\def\B{\mathcal{B}}
\def\C{\mathcal{C}}
\def\D{\mathcal{D}}
\def\K{\mathcal{K}}
\def\M{\mathbb{M}}
\def\a{\alpha}
\def\g{\gamma}
\def\o{\omega}
\def\raj{\upharpoonright}

\title{Quasiminimal structures, groups and Zariski-like geometries}

\begin{document}
  
\thanks{Research of the first author was partially supported by grant 40734 of the Academy of Finland}
\thanks{Research of the second author was supported by Finnish National Doctoral Programme in Mathematics and its Applications}

\maketitle

\begin{center}
Tapani Hyttinen \\
Department of Mathematics, University of Helsinki \\
P.O. Box 68, 00014, Finland \\
tapani.hyttinen@helsinki.fi\\
\vspace{5 mm}
 Kaisa Kangas \\
Department of Mathematics, University of Helsinki \\
P.O. Box 68, 00014, Finland \\ 
kaisa.kangas@helsinki.fi \\
+358 40 561 1543
\vspace{5 mm}
\end{center}

\begin{abstract}
We generalize Hrushovski's Group Configuration Theorem to quasiminimal classes.
As an application, we present Zariski-like structures, a generalization of Zariski geometries, and show that a group can be found there if the pregeometry obtained from the bounded closure operator is non-trivial.
\end{abstract}
 
\vspace{5 mm}
\begin{center}
Mathematics Subject Classification: 03C45, 03C48, 03C50, 03C98 \\
Key words: AECs, strong minimality, group configuration 
\end{center}
 
\tableofcontents
 
\section{Introduction}

We will study quasiminimal classes, i.e. abstract elementary classes (AECs) that arise from a quasiminimal pregeometry structure. 
Abstract elementary classes were introduced to model theory by Shelah \cite{shelahaec}, and they provide the standard framework for the study of non-elementary classes. 
Quasiminimal classes originate from Zilber's \cite{Ziex} work, 
where he introduced quasiminimal excellent classes in order to prove categoricity of the non-elementary theory of pseudoexponential fields.

A quasiminimal class is an AEC that can be constructed from a infinite dimensional quasiminimal pregeometry structure (see \cite{monet} for details).
Quasiminimal pregeometry structures provide an analogue to the strongly minimal first order setting.
In the latter, a canonical pregeometry is obtained from the model theoretic algebraic closure operator, and ranks can be calculated as dimensions in the pregeometry.
In the quasiminimal case, a pregeometry is obtained from the bounded closure operator (defined as Galois types having only boundedly many realizations, denoted bcl).
A good exposition on quasiminimal pregeometries can be found in \cite{kir} and another one in \cite{baldwin}. 

Quasiminimal classes are uncountably categorical and have arbitrarily large models.
They have both AP and JEP and thus a universal model homogeneous monster model.
They are also excellent in the sense of Zilber (this is weaker than the original notion of excellence due to Shelah, see below).
 
Our main result, Theorem \ref{groupmain}, is a generalization of Hrushovski's group configuration theorem to this setting.
Hrushovski proved the group configuration theorem as a part of his Ph.D. thesis \cite{hrthesis}.  
It has been the source of many applications of model theory to other fields of mathematics.
The theorem holds for stable first order theories.
It states that whenever there is a certain configuration of elements in a model, a group can be interpreted there.
The proof can be found in e.g. \cite{pi}.
In the present paper, we generalize it to the non-elementary setting of quasiminimal classes.
 
For the proof, we need an independence calculus that works in our context, so we develop it in Chapter 2.
Independence in AECs has been previously studied under various assumptions. 
For instance, Hyttinen and Lessman \cite{HyLess} developed an independence notion for excellent (in the sense of Shelah) classes and showed that under the further assumption of simplicity, it satisfies the usual properties of non-forking.
Hyttinen and Kes\"al\"a \cite{meeri1, meeri2} developed an independence calculus under the related assumption that an AEC be finitary. 
There are also various other approaches to independence in AECs, some of them under very general assumptions (see e.g. \cite{aecmonet}).
 
In the most general approaches, independence is considered mainly over models, usually rather saturated ones.
However, we take a different direction and develop our independence notion with independence over finite tuples in  mind.
Indeed, this is what is needed for proving the group configuration theorem. 
For this purpose, we introduce FUR-classes (``Finite $U$-Rank").
The main reason for introducing these classes is to have a context in which certain variants of quasiminimal classes can be treated uniformly.
We note that the previous approaches from \cite{meeri1} and \cite{HyLess} don't suffice for this, since all quasiminimal classes are not excellent in the sense of Shelah (as assumed in \cite{HyLess}) nor finitary in the sense of \cite{meeri1}. 
A counterexample for both criteria can be found in Example \ref{equivalence}.
In the elementary context, a model class of a first order $\o$-stable theory with finite Morley rank (with $\preccurlyeq$ the first order elementary submodel relation) provides an example of a FUR-class.
  
When developing the theory of independence, we use ideas from \cite{HyLess} and \cite{HyLess2}, and thus also from \cite{Hy2}.
However, the classes treated there are excellent in the sense of Shelah, and quasiminimal classes don't satisfy all the hypothesis made there, so we cannot directly apply the results.
The basic idea is that we first assume that we have a class with some kind of independence relation (given in terms of non-splitting) and then show that we actually also have another independence notion (given in terms of Lascar splitting, see Definition \ref{freedom}) that has all the properties we could expect from and independence notion (Theorem \ref{main}).  
  
Hrushovski's group configuration theorem yields a group that consists of imaginaries (equivalence classes in definable equivalence relations), and canonical bases are used in the proof. 
Thus, in addition to having a theory of independence in the monster model $\M$, we need to construct $\M^{eq}$ and show that the theory can be applied there as well. 
We do this in section 2.4.
In our context, $\M^{eq}$ cannot be constructed so that it is both $\o$-stable (in the sense of AECs) and has elimination of imaginaries. 
Since $\o$-stability is vital for our arguments, we build the theory (using ideas from \cite{HyLess2}) so that we can always move from $\M$ to $\M^{eq}$ and then, if needed, to $(\M^{eq})^{eq}$ etc.
We then show that if $\M$ is a monster model for a FUR-class, then so is $\M^{eq}$.
At the end of section 2.4, we show that we have canonical bases and that they have the usual properties one would expect.
 
The main reason for presenting FUR-classes is to prove that quasiminimal classes have a perfect theory of independence, and in section 2.5 we show that every quasiminimal class is indeed a FUR-class.
Moreover, we note that in quasiminimal classes, ranks can be calculated as dimensions in the pregeometry obtained from the bounded closure operator.
However, we also give an example of a non-elementary class that is a FUR-class but does not arise from a quasiminimal pregeometry structure (Example \ref{suorasumma}).

 
In Chapter 3, we generalize Hrushovski's group configuration theorem to our setting.
The original theorem states that if there is a certain configuration of six tuples $a,b,c,x,y,z$ in a model, then the model interprets a group. 
The idea is roughly that the tuples $x$, $y$, and $z$ are seen as points on a plane, and the tuples $a$, $b$, and $c$ as transformation functions that move the points around.
Hrushovski shows that (after making some small modifications) $y$ and $z$ are interdefinable over $b$, 
and then treats $b$ as a function sending $y$ to $z$.
In our setting, it is much more difficult to prove the interdefinability than in the first order context, so the construction becomes more complicated. 
Another difference is that quite soon in the beginning of the proof, Hrushovski replaces the tuples $x$, $y$, and $z$ with their finite sets of conjugates.
However, in our context, the conjugates are defined in terms of bounded closure, and unlike in the first order case, the set of conjugates can be (countably) infinite.
Since we cannot take arbitrary countable sets as elements of $\M^{eq}$, Hrushovski's trick cannot be used.
To overcome the problem, we move from the pregometry to the canonical geometry associated to it and work there.
 
In Chapter 4, we look at possibilities of generalizing Zariski geometries (see \cite{HrZi}) to our setting.
After presenting a generalization, \emph{Zariski-like structures}, we apply our group configuration theorem to show that a group can be found there as long as the canonical pregeometry obtained from the bounded closure operator is non-trivial (Theorem \ref{grouptheorem}).
 
Zariski geometries were introduced by Hrushovski and Zilber \cite{HrZi93, HrZi}, and Hrushovski \cite{mordelllang}
made use of them in his model theoretic proof of the geometric Mordell-Lang conjecture.
Zariski geometries generalize the notion of an algebraically closed field $k$ together with the Zariski topologies on each $k^n$.
An expository introduction to Zariski geometries can be found in \cite{MarZar}.

According to Zilber \cite{Zi}, the main motivation behind this notion was to identify all classes where Zilber's trichotomy principle holds.
In \cite{HrZi}, Hrushovski and Zilber showed that it holds for Zariski geometries: there every non locally modular strongly minimal set interprets an algebraically closed field.
The proof is carried out by first formulating Hrushovski's group configuration theorem in terms of certain kind of infinite arrays that are called \emph{indiscernible arrays}, then using it to find a $1$-dimensional group, and finally using the group elements to find a field.  
In finding the group and the field, families of plane curves play an important role.

In Chapter 4 of the present paper, we follow the lines of the first part of Hrushovski's and Zilber's proof to find a group in a Zariski-like structure.
In 4.1, we present a notion of plane curve that works in our setting, in 4.2 we reformulate our group configuration theorem in terms of indiscernible arrays (Lemma \ref{groupexists}), and finally in 4.3 we adopt Hrushovski's and Zilber's proof to our setting and show that a suitable plane curve can be used to find the group configuration.
The work is continued in a forthcoming paper by the second author \cite{oma}, where it is shown that an analogue for Zilber's trichotomy holds in a Zariski-like structure: a non locally modular Zariski-like structure interprets either an algebraically closed field or a non-classical group (see \cite{HLS}).
   
In \cite{HrZi}, a Zariski geometry is defined as a set $D$ together with Noetherian topologies on each $D^n$, satisfying four axioms labeled (Z0)-(Z3).
It follows from the axioms that the set $D$ must be strongly minimal.
In this elementary framework, there is a canonical way of getting a class of structures from this single structure.
However, a similar approach cannot be used for AECs.
Thus, since quasiminimality provides a natural analogue for strong minimality, we work with quasiminimal classes and formulate our axioms in this context without using a single structure as a starting point.
  
Instead of assuming  the presence of a topology, we define a Zariski-like structure as a quasiminimal pregeometry structure together with a collection $\mathcal{C}$ of Galois definable sets satisfying nine axioms.
We call these sets ``irreducible" since they generalize the properties of irreducible closed sets in the Zariski geometry context. 
Indeed, Zariski geometries satisfy our axioms if the collection $\mathcal{C}$ is chosen to consist of the irreducible $\emptyset$-closed sets.
The more general notion of a closed set could also be useful, as can be seen in the example of a cover of the multiplicative group of an algebraically closed field (see \cite{cover}).
However, as the example demonstrates, the natural topologies do not seem to be Noetherian, 
and thus it is hard to find a natural notion of an irreducible set.
Hence, we don't feel our insight is strong enough to formulate the axioms for arbitrary closed sets.
In \cite{Zi}, Zilber has given one axiomatization for closed sets in a non-elementary case, which he calls analytic Zariski structures, but we have chosen a somewhat different route.

Partially because of not using the more general concept of a closed set, some of our axioms come from Assumptions 6.6. in \cite{HrZi} rather than from the axiomatization (Z0)-(Z3) for Zariski geometries given in the very beginning of \cite{HrZi}.
In our axiomatization, axioms (ZL1)-(ZL6) give meaning to the key axioms (ZL7)-(ZL9).
If, in (ZL9), we take $\kappa$ to be finite and choose $S=\{\kappa\}$, then we get just the axiom (Z3) of Zariski geometries (the dimension theorem).
In the elementary case, (ZL9) is the immediate consequence of (Z3) and Compactness.
The axioms (ZL7) and (ZL8) come from Assumptions 6.6. in \cite{HrZi}.

The main result in \cite{HrZi} is that every very ample Zariski geometry arises from the Zariski topology on some algebraic curve over an algebraically closed field.
One  example of a Zariski-like geometry is the cover of an algebraically closed field of characteristic zero, treated in \cite{cover}.
An analogue to the result from \cite{HrZi} might be that all non-trivial Zariski-like structures in some sense resemble this example.
This would mean that on the level of the canonical geometry we would be in some sense back in the elementary case
(inside the closure of a single point the structure may be complicated).  
A result like this would be in line with the existing studies of geometries in non-elementary cases.
However, since the existence of a non-classical group (see \cite{HLS} and \cite{Hy3} for locally modular cases) is still open, to prove something like this seems very difficult, and if it turns out that there are non-classical groups, the playground is completely open.
Another open question is whether B. Zilber's pseudo-exponentiation and quantum algebras (not at a root of unity) satisfy our axioms for Zariski-like geometries.

All results presented can also be found in \cite{lisuri}, and we refer the reader there for omitted details.
  
\section{Independence in Abstract Elementary Classes}
 
In this chapter, we will develop an independence notion within the context of abstract elementary classes satisfying certain requirements.  
The ideas used originate from \cite{HyLess} and \cite{Hy2}.
  
We suppose that $(\mathcal{K} ,\preccurlyeq )$ is an AEC with AP and JEP and with arbitrarely large structures, $LS(\mathcal{K} )=\omega$ and $\mathcal{K}$ does not contain finite models.
Let $\M$ be a monster model for $\mathcal{K}$, i.e. a $\delta$- model homogeneous and $\delta$-universal model of $\mathcal{K}$ for a large enough cardinal $\delta$.
  
We will list six axioms (AI-AVI) and show that if these axioms hold for $\mathcal{K}$, then Lascar non-splitting will satisfy the usual properties of an independence notion.
We will then give the name \emph{FUR-classes} (``Finite $U$-Rank") to AECs satisfying these requirements.
The axioms are designed with quasiminimal classes in mind, and their main purpose is to obtain a full independence calculus.

In the AEC setting, Galois types, defined as orbits of automorphisms, provide a natural analogue for first order types.
However, as our main notion of type, we will use weak types.

\begin{definition}
Suppose $A \subset \mathbb{M}.$
We denote by $\textrm{Aut}(\mathbb{M}/A)$ the subgroup of the automorphism group of $\mathbb{M}$ consisting of those automorphisms $f$ that satisfy $f(a)=a$ for each $a \in A$.

We say that $a$ and $b$ have the same \emph{Galois type} over $A$ if there is some $f \in \textrm{Aut}(\mathbb{M}/A)$ such that $f(a)=b$.
We write $t^{g}(a/A)=t^{g}(a/A;\mathbb{M} )$ for the Galois-type of $a$ over $A$.

We say that $a$ and $b$ have the same $\emph{weak type}$ over $A$ if for all finite subsets $B \subseteq A$, it holds that $t^{g}(a/B)=t^{g}(b/B)$.
We write $t(a/A)$ for the weak type of $a$ over $A$.
\end{definition}

We first define independence in terms of non-splitting.
In section 2.1, we will give the axioms AI-AVI for a FUR-class and look at the properties that non-splitting independence has under these axioms.
We will see that types over models have unique free extensions, that symmetry and transitivity hold over models, that the setting is $\o$-stable (in the sense of AECs), that weak types over models determine Galois types, and that there are no infinite splitting chains of models.
We then define $U$-ranks over models and finite sets in terms of non-splitting (Definition \ref{Urank}), and show that $U$-rank is preserved in free extensions over models.

However, in section 2.3 we will define another independence notion in terms of Lascar splitting (see Definition \ref{freedom}) and extend our definition of $U$-rank to cover ranks over arbitrary sets.
This will be our main independence notion, and the properties of non-splitting independence are only used to show that the main independence notion has all the usual properties of non-forking (listed in Theorem \ref{main}).

\begin{definition}
Let $A$ and $B$ be sets such that $A \subseteq B$ and $A$ is finite.
We say that $t(a/B)$ \emph{splits} over $A$ if there are $b,c \in B$ such that $t(b/A)=t(c/A)$ but $t(ab/A) \neq t(ac/A)$.

We write $a\da^{ns}_{B}C$ if there is
some finite $A\subseteq B$ such that $t(a/B \cup C )$ does not split
over $A$. By $A\da^{ns}_{B}C$ we mean that $a\da^{ns}_{B}C$
for each $a\in A$. 
\end{definition}

We note that if $A \subseteq B \subseteq C$ for some finite $B$, and $t(a/C)$ does not split over $A$, then it is easy to see $t(a/C)$ does not split over $B$ either.

\subsection{FUR-classes} 
 
In this section, we will introduce FUR-classes and study the properties of non-splitting independence there.
The main reason for introducing the notion is to show that quasiminimal classes (see section 2.5) have a perfect theory of independence.
Indeed, we will use the properties of non-splitting independence to show that in FUR-classes, an independence notion obtained from Lascar splitting (see Definition \ref{freedom}) has all the usual properties of non-forking (Theorem \ref{main}).
It will then turn out that every quasiminimal class is a FUR-class (Lemma \ref{qmfur}). 

A FUR-class will be defined using six axioms, AI-AVI.
For the sake of readability, instead of first presenting all the definitions needed and then giving the axioms in the form of a simple list, we will start listing the axioms and give the related definitions, lemmas and remarks in midst of them.

The reader can check that the model class of any $\o$-stable first order theory with finite Morley rank satisfies the axioms.
We also give the following two non-elementary examples.
The first one is also an example of a quasiminimal class (see section 2.5), and it was originally presented in \cite{baldwin}.
The second example, on the other hand, is a FUR-class that does not arise from a quasiminimal pregeometry.
It is easy to check that the examples satisfy our axioms, but the details for Example \ref{equivalence} can be found in \cite{lisuri}.
 sucthe
\begin{example}\label{equivalence}
Let $\K$ be the class of all models $M=(M,E)$ such that $E$ is an equivalence relation on $M$ with infinitely many classes, each of size $\aleph_0$.
For any set $X$, we define the closure of $X$ to be 
$$\textrm{cl}(X)=\bigcup\{x/E \, | \, x \in XÊ\}.$$
We define $\preccurlyeq$ so that $\A \preccurlyeq \B$ if $\A \subseteq \B$ and $\A=\textrm{cl}(\A)$.
\end{example} 
  
\begin{example}\label{suorasumma}
Let $\K$ be the class of infinite direct sums of the Abelian groups $(\mathbb{Z}/p\mathbb{Z})^{(\lambda_p)}$, 
where $p$ ranges over all primes, and $\lambda_p$ is some cardinal associated to the prime $p$.
Define $\preccurlyeq$ so that $\A \preccurlyeq \B$ if $\A$ is a subgroup of $\B$.

Note that this class is not first order definable.
Indeed, all the models contain only elements of finite order, but since there are elements of arbitrarily large order in $\mathbb{G}$, Compactness would imply the existence of non-torsion elements.
\end{example}

The first axiom states that models are $\aleph_0$- Galois saturated.

\textbf{AI:
Every countable model $\A \in \K$ is s-saturated, i.e. for any $b \in \M$ and any finite $A \subseteq \A$, there is $a \in \A$ such that $t(a/A)=t(b/A)$.}
  
We now apply AI to show that for non-splitting independence, free extensions of types over models are unique.  
\begin{lemma}\label{vaplaajykskas}
Let $\B$ be a model.
If $a \da^{ns}_\B A$, $b \da^{ns}_\B A$ and $t(a/\B)=t(b/\B)$, then  $t(a/A)=t(b/A)$.
\end{lemma}
 
\begin{proof}
Let $c \in A$ be arbitrary.
We show that $t(ac/\emptyset)=t(bc/\emptyset)$.
Let $B \subset \B$ be a finite set such that neither $t(a/ \B \cup A)$ nor $t(b/ \B \cup A)$ splits over $B$.
 By AI, there is some $d \in \B$ such that $t(d/B)=t(c/B)$. 
We have
$$t(ac/\emptyset)=t(ad/\emptyset)=t(bd/\emptyset)=t(bc/\emptyset).$$
\end{proof}
 
\begin{lemma}\label{vaplaajol}
Suppose  $\A$ and $\B$ are countable models, $t(a/\A )$ does not split over some finite $A\subseteq\A$,
and   $\A \subseteq \B$.
Then there is
some $b$ such that  $t(b/\A )=t(a/\A )$ and $b\da^{ns}_{A} \B$.
\end{lemma}

\begin{proof} 
As both $\A$ and $\B$ are countable and contain $A$, we can, using AI and back-and-forth methods, construct an automorphism $f\in Aut(\M /A)$ such that $f(\A )=\B$.
Choose $b=f(a).$
\end{proof}

Axioms AII and AIII together will guarantee that unique prime models exist.
To be able to state AII, we need the notion of $s$-primary models.

\begin{definition}
We say that a model $\B =\A a\cup\bigcup_{i<\o}a_{i}$, where $a_i$ is a singleton for each $i$,
is \emph{$s$-primary} over $\A a$ if for all $n<\o$, there is
a finite $A_{n}\subset \A$
such that for all $(a', a_0', \ldots, a_n') \in \M$ such that $t(a'/\A)=t(a/\A)$, $t(a', a_0', \ldots, a_n'/A_{n})=t(a, a_0, \ldots, a_n/A_{n})$
implies  
$t(a', a_0', \ldots, a_n'/\A)=t(a, a_0, \ldots, a_n/\A)$
\end{definition}
 
\textbf{AII: For all $a$ and countable $\A$, there is
an $s$-primary model
$\B =\A a\cup\bigcup_{i<\o}a_{i}$ ($\le\M$) over $\A a$}.

We denote a countable $s$-primary model $\B= \A a\cup\bigcup_{i<\o}a_{i}$ over $\A a$ that is as above by $\A [a]$.

We will use AII to show that weak types over countable models determine Galois types.
For this, we need the following lemma.
 
 \begin{lemma}\label{primary}
 Let $\A$ be a countable model, and let $t(b/\A )=t(a/\A )$. 
Then, there is an isomorphism $f:\A[a] \rightarrow\A[b]$ such that $f\raj\A =id$ and $f(a)=b$.
\end{lemma}

\begin{proof}
Let $\A[a]=\A a\cup\bigcup_{i<\o}a_{i}$ and $\A[b]=\A b\cup\bigcup_{i<\o}b_{i}$.
Now there is some finite $A_0 \subset \A$ such that it holds for any $a', a_0'$ that if $t(a/\A)=t(a'/\A)$ and $t(a',a_0'/A_0)=t(a,a_0/A_0)$, then $t(a,a_0/\A)=t(a',a_0'/\A)$.
As $t(b/\A )=t(a/\A )$, there is an automorphism $F \in \textrm{Aut}(\mathbb{M}/\A)$ such that $F(a)=b$.
Let $a_0'=F(a_0)$.
By AI, there is some $i$ such that $t(b_i/A_0 b)=t(a_0'/A_0 b)$, and in particular $t(b_i, b/A_0)=t(a_0, a/A_0)$.
Thus, $t(b_i, b / \A)=t(a_0, a/ \A)$.
Now we can construct $f$ using back and forth methods.
\end{proof} 
 
\begin{corollary}
If $\A$ is a countable model, then $t(a/\A )$ determines $t^{g}(a/\A )$.
\end{corollary}

\begin{definition}
We say $a$ \emph{dominates} $B$ over $A$ if the following holds for all $C$: If there is a finite $A_0\subseteq A$ such that $t(a/AC)$ does not split over $A_0$, then $B \da^{ns}_{A}C$.
\end{definition}

\begin{lemma}\label{domination}
If $\A$ is a countable model, then the element $a$ dominates $\A[a]$ over $\A$. 
\end{lemma}

\begin{proof}
Let $A \subset \A$ be finite, and let $B$ be such that $t(a/\A B)$ does not split over $A$.
It suffices to show that for each $n$, it holds that
$$a, a_0, \ldots, a_n \da^{ns}_{\A} B.$$  
We make a counterassumption and suppose that $n$ is the least number such that
$$a, a_0, \ldots, a_n \nda^{ns}_{\A} B.$$
Let $C \subset \A$ be a finite set so that $A \subseteq C$, $A_\gamma  \subseteq C$ for each $\g \le n$, and $t(a, a_0, \ldots, a_{n-1}/ \A B)$ does not split over $C$.
By the counterassumption, there are  $c,d \in \A \cup B$ such that $t(c/C)=t(d/C)$ but $t(c, a, a_0, \ldots, a_n/C) \neq t(d, a, a_0, \ldots, a_n/C)$.
By AI, there is some $d' \in \A$ so that $t(d'/C)=t(d/C)$.
Then, either  $t(d', a, a_0, \ldots, a_n/C) \neq t(d, a, a_0, \ldots, a_n/C)$ or $t(d', a, a_0, \ldots, a_n/C) \neq t(c, a, a_0, \ldots, a_n/C)$. 
We may without loss suppose the latter.
Since $t(a, a_0, \ldots, a_{n-1}/\A B)$ does not split over $C$, we have that 
$$t(c/C, a, a_0, \ldots, a_{n-1})=t(d'/C, a, a_0, \ldots, a_{n-1}).$$
Thus, there is some $f \in \textrm{Aut}(\M/C, a, a_0, \ldots, a_{n-1})$ such that $f(c)=d'$.
Denote $a_n'=f(a_n)$.
Then, $t(a_n', a, a_0, \ldots, a_{n-1}/A_n)=t(a_n,a, a_0, \ldots, a_{n-1}/A_n)$, and thus $t(a_n'/\A, a, a_0, \ldots, a_{n-1})=t(a_n/\A,a, a_0, \ldots, a_{n-1})$.
In particular, $$t(a_n,d' / C,a, a_0, \ldots, a_{n-1})=t(a_n', d'/ C,a, a_0, \ldots, a_{n-1}),$$
as $d' \in \A$.
But 
$$t(a_n, d'/C,a, a_0, \ldots, a_{n-1}) \neq t(a_n,c/C,a, a_0, \ldots, a_{n-1})=t(a_n', d'/C,a, a_0, \ldots, a_{n-1}),$$  
a contradiction.
\end{proof} 
 
To be able to state AIII, we need the notion of a prime model.
It is defined in terms of weakly elementary maps. 
 
\begin{definition}
Let $\alpha$ be a cardinal and $\A_i \preccurlyeq \M$ for $i<\alpha$, and let $A=\bigcup_{i<\alpha} \A_i$.
We say that $f:A\rightarrow\M$ is \emph{weakly elementary} with respect to the sequence $(\A_i)_{i<\alpha}$
if for all $a\in A$, $t(a/\emptyset )=t(f(a)/\emptyset )$ and for all $i<\alpha$, $f(\A_i) \preccurlyeq \M$.
\end{definition}

\begin{definition}
We say a model $\A$ is \emph{$s$-prime} over $A=\bigcup_{i<\alpha} \A_i$, where $\alpha$ is a cardinal and $\A_i$ is a model for each $i$,
if for every model $\B$ and
every map $f:A\rightarrow\B$ that is weakly elementary with respect to $(\A_i)_{i<\alpha}$,  
there is an elementary embedding $g:\A\rightarrow\B$ such that
$f\subseteq g$.
\end{definition}

\textbf{AIII: Let $\A, \B, \C$ be models.
If $\A\da^{ns}_{\B}\C$ and $\B=\A\cap\C$,
then there is a unique (not only up to isomorphism)
$s$-prime model $\D$ over $\A\cup\C$. Furthermore,
if $\C'$ is such that $\C \subseteq \C'$ and $\A\da_{\B}\C'$, then
$\D\da_{\C}\C'$.}

It follows that
if also $\A',\B', \C'$ and $\D'$ are as in AIII,
$f:\A\rightarrow\A'$ and $g:\C\rightarrow\C'$
are isomorphisms and $f\raj\B =g\raj\B$, then there is an
isomorphism $h:\D\rightarrow\D'$ such that
$f\cup g\subseteq h$.

The axiom AIV is given in terms of player II having a winning strategy in the game $GI(a, A, \A)$, defined below.
The motivation behind this is to be able to show that there are no infinite splitting chains of models (Lemma \ref{eiaaretketjuja}) so that a well behaving notion of rank can be developed. 
In the beginning of the game, $t(a/A)$ is considered.
The idea is that player I has to show that $t(a/A)$ does not imply $t(a/\A)$, and player II tries to isolate $t(a/\A)$ 
by enlarging the set $A$.
If the game is played in a setting with a well behaving rank, player II will be able to win in finitely many moves since there are no infinite descending chains of ordinals.

\begin{definition}
Let $\A$ be a model, $A\subseteq\A$ finite and
$a\in\M$. The game $GI(a,A,\A )$ is played as follows:
The game starts at the position $a_{0}=a$ and $A_{0}=A$.
At each move $n$, player I first chooses $a_{n+1}\in\M$ and
a finite subset $A'_{n+1}\subseteq\A$ such that
$t(a_{n+1}/A_{n})=t(a_{n}/A_{n})$, $A_{n}\subseteq A'_{n+1}$
and $t(a_{n+1}/A'_{n+1})\ne t(a_{n}/A'_{n+1})$.
Then player II chooses a finite subset $A_{n+1}\subseteq\A$
such that $A'_{n+1}\subseteq A_{n+1}$.
Player II wins if player I can no longer make a move.
 \end{definition}

Axiom AIV now states that for every tuple, there is some finite number $n$ so that player II has a winning strategy in $n$ moves.

\textbf{AIV: For each $a \in \M$, there is a number $n<\o$ such that for any countable model $\A$ and any finite subset $A \subset \A$, player II has a winning strategy in $GI(a,A,\A )$ in $n$ moves.}

As an example, consider a model class of a first order $\o$-stable theory.
There, types have only finitely many free extensions, so player II can always enlarge the set $A'_n$ to some $A_n$ such that $t(a_n/A_n)$ has a unique free extension.
After this, player I has no choice but to play some $a_{n+1}$ and $A_{n+1}'$ so that $MR(a_{n+1}/A_{n+1})<MR(a_n/A_n)$ (here, $MR$ stands for Morley rank).
Hence, AIV is satisfied (with $n=MR(a/A)$).
 
We now apply AIV to prove that any tuple is free over a model from the model itself and that the number of weak types over a model equals the cardinality of the model.
 
\begin{lemma}\label{ansA}
Let $a \in \M$ be arbitrary, and let $\A$ be a model.
Then, $a \da^{ns}_\A \A$.
\end{lemma}

\begin{proof}
It suffices to show that there is a finite $A\subseteq\A$ such that $t(a/\A )$ does not split over $A$.
Suppose not.
Assume first that $\A$ is countable.
We claim that then player I can survive $\o$ moves in $GI(a,A,\A )$ for any finite subset $A \subset \A$, which contradicts AIV.
Suppose we are at move $n$ and that $t(a_n/ \A)$ splits over every finite subset of $\A$ containing $A_n$.
In particular, it splits over $A_n$.
Let $b,c$ be tuples witnessing this splitting.
Let $f \in \textrm{Aut}(\M/A_n)$ be such that $f(b)=c$ and $f(\A)=\A$.
Now player I chooses $a_{n+1}=f(a_n)$ and $A_{n+1}=A_n \cup \{c\}$.
Then,  $t(a_n/A_n)=t(a_{n+1}/A_n)$ but $t(a_{n+1} c/ A_n)=t(a_n b/A_n) \neq t(a_n c/A_n)$ and thus $t(a_{n+1}/A_{n+1}) \neq t(a_{n}/A_{n+1})$.
As $t(a_n/\A)$ splits over every finite subset of $\A$ containing $A_n$, the same is true for $t(a_{n+1}/ \A)$.
 
Let now $\A$ be arbitrary and suppose that $t(a/\A)$ splits over every finite $A \subset \A$.
Let $\B$ be a countable submodel of $\A$.
Then, $\B$ contains only countably many finite subsets.
For each finite $B \subset \B$, we find some tuples $b, c \in \A$ witnessing the splitting of $t(a/\A)$ over $B$.
We now enlarge $\B$ into a countable submodel of $\A$ containing all these tuples.
After repeating the process $\omega$ many times we have obtained a countable counterexample.
\end{proof}

\begin{lemma}\label{notypes}
For all models $\A$, the number of weak types
$t(a/\A )$ for $a\in\M$, is $\vert\A\vert$.
\end{lemma}

\begin{proof}
We prove this first for countable models.
Suppose, for the sake of contradiction, that there is a countable model $\A$ and elements $a_i \in \M$, $i<\o_1$ so that $t(a_i/\A) \neq t(a_j/\A)$ if $i \neq j$.
As countable models are $s$-saturated, there are only countably many types over a finite set.
In particular, by the pigeonhole principle, we find an uncountable set $J \subseteq \o_1$ so that $t(a_i/\emptyset)$ is constant for $i \in J$.
After relabeling, we may set $J=\o_1$.
For each $i$, there is a number $n<\o$ such that player II wins  $GI(a_i, \emptyset, \A)$ in $n$ moves.
Using again the pigeonhole principle, we may assume that the number $n$ is constant for all $i<\o_1$.

Now we start playing $GI(a_i, \emptyset, \A)$ simultaneously for all $i< \o_1$.
Since the $a_i$ have different weak types over $\A$, for each $i$ of the form $i=2 \alpha$ for some $\alpha< \o_1$, we can find a finite set $A_\alpha \subset \A$ such that $t(a_{2 \alpha} /A_\alpha) \neq t(a_{2\alpha+1}/A_\alpha)$.
We write $A_0^i=A_\alpha$ for $i=2 \alpha$ and $i=2 \alpha+1$.
As there are only countably many finite subsets of $\A$, we find an uncountable $I \subseteq \o_1$ so that for all $i \in I$, $A_0^i=A$ for some fixed, finite $A \subset \A$.
In $GI(a_i, \emptyset, \A)$ for $i \in I$, on his first move player I plays $a_{2\alpha+1}$ and $A$ if $i=2 \alpha$ for some $\alpha<\o_1$, and $a_{2\alpha}$ and $A$ if $i=2 \alpha+1$ for some $\alpha <\o_1$.
All the rest of the games he gives up. 
Now, in each game $GI(a_i, \emptyset, \A)$ player II plays some finite $A_1^i \subset \A$ such that $A \subseteq A_1^i$.
Again, there is an uncountable $I_1' \subseteq I$ such that for $i \in I_1'$, we have $A_1^i=A_1$ for some fixed, finite $A_1$.
As there are only countably many types over $A_1$, we find an uncountable $I_1 \subset I_1'$ so that $t(a_i/A_1)=t(a_j/A_1)$ for all $i,j \in I_1$.
Again, player I gives up on all the games except for those indexed by elements of $I_1$.
Continuing like this, he can survive more than $n$ moves in uncountably many games.
This contradicts AIV.

Suppose now $\A$ is arbitrary.
Denote $X=\mathcal{P}_{<\o}(\A)$.
Then, $\vert X \vert= \vert \A \vert$.
For each $A \in X$, choose a countable model $\A_A \preccurlyeq \A$ such that $A \subset \A_A$.
By Lemma \ref{ansA}, for each weak type $p=t(a/\A)$, there is some $A_p \in X$ so that $a \da_{A_p}^{ns} \A$, and hence also $a \da_{\A_{A_p}}^{ns} \A$.
By Lemma \ref{vaplaajykskas}, $t(a/\A_{A_p})$ determines $t(a/\A)$ uniquely.
As there are only countably many types over countable models, the number of weak types over $\A$ is
$$\vert X \vert \cdot \omega=\vert \A \vert.$$
\end{proof}

\begin{lemma}\label{typegalois}
For any $a \in \M$ and any model $\A$, the weak type $t(a/\A )$ determines the Galois type $t^{g}(a/\A )$.
\end{lemma}

\begin{proof}
Suppose $t(a/\A)=t(b/\A)$.
By Lemma \ref{ansA}, we can find a countable submodel $\B$ of $\A$ so that $a \da^{ns}_\B \A$ and $b \da^{ns}_\B \A$.
By Lemma \ref{primary}, there is some $f \in \textrm{Aut}(\M/ \B)$ such that $f(\B[a])=\B[b]$ and $f(a)=b$.
Moreover, by Lemma \ref{domination}, $\B[a] \da^{ns}_\B \A$ and $\B[b] \da^{ns}_\B \A$.
Now the map $g=(f \raj \B[a]) \cup \textrm{id}_\A$ is weakly elementary.
For this, it suffices to show that $t(c/\A)=t(f(c)/\A)$ for every $c \in \B[a]$.
But $t(c/\B)=t(f(c)/\B)$, $c \da_{\B}^{ns} \A$, and $f(c) \da_{\B}^{ns} \A.$
Thus, by Lemma \ref{vaplaajykskas}, $t(c/\A)=t(f(c)/\A)$.
 
By AIII, there are unique $s$-prime models $\D_a$ and $\D_b$, over $\B[a] \cup \A$ and $\B[b] \cup \A$, respectively.
The map $g$ extends to an automorphism $h \in \textrm{Aut}(\M/\A)$ so that $h(\D_a) \subseteq \D_b$. 
The $s$-prime models are unique and preserved by automorphisms, thus we must have $h(\D_a)=\D_b$.
Since $h(a)=b$, we have $t^g(a/\A)=t^g(b/\A)$.
\end{proof}

Note that it follows from lemmas \ref{ansA} and \ref{typegalois} that the class $\mathcal{K}$ is $\o$-stable (in the sense of AECs).

Axiom AV is a weak form of symmetry (over models).
We will use it to show that symmetry actually holds over models.
  
\textbf{AV: If $\A$ and $\B$ are countable models,
$\A\subseteq\B$ and $a\in\M$,
and $\B\da^{ns}_{\A}a$, then $a\da^{ns}_{\A}\B$.}
 
\begin{lemma}\label{symlemma}
Let $A,C\subseteq\M$ and let $\B\subseteq A\cap C$ be a model.
If $A\da_{\B}^{ns}C$, then $C\da^{ns}_{\B}A$.
\end{lemma}

\begin{proof}  
We note first that for any finite tuples $a, c \in \M$,  and for any countable model $\B$ it holds that if $a \da_\B^{ns} c$, then $c \da_\B^{ns} a$. 
Indeed, then by dominance in $s$-primary models, it holds that $\B[a] \da_\B^{ns} c$, and thus by AV, $c \da_\B^{ns} \B[a]$, and in particular,  $c \da_\B^{ns} a$.

Let now $\B$ be arbitrary, and suppose $a \da_\B^{ns} c$ but $c \not\da_\B^{ns} a$.
Then, there is some finite $B \subset \B$ so that $t(a/\B c)$ does not split over $B$.
However, $t(c/ \B a)$ splits over $B$.
Let $b, d \in \B a$ be tuples witnessing this.
If $\B' \preccurlyeq \B$ is a countable model containing $B$, $b \cap \B$ and $d \cap \B$, then $a \da_{\B'}^{ns} c$ but 
$c \not\da_{\B'}^{ns} a$, which contradicts what we have just proved.

Suppose now $A\da_{\B}^{ns}C$ but $C\not\da^{ns}_{\B}A$.
Then, there is some $c \in C$ so that $c\not\da^{ns}_{\B}A$, and this is witnessed by some finite $a \in A$, i.e. $c\not\da^{ns}_{\B}a$.
But we have $a\da_{\B}^{ns}C$ and hence $a\da_{\B}^{ns}c$, a contradiction.
 \end{proof}

\begin{remark}\label{symremark}
Note that from Lemma \ref{symlemma} it  follows that for any $a, b \in \M$ and any model $\A$, it holds that $a \da_\A^{ns} b$ if and only if $b \da_\A^{ns} a$. 
\end{remark}

Axiom AVI states the existence of free extensions.

\textbf{AVI: For all models $\A, \B$ and $\D$ such that $ \A \subseteq \B \cap \D$,
there is a model $\C$
such that $t(\C/\A )=t(\B/\A )$ and $\C\da^{ns}_{\A}\D$.}

It follows that AVI holds also without the assumption
that $\B$ and $\D$ are models, as we can always find models extending these sets.
 
We now show that a form of transitivity holds for non-splitting independence.
 
\begin{lemma}\label{transitivity}
If $\B$ is a model, $A \subseteq \B$ and
$\B \subseteq C$,
then $a\da^{ns}_{A}C$ if and only if $a\da^{ns}_{A}\B$ and
$a\da^{ns}_{\B}C$.
\end{lemma}
 
\begin{proof}
If $a\da^{ns}_{A}C$, then $a\da^{ns}_{A}\B$ and
$a\da^{ns}_{\B}C$ follow by monotonicity.

Suppose now $a\da^{ns}_{A}\B$ and
$a\da^{ns}_{\B}C$.
Let $A_0 \subset A$ and $B_0 \subset \B$ be finite sets so that $A_0 \subseteq B_0$, $t(a/\B)$ does not split over $A_0$ and $t(a/C)$ does not split over $B_0$.
Suppose $a\not\da^{ns}_{A}C$.
Then, $t(a/C)$ splits over $A_0$.
Let $b, c \in C$ witness the splitting, i.e. $t(b/A_0)=t(c/A_0)$ but $t(ab/A_0) \neq t(ac/A_0)$.
By AI, there are $b', c' \in \B$ so that $t(b'/B_0)=t(b/B_0)$ and $t(c'/B_0)=t(c/B_0)$.
Since $t(a/C)$ does not split over $B_0$, we have $t(ab'/B_0)=t(ab/B_0)$ and $t(ac'/B_0)=t(ac/B_0)$.
Thus,
$$t(ab'/A_0)=t(ab/A_0)\neq t(ac/A_0)=t(ac'/A_0),$$
a contradiction since $t(a/\B)$ does not split over $A_0$.
\end{proof} 

Next, we prove a stronger version of Lemma \ref{vaplaajol}.

\begin{lemma}\label{vaplaajol2}
Suppose $\A$ is a model, $t(a/\A)$ does not split over some finite $A \subset \A$ and $B$ is such that $\A \subseteq B$.
Then, there is some $b$ such that $t(b/\A)=t(a/\A)$ and $b \da_A^{ns} B$.
\end{lemma}

\begin{proof}
Let $\B$ be a model such that $B \subseteq \B$.
Let $\C$ be a model containing $\A a$.
By AVI, there is a model $\C'$ such that $t(\C/\A)=t(\C'/\A)$ and $\C' \da^{ns}_\A \B$.
In particular, there is some $b \in \C'$ such that $t(b/\A)=t(a/\A)$ and $b \da^{ns}_ \A \B$.
Let $A' \subseteq \A$ be a finite set such that $A \subseteq A'$ and $b \da^{ns}_{A'} B$.
Then, by Lemma \ref{transitivity}, $b \da_A^{ns} B$.
\end{proof}

We now apply AIV to show that there are no infinite descending chains of models.
This will guarantee that $U$-ranks (see Definition \ref{Urank}) will be finite.
 
\begin{lemma}\label{eiaaretketjuja}
For all $a\in\M$, there is a number $n<\o$ such that
there are no models $\A_{0}\subseteq\A_{1}\subseteq ...\subseteq \A_{n}$
so that for all $i<n$, $a\nda^{ns}_{\A_{i}}\A_{i+1}$.
\end{lemma}

\begin{proof} 
Suppose models $\A_i$, $i \le n$, as in the statement of the lemma, exist.
We note first that if $a \nda^{ns}_{\A_{i}} \A_{i+1}$, then $a \nda^{ns}_{\A_{i}'} \A_{i+1}$ for every countable submodel $\A_{i'} \subset \A_i$.
Since a countable set only has countably many finite subsets, all the tuples witnessing the splittings are contained in some countable submodel $\A_{i+1}' \subset \A_{i+1}$.
Then, $a \nda_{\A_i'} \A_{i+1}'$.
Thus, we may assume each $\A_i$ is countable.
 We will show that player I can survive $n$ moves in $GI(a,\emptyset,\A_0 )$.
Then, the lemma will follow from AIV.

On the first move, player I chooses some finite $B_1 \subset \A_0$ so that $t(a/\A_0)$ does not split over $B_1$.
Then, there is some finite set $C_1 \subset \A_1$ so that $B_1 \subseteq C_1$ and $t(a/C_1)$ splits over $B_1$ and some $f_1 \in \textrm{Aut}(\M/B_1)$ such that $f(\A_1)=\A_0$.
Now player I plays $a_1=f(a)$ and $A_1'=f_1(C_1)$.
As $t(a/f_1(C_1))$ does not split over $B_1$ and $t(f_1(a)/f_1(C_1))$ splits over $B_1$, we have $t(a/f_1(C_1))\neq t(f_1(a)/f_1(C_1))$, and this is indeed a legitimate move.

On her move, player II chooses some finite $A_1 \subset \A_0$ such that $A_1' \subseteq A_1$.
On his second move, player I chooses some finite $B_2 \subset \A_0=f_1(\A_1)$ so that $A_1 \subset B_2$ and $t(a_1/\A_0)$ does not split over $B_2$.
Now there is some finite set $C_2 \subset f_1(\A_2)$ so that $t(a_1/C_2)$ splits over $B_2$ and some automorphism $f_2 \in \textrm{Aut}(\M/B_2)$ so that $f_2(f_1(\A_2))=\A_0$.
Player I plays $a_2=f_2(a_1)$ and $A_2'=f_2(C_2)$.
Continuing in this manner, he can survive $n$ many moves. 
\end{proof}
 
We now define $U$-ranks over models and finite sets.

\begin{definition}\label{Urank}
For $a$ and a model $\A$, we define the \emph{$U$-rank} of $a$ over $\A$, denoted $U(a/\A )$, as follows:
\begin{itemize}
\item  $U(a/\A )\ge 0$ always;
\item $U(a/\A )\ge n+1$ if there is some model $\B$ so that $\A \subseteq \B$, 
$a\nda^{ns}_{\A}\B$ and $U(a/\B )\ge n$;
\item $U(a/\A )$ is the largest $n$ such that $U(a/\A )\ge n$.
\end{itemize}

For finite $A$ we write $U(a/A)$ for 
$\textrm{max}(\{U(a/\A) \, | \, \A \textrm{ is a model s.t. } A \subset \A \})$.

Note that by Lemma \ref{eiaaretketjuja}, $U(a/A)$ is finite for finite $A$.
 \end{definition}

Later, we will define $U$-ranks over arbitrary sets, and it will turn out that they are always finite.
Thus, we call a class that satisfies the axioms AI-AVI a \emph{FUR-class}, for ``Finite $U$-Rank".
Eventually, we will show that FUR-classes have a perfect theory of independence (Theorem \ref{main}).
  
\begin{definition}
We say that an abstract elementary class $\K$ is a FUR-class if $\K$ has AP and JEP, $LS(\mathcal{K} )=\omega$, $\K$ has arbitrarily large structures and does not contain finite models, and $\K$ satisfies the axioms AI-AVI. 
\end{definition}

As the last result of this section, we show that non-splitting over models can be expressed in terms of preserving $U$-ranks. 

\begin{lemma}\label{Unonsplit}
Let  $\A\subseteq\B$ be models. Then
$a\da_{\A}^{ns}\B$ if and only if $U(a/\B )=U(a/\A )$.
\end{lemma}

\begin{proof}
From right to left the claim follows from the definition of $U$-rank.

For the other direction, suppose $a\da^{ns}_{\A}\B$.
It follows from the definition of $U$-rank that $U(a/\B) \le U(a/\A)$.
We will prove $U(a/\A) \le U(a/\B)$.

Let $n=U(a/\A )$, and choose models $\A_i'$, $i \le n$ so that $\A_0'=\A$ and for each $i<n$, $\A_i' \subseteq \A_{i+1}'$ and $a\nda^{ns}_{\A_{i}'}\A_{i+1}'$. 
Choose a model $\C$ so that $\A_n'  a \subseteq \C$.
By AVI, there is a model $\B'$ so that $t(\B'/\A)=t(\B/\A)$ and $\B' \da^{ns}_\A \C$.
Let $f \in \textrm{Aut}(\M/\A)$ be such that $f(\B')=\B$.
Denote $f(a)=b$ and $f(\A_i')=\A_i$ for $i \le n$.
Then, $\A_0=\A$, $t(b/\A)=t(a/\A)$ and
$b\nda^{ns}_{\A_{i}}\A_{i+1}$ for all $i<n$,
and $\B\da^{ns}_{\A} \A_n b$.

Let $\B_1$ be the unique $s$-prime model over $\B \cup \A_1$ (It exists by AIII since $\A \subseteq \B \cap \A_1$ and $\B \da_\A^{ns} \A_1$).
Suppose now that for $1 \le i <n$, $\B_{i-1} \da_{\A_{i-1}}^{ns} \A_i$, and  that we have defined $\B_i$ as the unique $s$-prime model over $\B_{i-1} \cup \A_i$ (taking $\B_0=\B$).
Then, we letÊ $\B_{i+1}$ be the unique $s$-prime model over $\B_i \cup \A_{i+1}$.
It exists, since from the "Furthermore" part in AIII it follows that $\B_i  \da_{\A_{i}}^{ns} \A_{i+1}.$
 
By Lemma \ref{vaplaajykskas}, $t(b/\B)=t(a/\B)$.
Thus, to show that $U(a/\B) \ge U(a/\A)$, it is enough that
$b\nda^{ns}_{\B_{i}}\B_{i+1}$ for all $i<n$.  
Suppose for the sake of contradiction that $b\da^{ns}_{\B_{i}}\B_{i+1}$ for some $i<n$.
Using induction and the "Furthermore" part in AIII, we get that $\B_i \da_{\A_i}^{ns} \A_n b$, and hence
by monotonicity and AV, $b \da_{\A_i}^{ns} \B_i$.
On the other hand, the counterassumption and monotonicity give $b\da^{ns}_{\B_{i}}\A_{i+1}$.
But from these two and Lemma \ref{transitivity}, it follows that
$b \da^{ns}_{\A_{i}}\A_{i+1}$,
a contradiction.
\end{proof}

\subsection{Indiscernible and Morley sequences}
 
In the next section (2.3), we will define Lascar types, an analogue to first order strong types.
We then define our main independence notion in terms of Lascar splitting.
However, there we will need technical tools to prove some properties of Lascar types.
For this purpose, we now define strongly indiscernible and Morley sequences and show that 
Morley sequences are strongly indiscernible.

\begin{definition}
We say that a sequence $(a_{i})_{i<\a}$ is \emph{indiscernible} over $A$
if every permutation of the sequence $\{ a_{i}\vert\ i<\a\}$
extends to an automorphism $f\in \textrm{Aut}(\M /A)$.

We say that a sequence $(a_{i})_{i<\a}$ is \emph{weakly indiscernible} over $A$
if every permutation of a finite subset of the sequence $\{ a_{i}\vert\ i<\a\}$
extends to an automorphism $f\in \textrm{Aut}(\M /A)$.

We say a sequence $(a_{i})_{i<\a}$ is \emph{strongly indiscernible} over $A$ if for all
cardinals $\kappa$, there are $a_{i}$, $\a\le i<\kappa$, such that
$(a_{i})_{i<\kappa}$ is indiscernible over $A$.

Let $\A$ be a model.
We say a sequence $(a_{i})_{i<\a}$ is \emph{Morley} over $\A$,  
if for all $i<\a$, $t(a_{i}/\A )=t(a_{0}/\A )$ and
$a_{i}\da_{\A}^{ns}\cup_{j<i}a_{j}$.
 \end{definition}

In the rest of this chapter, we will assume that all indiscernible sequences and Morley sequences that we consider are non-trivial,  i.e. they do not just repeat the same element.

Applying Fodor's lemma, we now show that every uncountable sequence contains a Morley sequence as a subsequence.
This will be extremely useful in many places later on. 

\begin{lemma}\label{fodor}
Let $A$ be a finite set and $\kappa$ a cardinal such that $\kappa=\textrm{cf}(\kappa)>\omega$.
For every sequence $(a_i)_{i<\kappa}$, there is a model $\A \supset A$ and some $X \subset \kappa$ cofinal so that $(a_i)_{i \in X}$ is Morley over $\A$.
\end{lemma}

\begin{proof}
For $i<\kappa$, choose models $\A_i$ so that for each $i$, $A \subset \A_i$, $a_i \in \A_{i+1}$, $\A_j \subset \A_i$ for $j<i$, $\A_\g=\bigcup_{i<\g} \A_i$ for a limit $\g$, and $\vert \A_i \vert = \vert i \vert +\o$.
Then, for each limit $i$, there is some $\a_i<i$ so that $a_i \da_{\A_{\a_i}}^{ns} \A_i$ (By Lemma \ref{ansA}, there is some finite $A_i \subset \A_i$ so that $a_i \da_{A_{i}}^{ns} \A_i$; just choose $\a_i$ so that $A_i \subset \A_{\a_i}$).
By Fodor's Lemma, there is some $X' \subset \kappa$ cofinal and some $\a<\kappa$ so that $\a_i=\a$ for all $i \in X'$.
Choose $\A=\A_\a$.
By Lemma \ref{notypes}, there are at most $\vert \A\vert<\kappa$ many weak types over $\A$, and thus by the pigeonhole principle, there is some cofinal $X \subseteq X'$ so that $t(a_i/\A)=t(a_j/\A)$ for all $i,j \in X$.
\end{proof}

\begin{lemma}\label{morleyvapaa}
If $(a_{i})_{i<\a}$ is Morley over a countable model $\A$, then
for all $i<\a$, $a_{i}\da_{\A}^{ns}\cup\{ a_{j}\vert\ j<\a,\ j\ne i\}$.
\end{lemma}
 
\begin{proof}  
The claim holds if $a_i \da_\A^{ns} S$ for every finite $S \subset \cup\{ a_{j}\vert\ j<\a,\ j\ne i\}$.
Since we can always relabel the indices, 
it  thus suffices to show that for all $n<\o$, $a_i \da_\A^{ns} \{a_j \, | \, j \neq i, j \le n\}$.
We will prove that for any $n<\o$, if $n=I \cup J$, where $I \cap J= \emptyset$, then $\bigcup_{i \in I}  a_i \da^{ns}_\A \bigcup_{i \in J} a_i$, and the lemma will follow.
We do this by induction on $n$.
If $n=1$, the claim holds trivially, and if $n=2$, it follows directly from Remark \ref{symremark}.
Suppose now the claim holds for $n$, and consider the partition of $n+1$ into the sets $I$ and $J \cup \{n\}$.
Let $a_n'$ be such that $t(a_n'/\A)=t(a_n/\A)$ and
$$a_n' \da^{ns}_\A \A[a_i \, | \,i \in J]\cup \bigcup_{i<n}a_i.$$
Then, in particular, $a_n' \da_\A \bigcup_{i<n} a_i$, so 
 $t(a_n'/ \A \cup \bigcup_{i<n} a_i)=t(a_n/ \A \cup \bigcup_{i<n} a_i)$.
Now, 
$$a_n' \da_{\A[a_i \, | \,i \in J]} \bigcup_{i \in I} a_i,$$
and by Remark \ref{symremark} and monotonicity,
$$\bigcup_{i \in I} a_i \da_{ \A[a_i \, | \,i \in J]} a_n' \cup \bigcup_{i \in J} a_i.$$
By the inductive assumption, we have $\bigcup_{i \in I} a_i \da^{ns}_\A \bigcup_{i \in J} a_i,$ and thus, by Remark \ref{symremark} and Lemma \ref{domination},
$$\bigcup_{i \in I} a_i \da_\A^{ns} \A [a_i \, | \,i \in J].$$
Hence, by Lemma \ref{transitivity},
$$\bigcup_{i \in I} a_i \da_\A^{ns} a_n' \cup\bigcup_{i \in J} a_i,$$
and since $t(a_n'/ \A \cup \bigcup_{i<n} a_i)=t(a_n/ \A \cup \bigcup_{i<n} a_i)$, we have 
$$\bigcup_{i \in I} a_i \da_\A^{ns} \bigcup_{i \in J} a_i \cup \{a_n\},$$
as wanted.
\end{proof}

\begin{lemma}
If $\A$ is a countable model, then Morley sequences over $\A$
are strongly indiscernible over $\A$.
\end{lemma} 
 
\begin{proof}
We show first that Morley sequences are weakly indiscernible.
If a sequence $(a_i)_{i<\a}$ is Morley over some model $\A$, then also every finite subsequence is Morley over $\A$.
Thus, as we may relabel any finite subsequence, it suffices 
to show that if a sequence $(a_i)_{i\le n}$, where $n \in \o$, is Morley over a model $\A$, then it is indiscernible over $\A$, i.e. that every permutation extends to an automorphism $f \in \textrm{Aut}(\M/\A)$.
We do this by induction on $n$.
The case $n=0$ is clear.
Suppose now $n=m+1$, where $m\ge0$.
We can obtain any permutation of the $a_i$, $i\le m+1$, by first permuting the $m$ first elements, then changing the place of the two last elements and permuting the $m$ first elements again.
Thus, it is enough to find some  $f \in \textrm{Aut}(\M / \A)$ so that $f(a_i)=a_i$ for $i<m$, $f(a_m)=a_{m+1}$ and $f(a_{m+1})=a_m$.

Since $t(a_m/\A)=t(a_{m+1}/\A)$, $a_m \da_\A^{ns} (a_i)_{i<m}$ and $a_{m+1} \da_\A^{ns} (a_i)_{i<m}$, we have by Lemma \ref{vaplaajykskas} that 
$t(a_m/ \A (a_i)_{i<m})=t(a_{m+1}/ \A (a_i)_{i<m})$ and thus there is some $g_1 \in \textrm{Aut}(\M/\A (a_i)_{i<m})$ such that $g_1(a_{m})=a_{m+1}$.
By Lemma \ref{morleyvapaa}, we have 
$$a_m \da_\A^{ns} (a_i)_{i < m} a_{m+1}.$$
Since $a_{m+1} \da_\A^{ns} (a_i)_{i\le m}$, we have 
$$g_1(a_{m+1}) \da_\A^{ns} (a_i)_{i< m} g_1(a_m),$$ 
so 
$$g_1(a_{m+1}) \da_\A^{ns} (a_i)_{i< m} a_{m+1}$$ since $g_1(a_m)=a_{m+1}$.
Thus, by Lemma \ref{vaplaajykskas}, there is some $g_2 \in \textrm{Aut}(\M/\A (a_i)_{i<m} a_{m+1})$ such that $g_2(g_1(a_{m+1}))=a_{m}$.
Then, $f= g_2 \circ g_1$ is the desired automorphism.
 
Next, we show that Morley sequences are indiscernible. 
Let $(a_i)_{i \in I}$ be a Morley sequence over $\A$, and let $\pi \in \textrm{Sym}(I)$ be a permutation.
We need to show that $\pi$ extends to some $F \in \textrm{Aut}(\M/\A)$.
This is done by constructing models $\A_i$ for $i<\kappa$ so that $\A_0=\A$, for each $i$, $\A_{i+1}$ is the unique $s$-prime model over $\A_i \cup \A[a_i]$, and unions are taken at limit steps.
For this we need to show that these $s$-prime models exist, i.e. that for each $i$, $\A_i \da^{ns}_\A \A[a_i]$.
 
By Lemmas \ref{symlemma} and \ref{domination}, it suffices to show that $a_i \da_{\A}^{ns} \A_i$.
For this, we will show that $a_{i_0}, a_{i_1}, \ldots, a_{i_n} \da_\A^{ns} \A_i$ for $i \le i_0< \ldots <i_n$ (the claim then clearly follows).
We prove this by induction on $i$.
The claim holds for $i=0$, since $\A_0=\A$.
Suppose now it holds for $j$. We show it holds for $j+1$.
For this, we will need two auxiliary claims.

\begin{claim}\label{domi}
The element $a_j$ dominates $\A[a_j]$ over $\A_j$.
\end{claim}

\begin{proof}
Let $c$ be such that $c \da_{\A_j}^{ns} a_j$.
By the inductive assumption, we have $a_j \da^{ns}_\A \A_j$, and thus, by symmetry and transitivity, $\A_j c \da^{ns}_\A a_j$.
By Lemma \ref{domination}, $\A_j c \da^{ns}_\A \A[a_j]$, and hence $c \da^{ns}_{\A_j} \A[a_j]$, as wanted.
\end{proof}

\begin{claim}\label{2}
Let $\B$ be a model such that $\A \subseteq \B$.
Suppose $a \da_\A^{ns} b$ and $ab \da_\A^{ns} \B$. 
Then, $a \da_\B^{ns} b$.
\end{claim}

\begin{proof}
Suppose not.
Choose some finite $A \subset \A$ such that $ab \da_A^{ns} \B$ and $a \da_A^{ns} \A b$.
Since $t(a/\B b)$ splits over $A$, there is some $c \in \B$ so that $a \not\da^{ns}_A bc$.
Choose $c' \in \A$ so that $t(c'/A)=t(c/A)$.
If we would have $t(c'/Aab) \neq t(c/Aab)$, then the pair $c,c'$ would witness that $ab \not\da_A \B$.
Hence, $t(c/Aab)=t(c'/Aab)$.
But now we have $a \not\da^{ns}_A bc'$, so $a \not\da^{ns}_A \A b$, a contradiction.
\end{proof}

Let  now $j<i_0< \ldots < i_n$.
By the inductive assumption, $a_j, a_{i_0}, \ldots, a_{i_n} \da^{ns}_\A \A_j$.
Thus, by Lemma \ref{morleyvapaa} and Claim \ref{2},
$$a_{i_0}, \ldots, a_{i_n} \da^{ns}_{\A_j} a_j,$$
and hence, by  Remark \ref{symremark} and Claim \ref{domi},
$$\A[a_j] \da^{ns}_{\A_j} \A_j a_{i_0}, \ldots, a_{i_n}.$$
By the inductive assumption we have  $ \A[a_j] \da_\A^{ns} \A_j$.
This, and Lemma \ref{transitivity} give
 $$\A[a_j] \da^{ns}_\A \A_j a_{i_0}, \ldots, a_{i_n},$$
and hence by the domination part in AIII,
$$\A_{j+1} \da^{ns}_{\A_j} \A_j a_{i_0}, \ldots, a_{i_n},$$
so
$$a_{i_0}, \ldots, a_{i_n} \da_{\A_j}^{ns} \A_{j+1}.$$ 
By applying the inductive assumption and transitivity, we get
$a_{i_0}, \ldots, a_{i_n} \da^{ns}_\A \A_{j+1},$
as wanted.

Let now $i$ be a limit ordinal.
Then, $\A_j \da^{ns}_{\A} a_i, a_{i_0}, \ldots, a_{i_n}$ for all successor ordinals $j<i$ and $i<i_0<\ldots<i_n$. 
Since $\A_i=\bigcup_{j<i} \A_i$, we have $\A_i \da^{ns}_{\A} a_i, a_{i_0}, \ldots, a_{i_n}$.

Thus, we have shown that the $s$-prime models required for the construction indeed exist.
Now, we construct models $\A_i^\pi$ so that $\A_0^\pi=\A$, for each $i$, $\A_{i+1}^\pi$ is the unique $s$-prime model over $\A_{i}^\pi \cup \A[a_{\pi(i)}]$, and at limit stages unions are taken.
We have already shown that any permutation of finitely many elements of the sequence $(a_i)_{i \in I}$ extends to an automorphism of $\M$ fixing $\A$.
Since being a Morley sequence is a local property (i.e. determined by finite subsequences of a sequence), also the sequence $(a_{\pi(i)})_{i \in I}$ is Morley.
Thus, the models $\A_i^\pi$ exist for each $i \in I$.

We claim that for each $i$, there is an isomorphism $F_i: \A_i \to \A_i^\pi$ fixing $\A$ pointwise.
Since $(a_{\pi(i)})_{i \in I}$ is a Morley sequence, we have $\A[a_{\pi(i)}] \da_\A^{ns} \A_i^{\pi}$.

Clearly we may choose $F_0=id \raj \A$.
Suppose now the claim holds for $i$.
Now, $\A_i^\pi$ is isomorphic to $\A_i$ over $\A$, and by Lemma \ref{primary}, there is some mapping $f_i \in \textrm{Aut}(\M/\A)$ such that $f_i(\A[a_i])=\A[a_{\pi(i)}]$.
 
Now $\A[a_{\pi(i)}] \da^{ns}_\A f_i(\A_i)$, and similarly as in the proof of Lemma \ref{typegalois}, on sees that the map $F_i \cup f_i: \A_i \cup \A [a_i] \to \A_i^{\pi} \cup \A[\pi(i)]$ is weakly elementary.
Thus, it extends to an elementary map $F_{i+1}: \A_{i+1} \to \A_{i+1}^\pi$.
If $i$ is a limit, then we set $F_i=\bigcup_{j<i} F_j$.

Now $F=\bigcup_{i \in I} F_i$ is as wanted.

Clearly a Morley sequence can be extended to be arbitrarily long.
Thus, Morley sequences are strongly indiscernible.
\end{proof}

\subsection{Lascar types and the main independence notion}

In this section, we will present our main independence notion and prove that it has all the usual properties of non-forking. The notion will be based on independence in the sense of Lascar splitting.
The key here is that over models, our main independence notion will agree with non splitting independence (Lemma \ref{kolme}), so we will be able to make use of the properties that we proved in section 2.1. 

We start by giving the definition for Lascar types.
These can be seen as an analogue for first order strong types, and we will eventually show that Lascar types are stationary.
We will see that Lascar types imply weak types but also that over models weak types imply Lascar types.

\begin{definition}
We say that a set $A$ is \emph{bounded} if $\vert A \vert < \delta$, where $\delta$ is the number such that $\M$ is $\delta$- model homogeneous.
\end{definition}

\begin{definition}
Let $A$ be a finite set, and let $E$ be an equivalence relation on $M^{n}$, for some $n<\o$.
We say $E$ is \emph{$A$-invariant} if for all $f \in Aut(\M /A)$ and $a,b \in \M$, it holds that if $(a,b) \in  E$, then $(f(a),f(b)) \in E$.
We denote the set of all $A$-invariant equivalence 
relations that have only boundedly many equivalence classes by $E(A)$.
 
We say that $a$ and $b$ have the same \emph{Lascar type} over a set $B$, denoted  $Lt(a/B)=Lt(b/B)$, 
if for all finite
$A\subseteq B$ and all $E\in E(A)$, it holds that $(a,b)\in E$.
\end{definition}

\begin{lemma}\label{silt}
If $(a_{i})_{i<\o}$ is strongly indiscernible over $B$,
then $Lt(a_{i}/B)=Lt(a_{0}/B)$ for all $i<\o$
\end{lemma}

\begin{proof}
For each $\kappa$, there are $a_i$, $\o \le i < \kappa$, so that $(a_i)_{i<\kappa}$ is indiscernible over $B$.
If $E \in E(A)$ for some finite $A \subset B$, then $E$ has only boundedly many classes, and thus, for a large enough $\kappa$, there must be some indices $i<j<\kappa$ so that $(a_i, a_j) \in E$.
But this implies that $(a_i, a_j) \in E$ for all $i,j<\kappa$, and the lemma follows.
\end{proof}
 
\begin{lemma}\label{frown}
Let $\A$ be a model and let $t(a/\A)=t(b/\A)$.
Then, $Lt(a/\A)=Lt(b/\A)$.
\end{lemma}

\begin{proof}
Since the equality of Lascar types is determined locally (i.e. it depends on finite sets only), we may without loss assume that $\A$ is countable.

Since $t(a/\A)=t(b/\A)$, there is a sequence $(a_i)_{i<\o}$ such that $(a)\frown(a_{i})_{i<\o}$
and $(b)\frown(a_{i})_{i<\o}$ are Morley over $\A$. 
Because Morley sequences are strongly indiscernible, $Lt(a/\A)=Lt(b/\A)$ by Lemma \ref{silt}.
\end{proof}

In particular, by Lemma \ref{notypes}, for any finite set $A$, the number of Lascar types $Lt(a/A)$ is countable.
It follows that every equivalence relation $E \in E(A)$ has only countably many equivalence classes.

\begin{lemma}\label{lascsat}
Let $\A$ be a countable model, $A$ a finite set such that $A \subset \A$ and $b \in \M$.
Then, there is some $a \in \A$ such that $Lt(a/A)=Lt(b/A)$.
\end{lemma}

\begin{proof}
Since there are only countably many Lascar types over $A$, there is some countable model $\B$ containing
$A$ and realizing all Lascar types over $A$.
By AI, we can construct an automorphism $f \in \textrm{Aut}(\M/A)$ such that $f(\B)=\A$.
Let $b'=f^{-1}(b)$.
Then, there is some $a' \in \B$ such that $Lt(a'/A)=Lt(b'/A)$.
Let $a=f(a')$.
Then, $a \in \A$ and 
$Lt(a/A)=Lt(f(b')/A)=Lt(b/A).$
\end{proof}

\begin{lemma}\label{lasindisc}
Let $A$ be a finite set and let $a, b \in \M$.
Then, $Lt(a/A)=Lt(b/A)$ if and only if there are
$n<\o$ and
strongly indiscernible sequences $I_{i}$ over $A$, $i\le n$,
such that $a\in I_{0}$, $b\in I_{n}$ and for all
$i<n$, $I_{i}\cap I_{i+1}\ne\emptyset$.
\end{lemma}

\begin{proof}
The implication from right to left follows from Lemma \ref{silt} and the fact that all the strongly indiscernible sequences intersect each other.

For the other direction, we note that "there are
$n<\o$ and
strongly indiscernible sequences $I_{i}$ over $A$, $i\le n$,
such that $a\in I_{0}$, $b\in I_{n}$ and for all
$i<n$, $I_{i}\cap I_{i+1}\ne\emptyset$" is an $A$-invariant equivalence relation.
Since we assume that $Lt(a/A)=Lt(b/A)$, it is enough to prove that this equivalence relation has only boundedly many classes.  

Suppose, for the sake of contradiction, that it has unboundedly many classes.
Then, there is a sequence $(a_i)_{i<\o_1}$ where no two elements are in the same class. 
By Lemma \ref{fodor}, there is some $X \subseteq \o_1$, $\vert X \vert =\o_1$, and a model $\A \supset A$ such that $(a_i)_{i \in X}$ is a Morley sequence over $\A$ and thus strongly indiscernible over $A$.
But now by the definition of our equivalence relation, all the elements $a_i$, $i \in X$ are in the same equivalence class, a contradiction. 
 \end{proof}
 
 Now we are ready to introduce our main independence notion.

\begin{definition}\label{freedom}
Let $A \subset B$ be finite.
We say that $t(a/B)$ \emph{Lascar splits} over $A$,
if there are $b,c\in B$ such that
$Lt(b/A)=Lt(c/A)$ but $t(ab/A)\ne t(ac/A)$.

We say $a$ is free from $C$ over $B$, denoted $a\da_{B}C$, if there is some finite $A\subset B$ such that
for all $D\supseteq B \cup C$, there is some $b$ such that
$t(b/B \cup C)=t(a/B \cup C)$ and $t(b/D)$ does not Lascar split over $A$.
\end{definition}

\begin{remark}\label{freedomremark}
Note that it follows from the above definition that if $ab \da_A B$, then $a \da_A B$.

Also, the independence notion is monotone, i.e. if $A\subseteq B \subseteq C \subseteq D$ and $a \da_A D$, then $a \da_B C$. 
\end{remark}

\begin{lemma}\label{Ltypesimplytypes}
If $Lt(a/A)=Lt(b/A)$, then $t(a/A)=t(b/A)$.
\end{lemma}

\begin{proof}
By Lemma \ref{notypes}, the equivalence relation ``$t(x/A)=t(y/A)$" has only boundedly many classes.
\end{proof}

\begin{lemma}\label{kolme}
Let $a \in \M$, let $\A$ be a model and let $B\supseteq\A$.  
The following are equivalent:
\begin{enumerate}[(i)]
\item $a\da_{\A}B$, 
\item $a\da^{ns}_{\A}B$ ,
\item $t(a/B)$ does not Lascar split over some finite $A\subseteq\A$.
\end{enumerate}
\end{lemma}

\begin{proof}
"$(i)\Rightarrow (iii)$" follows from Definition \ref{freedom} by choosing $D=B$.

For "$(ii) \Rightarrow (i)$", suppose $a\da^{ns}_{\A}B$.
Then, there is some finite $A \subset \A$ so that $t(a/B)$ does not split over $A$,
and in paricular $t(a/\A)$ does not split over $A$.
Let $D \supset B$ be arbitrary.
By Lemma \ref{vaplaajol2}, there is some $b$ such that $t(b/\A)=t(a/\A)$ and $t(b/D)$ does not split over $A$.
Since $a \da_\A^{ns} B$ and $b \da_\A^{ns} B$, we have by Lemma \ref{vaplaajykskas} that $t(b/B)=t(a/B)$.
Now, $t(b/D)$ does not Lascar split over $A$.
Indeed, if it would Lascar split, then we could find $c, d \in D$ so that $Lt(c/A)=Lt(d/A)$ 
but $t(bc/A) \ne t(bd/A)$.
By Lemma \ref{Ltypesimplytypes}, this implies that 
$t(b/D)$ would split over $A$, a contradiction.

For $(iii)\Rightarrow (ii)$,
suppose that $t(a/B)$ does not Lascar split over $A$.
We may without loss assume that
$t(a/\A )$ does not split over $A$ (just enlarge $A$ if necessary).
We claim that $t(a/B)$ does not split over $A$.
If it does, then there are $b,c \in B$  witnessing the splitting. 
Let $\B \subseteq \A$ be a countable model containing $A$.
By Lemma \ref{lascsat}, we find $(b',c') \in \B$ so that $Lt(b',c'/A)=Lt(b,c/A)$.
Since $Lt(b/A)=Lt(b'/A)$ and $Lt(c/A)=Lt(c'/A)$, we must have $t(ab/A)=t(ab'/A)$ and $t(ac/A)=t(ac'/A)$ 
(otherwise $t(a/B)$ would Lascar split over $A$).
But since $t(ab/A) \neq t(ac/A)$, we have
$$t(ab'/A)=t(ab/A) \neq t(ac/A)=t(ac'/A),$$
which means that $t(a/\A)$ splits over $A$, a contradiction.
\end{proof}

\begin{remark}\label{kolmeremark}
Note that from the proof of  ``(ii) $\Rightarrow$ (i)" for  Lemma \ref{kolme},
it follows that if $\A$ is a model such that $\A \subseteq B$ and $A \subset \A$ is a finite set so that
$a \da_A^{ns} B$, then $a \da_A B$.
In particular, for all models $\A$ and all $a \in \M$, there is some finite $A \subset \A$ such that $a \da_A \A$.
\end{remark}

\begin{lemma}\label{mallinyli}
Suppose $\A$ is a model, $A \subseteq \A$ finite, and $t(a/\A)$ does not Lascar split over $A$.
Then, $a \da_A \A$.
\end{lemma}
 
\begin{proof}
Choose a finite set $B$ such that $A \subseteq B \subset \A$ and $t(a/\A)$ does not split over $B$.
For an arbitrary $D \supseteq \A$,  
there is some $b$ so that $t(b/\A)=t(a/\A)$ and $t(b/D)$ does not split over $B$.
We will show that $t(b/D)$ does not Lascar split over $A$.
Suppose it does.
Then, we can find $c \in \A$ and $d \in D$ such that $Lt(c/A)=Lt(d/A)$ but $t(bc/A) \neq t(bd/A)$.
By Lemma \ref{lascsat}, there is some $d' \in \A$ such that $Lt(d'/B)=Lt(d/B)$.
Then, either $t(d'b/A) \neq t(cb/A)$ or $t(d'b/A) \neq t(db/A)$.
In the first case $t(b/\A)$ Lascar splits over $A$, and in the second case, $t(b/D)$ splits over $B$.
Both contradict our assumptions.
\end{proof} 

We now show that in certain special cases, preservation of $U$-ranks implies independence.
This lemma will be applied later, after we have defined $U$-ranks over arbitrary sets and show 
that independence can be expressed as preservation of $U$-ranks.
 
\begin{lemma}\label{lemmanen}
Suppose $\A$ is a model, $A\subseteq\A$ is finite and
$U(a/A)=U(a/\A )$. Then $a\da_{A}\A$. 
\end{lemma}

\begin{proof}
By Lemma \ref{mallinyli}, it is enough to show that $t(a/\A)$ does not Lascar split over $A$.
Suppose for the sake of contradiction, that 
$t(a/\A)$ does Lascar split over $A$. 
We enlarge the model $\A$ as follows.
First we go through all pairs $b,c \in \A$ so that $Lt(b/A)=Lt(c/A)$.
For each such pair, we find finitely many strongly indiscernible sequences over $A$ of length $\o_1$ as in Lemma \ref{lasindisc}.
We enlarge $\A$ to contain all these sequences.
After this, we repeat the process $\o$ many times.
Then, for every permutation of a sequence of length $\o_1$ that is strongly indiscernible over $A$ and contained in the model, we choose some automorphism $f \in \textrm{Aut}(\M/A)$ that extends the permutation.
We close the model under all the chosen automorphisms.
Next, we start looking again at pairs in the model that have same Lascar type over $A$ and adding $A$-indiscernible sequences of length $\o_1$ witnessing this.
After repeating the whole process sufficiently long, we have obtained a model $\A^*\supseteq \A$ such that for any $b,c \in \A^{*}$ with $Lt(b/A)=Lt(c/A)$, $\A^{*}$ contains $A$-indiscernible sequences witnessing this, and moreover
every permutation of a sequence of length $\o_1$ that is strongly indiscernible over $A$ and contained in $\A^{*}$ extends to an automorphism of $\A^{*}$.
 
Choose now an element $a^*$ so that $t(a^*/\A)=t(a/\A)$ and $a^* \da_\A^{ns} \A^*$.
Then, $U(a^*/\A^*)=U(a^*/A)$ by Lemma \ref{Unonsplit}.
Let $f \in \textrm{Aut}(\M/\A)$ be such that $f(a^*)=a$, and denote $\A'=f(\A^*)$.
Now, $U(a/\A')=U(a/A)$ and $t(a/\A')$ Lascar splits over $A$.

Let $b,c \in \A'$ witness the splitting.
Then, $Lt(b/A)=Lt(c/A)$ and inside $\A'$ there are for some $n<\o$, strongly indiscernible sequences $I_i$, $i \le n$, over $A$ of length $\o_1$ so that $b \in I_0$, $c \in I_n$ and $I_i \cap I_{i+1} \neq \emptyset$ for $i<n$.
Since $t(ab/A) \neq t(ac/A)$, in at least one of these sequences there must be two elements that have different weak types over $Aa$.
Since there are only countably many weak types over $Aa$, this implies that
there is inside $\A'$ a sequence $(a_{i})_{i<\o_{1}}$ strongly indiscernible over $A$
such that  
$t(aa_{0}/A)\ne t(aa_{1}/A)$ but $t(aa_{1}/A)=t(aa_{i}/A)$ for all $0<i<\o_{1}$.
Moreover, every permutation of $(a_{i})_{i<\o_{1}}$
extends to an automorphism $f\in Aut(\A' /A)$.
  
For each $i<\o_1$, let $f_i \in Aut(\M/A)$ be an automorphism permuting the sequence $(a_{i})_{i<\o_{1}}$ so that $f_i(a_0)=a_i$ and $f_i(\A')=\A'$.
Denote $b_i=f_i(a)$ for each $i<\o_1$.
Then, $U(b_i/\A')=U(b_i/A)$ and for all $j<i<\o_1$, $t(b_i/A)=t(b_j/A)$, but $t(b_i/\A') \neq t(b_j/\A')$ since  
$$t(b_i a_i/A)=t(f_i(a) f_i(a_0)/A)=t(a a_0 /A) \neq t(a f_j^{-1}(a_i) /A)=t(f_j(a) a_i/A)=t(b_j a_i /A).$$ 
  
Let $\B\subseteq\A$ be countable model such that $A\subseteq\B$.
Then for all $i<\o$, 
$$U(b_{i}/\A')=U(b_i/\B),$$
so $b_{i}\da^{ns}_{\B}\A'$ by Lemma \ref{Unonsplit}.
Thus, for all $i<j<\o_{1}$,
$t(b_{i}/\B )\ne t(b_{j}/\B )$, a contradiction by Lemma \ref{vaplaajykskas} since there are only countably many types over $\B$. 
\end{proof}

\begin{corollary}\label{typecoro}
For every $a \in \M$, every finite set $A$ and every $B \supseteq A$, there is some $b \in \M$ such that $t(a/A)=t(b/A)$ and $b \da_A B$.
\end{corollary}
  
\begin{proof}
Let $\A$ be a model such that $U(a/A)=U(a/\A).$
Let $\B$ be a model such that $\A \cup B \subseteq \B$, and let $b$ be such that $t(b/\A)=t(a/\A)$ and $b \da^{ns}_\A \B$.
Then, by Lemma \ref{Unonsplit}, 
$$U(b/\B)=U(b/\A)=U(a/\A)=U(a/A)=U(b/A).$$ 
By Lemma \ref{lemmanen}, $b \da_A \B$, and thus $b \da_A B$.
\end{proof}  
  
We now prove a weak form of transitivity.
We will apply it to prove some properties of the independence notion that we will need  before we can prove full transitivity.
  
\begin{lemma}\label{lasctrans}
Suppose $A\subseteq\A\subseteq B$. Then $a\da_{A}B$ if and only if
$a\da_{A}\A$ and $a\da_{\A}B$.
\end{lemma}

\begin{proof}
"$\Rightarrow$": $a\da_{A}\A$ is clear and $a\da_{\A}B$ follows from Lemma \ref{kolme}.
  
"$\Leftarrow$": 
Since $a \da_A \A$, there is by definition some finite $A_0 \subseteq A$ and some $b$ such that
$t(b/\A )=t(a/\A )$ and $t(b/B)$ does not Lascar split over $A_0$.
By Lemma \ref{kolme}, $b\da^{ns}_{\A}B$ and $a \da^{ns}_\A B$.
Thus, by Lemma \ref{vaplaajykskas}, $t(b/B)=t(a/B)$.
Hence $a\da_{A}B$, as wanted.
\end{proof}

Next, we prove symmetry over finite sets.
Later, this result will be applied when proving full symmetry.

\begin{lemma}\label{lascsym}
Let $A$ be finite. Then, $a\da_{A}b$ if and only if $b\da_{A}a$. 
\end{lemma}

\begin{proof}
Suppose $a \da_A b$.
Let $\A_0$ be a model such that $A \subset \A_0$.
By Corollary \ref{typecoro}, there exists some $b'$ such that $t(b'/A)=t(b/A)$
and $b' \da_A \A_0$.
Let $f \in \textrm{Aut}(\M/A)$ be such that $f(b')=b$, and denote $\A=f(\A_0)$.
Then, $A \subset \A$ and $b \da_A \A$.
Since $a \da_A b$, there is by Definition \ref{freedom} (take $D=\A b$),
some $a'$ such that
$t(a'/Ab)=t(a/Ab)$
and $t(a'/\A b)$ does not Lascar split over $A$.
It now follows from Lemma \ref{kolme}, that $a' \da_\A \A b$ and thus
$a' \da_\A b$. By Lemma \ref{kolme} and Remark \ref{symremark}, $b \da_\A a'$.
By Lemma \ref{lasctrans}, $b \da_A a'$, and thus $bÊ\da_A a$.
\end{proof}

Now, we prove a weak form of extension: that types over finite sets have free extensions.
We will apply the lemma when proving full extension. 
 
\begin{lemma}\label{lascvaplaaj}
For every $a$, every finite set $A$ and every  $B\supseteq A$, there is $b$
such that $Lt(b/A)=Lt(a/A)$ and $b\da_{A}B$.
\end{lemma}

\begin{proof}
Let $\A_0$ be a countable model such that $A \subset \A_0$.
By Corollary \ref{typecoro}, there is some element $a'$ so that $t(a'/A)=t(a/A)$ and $a' \da_A \A_0$.
Let $f \in \textrm{Aut}(\M/A)$ be such that $f(a')=a$.
Denote $\A=f(\A_0)$.
Now, $A \subset \A$ and $a \da_A \A$.
 
Choose now  $b$ so that $t(b/\A )=t(a/\A )$ and $b\da^{ns}_{\A}B$.  
Then, $b\da_{\A}B$.
By Lemma \ref{lasctrans}, $b\da_{A}B$.
Moreover, by Lemma \ref{frown}, $Lt(b/A)=Lt(a/A)$.
\end{proof}
 
 Now, we are ready to prove stationarity and transitivity.

\begin{lemma}[Stationarity]\label{lascvaplaajykskas}
If $A \subseteq B$, $a\da_{A}B$, $b\da_{A}B$ and $Lt(a/A)=Lt(b/A)$, then
$Lt(a/B)=Lt(b/B)$.
\end{lemma}

\begin{proof}
Clearly it is enought to prove this under the assumption
that $A$ and $B$ are finite (if $Lt(a/B_0)=Lt(b/B_0)$ for every finite $B_0 \subset B$, then $Lt(a/B)=Lt(b/B)$).
Suppose the claim does not hold.
We will construct countable models $\A_a$ and $\A_b$ so that $Aa \subset \A_a$, 
$Ab \subset \A_b$,  $B\da_{A}\A_{a}$ and $B\da_{A}\A_{b}$.  
Let $\A$ be a model such that $Aa \subset \A$.
By Lemma \ref{lascsym}, we have $B \da_A a$.
Thus, by Definition \ref{freedom}, there is some $B'$ such that $t(B'/Aa)=t(B/Aa)$ and $t(B'/\A)$ does not Lascar split over $A$.
Let $f \in \textrm{Aut}(M/Aa)$ be such that $f(B')=B$, and denote $\A_a=f(\A)$.
Then, $Aa \subset \A_a$ and $t(B/\A_a)$ does not Lascar split over $A$.
By Lemma \ref{mallinyli}, $B \da_A \A_a$.
Similarly, we find a suitable model $\A_b$.
Now by Lemma \ref{lascvaplaaj}, there is some $c$ such that $Lt(c/A)=Lt(a/A)$ and $c \da_A \A_a \cup \A_b \cup B$.
By monotonicity, we have $c \da_{\A_a} B$, and thus by Lemma \ref{lascsym}, $B \da_{\A_a} c$.
Hence, by Lemma \ref{lasctrans}, $B\da_{A}\A_{a}c$ and so
$ac\da_{A}B$.

By the counterassumption, we may without loss assume that $Lt(c/B) \neq Lt(a/B)$.
Choose a model $\B\supseteq B$  
so that $ac\da_{A}\B$. 
By Lemma \ref{frown}, $t(a/\B )\ne t(c/\B )$. 
So there is some $b'\in\B$ that
withesses this, i.e. $t(ab'/A)\ne t(cb'/A)$. 
As $Lt(c/A)=Lt(a/A)$, this means $t(b'/Aac)$ Lascar splits over $A$, a contradiction since $b' \da_A ac$.
\end{proof}
 
\begin{lemma}[Transitivity]\label{lasctrans2}
Suppose $A\subseteq B\subseteq C$, $a\da_{A}B$ and $a\da_{B}C$.
Then $a\da_{A}C$. 
\end{lemma}

\begin{proof}
Clearly it is enough to prove this for finite
$A$. 
Choose $b$ so that $Lt(b/A)=Lt(a/A)$ and
$b\da_{A}C$. 
Then, by monotonicity, $b \da_A B$, and thus by
Lemma \ref{lascvaplaajykskas}, 
$Lt(b/B)=Lt(a/B)$.
Again by monotonicity,
$b\da_{B}C$, and by Lemma \ref{lascvaplaajykskas}, 
$Lt(b/C)=Lt(a/C)$.
The claim follows.
\end{proof}

We don't yet have all the results needed for proving finite character, but we prove 
the following special case that we will need when proving other properties. 

\begin{lemma}\label{local}
Suppose $A \subset B$, $A$ is finite, and $a \not\da_{A} B$.
Then there is some $b \in B$ such that $a \not\da_{A} b$.
\end{lemma}

\begin{proof}
By Lemma \ref{lascvaplaaj}, there is some $c$ such that $Lt(c/A)=Lt(a/A)$ and $c \da_A B$.
We have $a \not\da_A B$, and thus $t(c/B) \neq t(a/B)$.
Hence, there is some $b \in B$ so that $t(cb/A) \neq t(ab/A)$.
Since $c \da_A B$, we have $c \da_A b$.
Now, $a \not\da_A b$.
Indeed, otherwise Lemma \ref{lascvaplaajykskas} would imply $Lt(c/Ab)=Lt(a/Ab)$ and thus $t(c/Ab)=t(a/Ab)$, a contradiction against the fact that $t(cb/A) \neq t(ab/A)$.
\end{proof}

We will now start working towards a more comprehensive definition of $U$-rank that will allow characterizing independence in terms of $U$-ranks.
For this, we need the notion of strong automorphism.

\begin{definition}
Let $A$ be a finite set and let $f \in \textrm{Aut}(\M/A)$.
We say that $f$ is a \emph{strong automorphism} over $A$ if it preserves Lascar types over $A$, i.e. if 
for any $a$, $Lt(a/A)=Lt(f(a)/A)$.
We denote the set of strong automorphisms over $A$ by $\textrm{Saut}(\M/A)$.
\end{definition}

\begin{lemma}
Suppose $A$ is finite and $Lt(a/A)=Lt(b/A)$. Then there is
$f\in Saut(\M /A)$ such that $f(a)=b$.
\end{lemma}

\begin{proof}
Choose a countable model
$\A$ such that $A \subseteq \A$ and
$ab\da_{A}\A$.
In particular, by Remark \ref{freedomremark}, 
$a\da_{A}\A$ and $b\da_{A}\A$.
By Lemma \ref{lascvaplaajykskas}, $Lt(a/\A)=Lt(b/\A)$. 
Thus, there is some $f\in \textrm{Aut}(\M /\A )$ such that $f(a)=b$.
By Lemma \ref{frown}, $f \in  \textrm{Saut}(\M /\A )$. 
\end{proof}

In order to give a general definition of $U$-rank, we will show
that if $\A$ is a model and $a \in \M$, then the rank $U(a/\A)$ equals the minimum
of ranks $U(a/A)$, where $A$ ranges over finite subsets of $\A$.
This is obtained as a corollary of the following lemma. 
 
\begin{lemma}\label{nonsplitU}
Suppose $\A$ is a model, $A\subseteq\A$ is finite and $t(a/\A )$
does not split over $A$. Then $U(a/A)=U(a/\A )$.
\end{lemma}

\begin{proof}
Suppose not. 
If we choose some countable model $\A'$ such that $A \subset \A' \subseteq \A$,
then $a \da_{\A'}^{ns} \A$, and thus, by Lemma \ref{Unonsplit}, $U(a/\A')=U(a/\A)$.
Hence, we may assume that $\A$ is countable.
 
Choose a countable model $\B$ such that $A \subset \B$ and $U(a/\B )=U(a/A)$. 
Now, there is some $f \in \textrm{Aut}(\M/A)$ so that $f(\B)=\A$.
Let $a'=f(a)$.
We have
$$U(a/\A) \neq U(a/A)=U(a/\B)=U(a'/\A),$$
and thus $t(a/\A) \neq t(a'/\A)$.
Hence there is some $c \in \A$ such that $t(ac/A) \neq t(a'c/A)$.
Let $b \in \B$ be such that $f(b)=c$ (and thus $t(b/A)=t(c/A)$).
Then, $t(a'c/A)=t(ab/A)$, so $t(ac/A) \neq t(ab/A)$.
Let $c' \in \A$ be such that $Lt(c'/A)=Lt(b/A)$, and thus $t(c'/A)=t(b/A)=t(c/A).$
Since $t(a/\A)$ does not split over $A$, we have $t(ac'/A) \neq t(ab/A).$
 
We note that since $a \da_{A}^{ns} \A$, we have by Remark \ref{kolmeremark}  $a \da_{A} \A$, and thus in particular $a \da_A c'$.
 Choose $g\in Saut(\M /A)$ so that
$g(b)=c'$. Let $a''=g(a)$. 
By Lemma \ref{lemmanen}, we have $a \da_A b$, and thus $a'' \da_A c'$.
But now $Lt(a''/A)=Lt(a/A)$,
$a\da_{A}c'$ and $a''\da_{A}c'$, yet 
$$t(ac'/A) \ne t(a'c'/A)=t(ab/A)=t(a''c'/A),$$
so in particular $t(a/Ac') \ne t(a''/Ac')$, 
a contradiction to Lemma \ref{lascvaplaajykskas}.
\end{proof}

\begin{corollary}\label{Ukoro}
Let $\A$ be a model.
Then, 
$$U(a/\A)=\textrm{min}(\{U(a/B) \, | \, B \subset \A \textrm{ finite } \}.$$
\end{corollary}

\begin{proof} 
By Definition \ref{Urank}, for each finite $B \subset \A$, it holds that $U(a/B) \ge U(a/\A)$. 
On the other hand, by Lemma \ref{ansA}, there is some finite $A \subset \A$ so that $t(a/\A)$ does not split over $A$.
By Lemma \ref{nonsplitU}, $U(a/\A)=U(a/A)$.
 \end{proof}

Corollary \ref{Ukoro} allows us to define $U(a/A)$ for arbitrary $A$ as follows.
By Definition \ref{Urank}, for finite $A$ it holds that $U(a/A_0) \ge U(a/A)$ for all $A_0 \subseteq A$.
Thus, the following definition corresponds to Definition \ref{Urank} also in the case that $A$ is finite.
  
\begin{definition}\label{U2}
Let $A$ be arbitrary.
We define  $U(a/A)$ to be the minimum of
$U(a/B)$, $B\subseteq A$ finite.
\end{definition}

Now we can finally characterize independence in terms of $U$-ranks.
As corollaries, we will get local character and extension for the independence notion.

\begin{lemma}\label{daiffU}
For all $A\subseteq B$ and $a$, $a\da_{A}B$ if and only if $U(a/A)=U(a/B)$.
\end{lemma}

\begin{proof} 
Suppose first $B$ is finite.

"$\Leftarrow$": Choose a model $\A\supseteq B$ such that
$U(a/\A )=U(a/B)$. 
Then $U(a/\A )=U(a/A )$ and thus by Lemma \ref{lemmanen},
$a\da_{A}\A$, and in particular $a\da_{A}B$.

 "$\Rightarrow$":
Choose a model $\A\supseteq A$ such that
$U(a/\A )=U(a/A)$ and a model  $\B\supseteq\A B$.
By Lemma \ref{vaplaajol}, there is some $a'$ 
such that $t(a'/\A )=t(a/\A )$ and $a'\da^{ns}_{\A}\B$. 
Then, by Lemma \ref{Unonsplit},
$$U(a'/A)=U(a'/\A )=U(a'/\B ),$$ 
so by Lemma \ref{lemmanen}, $a'\da_{A}\B$.
By Lemma \ref{frown}, $Lt(a'/A)=Lt(a/A)$, and thus by Lemma \ref{lascvaplaajykskas}, 
 $t(a'/B)=t(a/B)$. 
 Thus
 $$U(a/B)=U(a'/B)=U(a'/A)=U(a/A).$$
 
We now prove the general case.
Let $A, B$ be arbitrary such that $A \subseteq B$.

"$\Rightarrow$": 
Suppose $a \da_A B$.
There are some finite sets $A_0, A_0' \subseteq A$ such that $a \da_{A_0} B$ and some $U(a/A)=U(a/A_0')$.
We may without loss assume that $A_0=A_0'$.
Indeed, this follows from monotonicity and the fact that if $A_0''$ is any set such that
$A_0' \subseteq A_0'' \subseteq A$, then $U(a/A_0'')=U(a/A)$.
Let $B_0 \subseteq B$ be a finite set such that $U(a/B_0)=U(a/B)$.
By similar argument as above, we may without loss suppose that $A_0 \subseteq B_0$.
Thus, since the result holds for finite sets, we have 
$$U(a/A)=U(a/A_0)=U(a/B_0)=U(a/B).$$

"$\Leftarrow$":
Suppose $U(a/A)=U(a/B)$, but $a \nda_A B$.
Let $A_0 \subseteq A$ and $B_0 \subseteq B$ be finite sets such that 
$$U(a/A_0)=U(a/A)=U(a/B)=U(a/B_0).$$
By monotonicity, we have $a \nda_{A_0} B$, and by Lemma \ref{local}, there is some $b \in B$ such that
$a \nda_{A_0} b$.
Then, also $a \nda_{A_0} B_0 b$.
But 
$$U(a/B_0b)=U(a/B)=U(a/A_0),$$
a contradiction.
\end{proof}

\begin{corollary}[Local character]\label{coro}
For all $A$ and $a$ there is finite $B\subseteq A$ such that
$a\da_{B}A$. 
\end{corollary}

\begin{corollary}[Extension]\label{lascvaplaaj2}
For all $a$ and all sets $A \subseteq B$, there is some $b$ such that $Lt(b/A)=Lt(a/A)$ and $b \da_A B$.
\end{corollary}

\begin{proof}
Let $A_0 \subseteq A$ be a finite set such that $U(a/A_0)=U(a/A)$.
Then, $a \da_{A_0} A$ by Lemma \ref{daiffU}.
By Lemma \ref{lascvaplaaj}, there is some $b$ such that $Lt(b/A_0)=Lt(a/A_0)$ and $b \da_{A_0} B$.
By Lemma \ref{lascvaplaajykskas}, $Lt(b/A)=Lt(a/A)$.
\end{proof}
  
Now it is easy to prove also finite character and symmetry.  
   
\begin{lemma}[Finite character]\label{local2}  
Suppose $A \subset B$, and $a \not\da_{A} B$.
Then there is some $b \in B$ such that $a \not\da_{A} b$.
\end{lemma}

\begin{proof}
Choose a finite $C \subseteq A$ such that $a \da_C A$ and an element $c$ such that
$Lt(c/C)=Lt(a/C)$ and $c \da_C A \cup B$ (they exist by Corollary \ref{coro} and  Lemma \ref{lascvaplaaj}).
Then, by Lemma \ref{lascvaplaajykskas}, $Lt(c/A)=Lt(a/A)$.
We have $a \not\da_C B$, and thus $t(c/B) \neq t(a/B)$.
Hence, there is some $b \in B$ so that $t(cb/C) \neq t(ab/C)$.
By monotonicity, we have $c \da_A B$, and in particular $c \da_A b$.
If $a \da_A b$, then $a \da_C b$ by Lemma \ref{lasctrans2}.
Since  $t(a/Cb) \neq t(c/Cb)$, this contradicts Lemma \ref{lascvaplaajykskas}.
\end{proof}
  
\begin{lemma}[Symmetry]\label{lascfullsym}
Let $A$ be arbitrary.
If $a\da_{A}b$, then $b\da_{A}a$. 
\end{lemma}

\begin{proof}
Suppose not.
Choose some finite $B\subseteq A$ such that
$a\da_{B}Ab$ and $b\da_{B}A$ (such a set can be found by Corollary \ref{coro}).  
Since $b \nda_A a$, we have $b \nda_B Aa$.
By Lemma \ref{local2}, there is some finite set $C$ such that
$B \subseteq C \subseteq A$ and $b\not\da_{B}Ca$.  
By transitivity, $b\nda_{C}a$.
On the other hand, $a \da_B Ab$, and thus $a \da_B Cb$, so $a \da_C b$, which contradicts Lemma \ref{lascsym}.
\end{proof}

We now show that ranks can be added together in the usual way.

\begin{lemma}\label{Uaddit}
For any $a,b$ and $A$, it holds that
$$U(ab/A)=U(a/bA)+U(b/A).$$
\end{lemma}

\begin{proof}
We first note that it suffices to prove the lemma in case $A$ is finite.
Indeed, by definition $\ref{U2}$, we find finite $A_1, A_2, A_3 \subset A$ so that $U(ab/A)=U(ab/A_1)$,
$U(a/bA)=U(a/bA_2)$ and $U(b/A)=U(b/A_3)$.
Denote $A_0=A_1 \cup A_2 \cup A_3$.
Since the above ranks are minimal, we have $U(ab/A)=U(ab/A_0)$,
$U(a/bA)=U(a/bA_0)$ and $U(b/A)=U(b/A_0)$.
Thus it suffices to show that the lemma holds for $A_0$, a finite set.
 
Next, we show that for any $c$ and any finite set $B$, $U(c/B)$ is the maximal number $n$ such that there are
sets $B_i$, $i\le n$ so that $B_0=B$, and for all $i<n$, $B_i \subseteq B_{i+1}$ and $c \not\da_{B_i} B_{i+1}$.
By Lemma \ref{daiffU}, $U(c/B_i)>U(c/B_{i+1})$ for all $i<n$, and thus, $U(c/B) \ge n$.
On the other hand, by the definition of $U$-rank (Definition \ref{Urank}), there are models $\B_i$, $i \le m=U(c/B)$, so that 
$B \subset \B_0$, and for each $i<m$, $\B_i \subset \B_{i+1}$ and $c \not\da_{\B_i} \B_{i+1}$.
Write $B_0=B$.
By Lemma \ref{local2}, for each $1 \le i<m$, we find some finite $B_i \subset \B_i$ so that $c \not\da_{B_{i-1}} B_i$.
Thus, $n \ge m=U(c/B)$.

To show $U(ab/A) \le U(a/bA)+U(b/A)$, we let $n=U(ab/A)$ and $A_i$, $i\le n$ be as above for $U(ab/A)$.
Then, for each $i <n$, we must have either $a \not\da_{bA_i} A_{i+1}$ or $b \not\da_{A_i} A_{i+1}$.
Indeed, if we would have both $a \da_{bA_i} A_{i+1}$ and $b \da_{A_i} A_{i+1}$, then by Lemma \ref{lascsym}, we would have $A_{i+1} \da_{A_i} b$ and $A_{i+1} \da_{bA_i} a$, and thus by applying first Lemma \ref{lasctrans2} and monotonicity, then Lemma \ref{lascsym} again, we would get $ab \da_{A_i} A_{i+1}$.
Thus, $U(a/bA)+U(b/A) \ge n$.
 
Let now $U(b/A)=m$ and let $A_i'$, $i \le m$ be the sets witnessing this (here $A_0'=A$).
Choose $a'$ so that $t(a'/Ab)=t(a/Ab)$ and $a' \da_{bA} A_m'$.
Using a suitable automorphism, we find $A_i$, $i \le m$, also witnessing $U(b/A)=m$ so that $a \da_{bA} A_m$.
Thus, by Lemma \ref{daiffU}, $U(a/bA_m)=U(a/bA)$.
Let $U(a/bA_m)=k$ and choose $B_i$, $i \le k$ witnessing this.
Now, $A=A_0, \ldots, A_{m-1}, B_0, \ldots, B_k$ witness that $U(ab/A) \ge m+k$ (note that we may without loss assume that $A_m=B_0$).
\end{proof}
 
The results we have proved thus far allow us to prove also the following lemma.
It will be used in the next section, when discussing global types and canonical bases. 
 
\begin{lemma}\label{statilemma}
Suppose $\A$ is a model, $A$ a finite set such that $A \subset \A$, $B$ is such that $\A \subseteq B$,
$b \da_A^{ns} B$, $b' \da_A^{ns} B$ and $t(b/A)=t(b'/A)$.
Then, $t(b/B)=t(b'/B)$. 
\end{lemma}

\begin{proof}
By Lemma \ref{lascvaplaajykskas}, it suffices to show that $Lt(b/A)=Lt(b'/A)$. 
By Lemma \ref{frown}, this follows after we have shown that $t(b/\A)=t(b'/\A)$.
Suppose not. 
Then, there is some $a \in \A$ such that $t(ab/A) \neq t(ab'/A)$.
Let $f \in \textrm{Aut}(\M/A)$ be such that $f(b')=b$, and let $a'=f(a)$.
Since $b' \da_A^{ns} \A$, we have $b' \da_A \A$ by Remark \ref{kolmeremark}, and in particular, $b' \da_A a$.
Thus, $b \da_A a'$.
By Lemma \ref{lascsat} there is some $a'' \in \A$ such that $Lt(a''/A)=Lt(a'/A)$.
Since $b \da_A^{ns} \A$, we have $b \da_A a''$.
Then, $Lt(a''/Ab)=Lt(a'/Ab)$ by Lemmas \ref{lascfullsym} and \ref{lascvaplaajykskas}.
Now,
$$t(ab/A) \neq t(ab'/A)=t(a'b/A)=t(a''b/A),$$
a contradiction since $t(b/\A)$ does not split over $A$.
\end{proof}

Next, we give our analogue to first order algebraic closure: bounded closure.
We then show that models are closed in terms of the bounded closure, 
that the bounded closure operator has finite character, 
and that it really is a closure in the sense that the closure of a closed set is the set itself. 
In section 2.5, we will present quasiminimal classes.
This setting is analoguous to the first order strongly minimal setting.
In strongly minimal classes, the algebraic closure operator yields a pregeometry, 
and ranks can be calculated as pregeometry dimensions.
Similarly, in quasiminimal classes, a pregeometry is obtained from the bounded closure operator,
and $U$-ranks are given as pregeometry dimensions.

\begin{definition}
We say $a$ is in the \emph{bounded closure} of $A$, denoted $a \in \textrm{bcl}(A)$, if $t(a/A)$ has only boundedly many realizations.
\end{definition}

\begin{lemma}\label{bclmalli}
Let $\A$ be a model.
Then, $\textrm{bcl}(\A)=\A$.
\end{lemma}

\begin{proof}
Clearly $\A \subseteq \textrm{bcl}(\A)$.
For the converse, suppose towards a contradiction that $a \in \textrm{bcl}(\A) \setminus \A$.
By Lemma \ref{ansA}, there is some finite $A \subset \A$ so that $a \da_A^{ns} \A$.  Choose now an element $a'$ such that $t(a'/\A)=t(a/\A)$ and $a' \da^{ns}_A \textrm{bcl}(\A)$.
Then, $a' \in \textrm{bcl}(\A)$.
By Axiom I, there is some $b \in \A$ such that $t(b/A)=t(a'/A)$ and thus $b \neq a'$.
In particular, $t(a'a'/A) \neq t(ba'/A)$.
Thus, $a'$ and $b$ witness that $t(a'/\textrm{bcl}(\A))$ splits over $A$, a contradiction.
\end{proof}

\begin{lemma}\label{bcllocal}
If $a \in \textrm{bcl}(A)$, then there is some finite $B \subseteq A$ so that $a \in \textrm{bcl}(B)$.
\end{lemma}

\begin{proof}
There is some finite $B \subseteq A$ such that $a \da_B A$.
We claim that $a \in \textrm{bcl}(B)$.
Suppose not.
Let $\A$ be a model such that $A \subseteq \A$.
Now there is some $a'$ so that $Lt(a'/A)=Lt(a/A)$ and $a' \da_B \A$.
By Lemma \ref{bclmalli}, $a' \in \textrm{bcl}(A) \subseteq \textrm{bcl}(\A)=\A$.
Since $a \notin \textrm{bcl}(B)$, the weak type $t(a/B)$ has unboundedly many realizations.
Hence, by Lemma \ref{fodor}, there is a Morley sequence $(a_i)_{i<\o}$ over some model $\B \supset B$ so that $a_0=a'$ (just use a suitable automorphism to obtain this).
By Axiom AI, there is an element $a'' \in \A$ so that $t(a''/a'B)=t(a_1/a'B)$,
and by Lemma \ref{silt} , $Lt(a_1/B)=Lt(a'/B)$.
Thus, there is an automorphism $f \in \textrm{Aut}(\M/B)$ such that $f(a'')=a_1$ and $f(a')=a'$.
Using Lemma \ref{lasindisc}, one sees that automorphisms preserve equality of Lascar types.
Hence, the fact that $Lt(a_1/B)=Lt(a'/B)$ implies 
$Lt(a''/B)=Lt(a'/B).$ 
But we have $a'=a_0 \neq a_1$, and thus also $a'' \neq a'$, so
$t(a'a'/B) \neq t(a'a''/B)$, which contradicts Lemma \ref{lascvaplaajykskas} since we assumed $a' \da_B \A$.
\end{proof} 

\begin{lemma}\label{xiv}
For every $A$, $\textrm{bcl}(\textrm{bcl}(A))=\textrm{bcl}(A)$.
\end{lemma}
 
\begin{proof}
By Lemma \ref{bcllocal}, we may assume that $A$ is finite.
Suppose now $a \in \textrm{bcl}(\textrm{bcl}(A)) \setminus \textrm{bcl}(A)$.
By Lemma \ref{bcllocal}, there is some $b \in \textrm{bcl}(A)$ so that $a \in \textrm{bcl}(Ab)$.
Let $\kappa$ be an uncountable cardinal such that $\kappa> \vert \textrm{bcl}(\textrm{bcl}(A)) \vert$.
Since $a \notin \textrm{bcl}(A)$, there are $a_i$, $i<\kappa$ so that $a_i \neq a_j$ when $i \neq j$ and $t(a_i/A)=t(a/A)$ for all $i<\kappa$.
For each $i$, there is some $b_i \in \textrm{bcl}(A)$ such that $t(b_ia_i/A)=t(ba/A)$.
By the pigeonhole principle, there is some $b' $ and some $X \subseteq \kappa$
so that $\vert X \vert = \kappa$ and $b_i=b'$ for $i \in X$.
Hence, for any $i \in X$, $t(a_i/Ab')$ has unboundedly many realizations, a contradiction since $a_i \in \textrm{bcl}(Ab')$.
\end{proof} 

\begin{lemma}\label{xi}
Let $A \subset B$.
If $a \in \textrm{bcl}(A)$, then $a \da_A B$.
\end{lemma}

\begin{proof}
By Lemma \ref{bcllocal}, we may assume that $A$ is finite.
Choose $a'$ so that $Lt(a'/A)=Lt(a/A)$ and $a' \da_A B$.
Then, $a' \in \textrm{bcl}(A)$.
Consider the equivalence relation $E$ defined so that $(x,y) \in E$ if either $x,y \notin \textrm{bcl}(A)$ or $x=y \in \textrm{bcl}(A)$.
This is an $A$-invariant equivalence relation.
Moreover, since $A$ is finite, we may choose a countable model $\A$ so that $A \subset \A$. 
By Lemma \ref{bclmalli}, $\textrm{bcl}(A) \subset \A$, so $E$ has boundedly many classes, and thus $(a,a') \in E$.
It follows that $a=a'$.
\end{proof}

\begin{lemma}[Reflexivity]\label{xii}
If $a \in \textrm{bcl}(B) \setminus \textrm{bcl}(A)$, then $a \not\da_A B$. 
\end{lemma} 

\begin{proof}
Suppose $a \da_A B$.
Choose a model $\A$ so that $B \subseteq \A$ and $a'$ so that $t(a'/B)=t(a/B)$ and $a' \da_A \A$.
By Lemma \ref{bclmalli}, $a' \in \A$.
Now we proceed as in the proof of Lemma \ref{bcllocal} to obtain a contradiction.
\end{proof}

Now we have shown that our main independence notion $\da$ has all the properties of non-forking.

\begin{theorem}\label{main} 
Let $\K$ be a FUR-class, let $\M$ be a monster model for $\K$, and
suppose $A \subseteq B \subseteq C \subseteq D \subset \M$.
Then, the following hold.
\begin{enumerate}[(i)]
\item Local character: For each $a$, there is some finite $A_0 \subseteq A$ such that $a \da_{A_0} A$.
\item Finite character: If $a \not\da_{A} B$, then there is some $b \in B$ so that $a \not\da_{A} b$.
\item Stationarity: Suppose that $Lt(a/A)=Lt(b/A)$, $a \da_A B$ and $b \da_A B$.
Then, $Lt(a/B)=Lt(b/B)$.
\item Extension: For every $a$, there is some $b$ such that $Lt(b/A)=Lt(a/A)$ and $b \da_A B$.
\item Monotonicity: If $a \da_A D$, then $a \da_B C$.
\item Transitivity: If $a \da_A B$ and $a \da_B C$, then $a \da_A C$.
\item Symmetry: If $a \da_A b$, then $b \da_A a$.
\item $U$-ranks: $a \da_A B$ if and only if $U(a/B)=U(a/A)$.
\item Finiteness of $U$-rank: For all $a$, $U(a/\emptyset)<\o$.
\item Addition of ranks: For all $a,b$, $U(ab/A)=U(a/bA)+U(b/A)$.
\item Independence of bcl: If $a \in \textrm{bcl}(A)$, then $a \da_A B$.
\item Reflexivity:  If $a \in \textrm{bcl}(B) \setminus \textrm{bcl}(A)$, then $a \not\da_A B$. 
\item Local character of bcl: If $a \in \textrm{bcl}(A)$, then there is some finite $A_0 \subseteq A$ so that $a \in \textrm{bcl}(A_0)$.
\item Closure: $\textrm{bcl}(\textrm{bcl}(A))=\textrm{bcl}(A)$.
\item Models are closed: If $\A$ is a model, then $\textrm{bcl}(\A)=\A$.
\end{enumerate}
\end{theorem}
 
By easy calculations, one proves the following corollary (for (ii), use the definition of $U$-rank, and for (iii), choose a model $\A$ such that $a \in \A$ and $c \da_a \A$, and use the fact that if $b \in \textrm{bcl}(a)$, then $b \in \textrm{bcl}(\A)$ by Theorem \ref{main}, (xiv)).

\begin{corollary}
Suppose $\K$ is a FUR-class, and $\M$ is a monster model for $\K$.
Then, the following hold.
\begin{enumerate}[(i)]
\item If $A \da_B C$ and $D \subseteq \textrm{bcl}(BC)$, then $A \da_B CD$.
\item If $a \in \textrm{bcl}(Ab)$ and $b \in \textrm{bcl}(Aa)$, then $U(a/A)=U(b/A)$.
\item If $a \in \textrm{bcl}(Ab)$ and $b \in \textrm{bcl}(Aa)$, then for any $c$, it holds that $U(c/Aa)=U(c/Ab)$.
\end{enumerate}
\end{corollary}

\subsection{$\M^{eq}$ and canonical bases}\label{meqsect}

In this section, we will construct $\M^{eq}$ and show that canonical bases exist and have the usual properties one would expect.

In the first order context, there are only countably many formulae, and thus only countably many definable equivalence relations.
Hence, $\M^{eq}$ can be built by adding imaginaries corresponding to all of them without violating $\o$-stability.
Moreover, in this context, $\M^{eq}$ has elimination of imaginaries.

However, in the non-elementary case, the concept of definability is much broader than first order definability.
As an example, consider a class of models with unary predicates $P_i$, $i<\o$, such that each $P_i$ has an infinite interpretation.
For each $X \subseteq \o$, define an equivalence relation $E_X$ so that $a E_X b$ if and only if $a \in \bigcup_{i \in X} P_i \leftrightarrow b \in \bigcup_{i \in X} P_i$.
This gives us uncountably many equivalence relations.
The same construction can be done using Galois types in place of the predicates $P_i$, so in most non-elementary cases there will be uncountably many equivalence relations to consider.

Since we wish to be able to use our independence calculus in $\M^{eq}$, $\o$-stability is vital for our arguments.
Thus, we cannot add uncountably many imaginaries.
We will solve the problem by choosing a countable collection of equivalence relations and adding imaginaries only for them.
The choice will be made so that the collection contains everything needed for our arguments (in particular, so that we can have canonical bases).
However, the drawback here is that we cannot prove elimination of imaginaries.
Instead, we will move to $(\M^{eq})^{eq}$, $((\M^{eq})^{eq})^{eq}$, etc., when needed.
For our arguments, it will be important to be able to use the independence calculus in the extensions.
Thus, we will show that if $\M$ is a monster model for a FUR-class, then also $\M^{eq}$ is such.

We now make all this more precise.
Let $\mathcal{E}$ be a countable collection of $\emptyset$-invariant
equivalence relations $E$ such that $E \subseteq \M^n \times \M^n$ for some $n$.
By this we mean that if $E \in \mathcal{E}$, then $E$ is an equivalence relation on some model in $\mathcal{K}$ (note that from this it follows that $E$ is an equivalence relation on every model in $\K$; indeed, it takes at most three tuples to prove that a relation is not an equivalence relation, and by axiom AI all models are $s$-saturated)
and there is some countable collection $G_E$ of Galois-types so that $(a,b) \in E$ if and only if $t^g(ab/\emptyset) \in G_E$.
We assume that the identity relation is in $\mathcal{E}$, $= \in \mathcal{E}$ (note that there are only countably many Galois types over $\emptyset$).
 For every
$\A\in\K$ we let $\A^{eq}$ be the set
$\{ a/E\vert\ a\in\A ,\ E\in\mathcal{E}\}$. 
We identify each element $a$ with $a/=$. 
For each $E \in \mathcal{E}$, we add to
our language a predicate $P_{E}$
with the interpretation $\{ a/E\vert\ a\in\A\}$
and a function $F_{E}: \A^n \to \A^{eq}$ (for a suitable $n$) such that $F_{E}(a)=a/E$.  
Then, we have all the structure of $\A$ on $P_{=}$.  
We let $\K^{eq}=\{\A^{eq}\vert\ \A\in\K\}$.
We write $\A^{eq}\preccurlyeq^{eq}\B^{eq}$ if
$\A^{eq}$ is a submodel of $\B^{eq}$ and $\A\preccurlyeq\B$. 

We will now show that if $(\K, \preccurlyeq)$ is a FUR-class, then also
$(\mathcal{K}^{eq},\preccurlyeq^{eq})$ is a FUR-class.
 Notice first that for each model $\A$, the model $\A^{eq}$ is unique up to isomorphism over $\A$
and that every automorphism of $\A$ extends to an automorphism
of $\A^{eq}$. 
Thus it is easy to see that if $(\K, \preccurlyeq)$ is a FUR-class, then
$(\K^{eq},\preccurlyeq^{eq})$ is an AEC with AP, JEP and arbitrary
large models,  that $LS(\K^{eq})=\o$ and that $\K^{eq}$ does not contain finite models.
It is also easily seen that if the axioms AI, AIII, and AVI hold for $\mathcal{K}$, they hold also for $\mathcal{K}^{eq}$. 
 
We now show that also AII holds.

\begin{lemma}
Suppose that $\K$ is a FUR-class.
Then, axiom AII holds for $\K^{eq}$.
\end{lemma}

\begin{proof}
Let $\A^{eq} \in \K^{eq}$ be countable, and let $a$ be arbitrary.
We need to construct an $s$-primary model over $\A^{eq}a$.
Let $b_1, \ldots, b_n \in \M$ be such that $a=(F_{E_1}(b_1), \ldots, F_{E_m}(b_n))$ for some $E_1, \ldots, E_m \in \mathcal{E}$, and denote $b=(b_1, \ldots, b_n)$.
 We will first show that we may choose $b$ so that there is some finite $A \subset \A$ such that for all $b'$, $t(b'/Aa)=t(b/Aa)$ implies $t(b'/\A a)=t(b/\A a)$.
 
We note first that for this it suffices to find some $b=(b_1, \ldots, b_m)$ so that $a=(F_{E_1}(b_1), \ldots, F_{E_m}(b_n))$ and a finite set $A$ such that $t(b/\A)=t(b'/\A)$ whenever $t(b/A)=t(b'/A)$ and $(F_{E_1}(b_1'), \ldots, F_{E_m}(b_n'))=a$.
Indeed, suppose we have found such a tuple $b$ and such a set $A$.
Let $b'$ be such that $t(b'/Aa)=t(b/Aa)$.
Then, $a=(F_{E_1}(b_1'), \ldots, F_{E_m}(b_n'))$, and thus  $t(b/\A)=t(b'/\A)$.
We claim that moreover, $t(b' /\A a)=t(b / \A a)$.
If not, then there is some finite set $A' \subset \A$ such that $t(b'a/A') \neq t(ba/A')$.
Since $t(b'/A')=t(b/A')$, there is some $f \in \textrm{Aut}(\M/A')$ such that $f(b)=b'$.
But $f$ extends to an automorphism $f' \in \textrm{Aut}(\M^{eq}/A)$, and 
$$f'(a)=(F_{E_1}(f(b_1)), \ldots, F_{E_m}(f(b_m))=(F_{E_1}(b_1'), \ldots, F_{E_m}(b_n')=a,$$
where for each $i$, $f(b_i)$ denotes the relevant projection of the tuple $f(b)$.
Thus, $t(b'a/A')=t(ba/A')$, a contradiction.
 
To simplify notation, denote now by $F$ the function from $\M$ to $\M^{eq}$ that is given by $(F_{E_1}, \ldots, F_{E_m})$.
Let  $A_0 =\emptyset$ and let $b_0$ be such that $F(b_0)=a$.
If $b_0$ and $A_0$ are not as wanted, then there is some finite $A_1 \subset \A$ and some $b_1$ so that $F(b_1)=a$, $t(b_0/A_0)=t(b_1/A_0)$ and $t(b_1/A_1) \neq t(b_0/A_1)$.
Now we check if $A_1$ and $b_1$ are as wanted.
By AIV, we cannot continue this process infinitely, so at some step we have found $b=b_n$ and $A=A_n$ as wanted.  
  
Let now $\B=\A[b]=\A b \cup \bigcup_{i<\o} b_i \le \M$.
We claim that $\B^{eq}$ is $s$-primary over $\A^{eq}a$.
For this, we need to enumerate the elements of $\B^{eq}$ so that we may write $\B^{eq}=\A^{eq} a \cup \bigcup_{i<\o} c_i$.
Let $b=(b_0, \ldots, b_k)$, where each $b_i$ is a singleton.
For $0 \le i \le k$, we denote $c_i=b_i$.
By the above argument, the required isolation property is satisfied.
After this, we list the elements so that whenever $i<j$,  we have $b_i=c_{i'}$ and $b_j=c_{j'}$ for some $i'<j'$. 
Moreover, we take care that for each singleton $c \in \B^{eq}\setminus (\B \cup \A^{eq})$, the elements of some tuple $d \in \B$ such that $c=F_E(d)$ for some $E \in \mathcal{E}$ are listed before $c$ (i.e. if $d=(d_0, \ldots, d_k)$, then, $d_0=c_{i_0}, \ldots, d_k=c_{i_k}$ and $c=c_j$ for some $i_0< \ldots <i_k<j$).
Then, the required isolation properties are satisfied and we see that $\B^{eq}$ is indeed as wanted.
\end{proof}
  
\begin{lemma}
Suppose that $\K$ is a FUR-class.
Then, axiom AIV holds for $\K^{eq}$.
\end{lemma}
  
\begin{proof}
Let $a \in \M^{eq}$.
Again, there is some tuple $b \in \M$ so that $a=F(b)$ for some definable function $F$. 
Then, there is a number $n<\o$ so that for any countable $\A \in \K$ and finite $A' \subset \A$, player II wins $GI(b, A', \A)$ in $n$ moves.
Let now $\A^{eq} \in \K^{eq}$ be countable and $A \subset \A^{eq}$ finite.
We claim that player  II will win $GI(a, A, \A^{eq})$ in $n$ moves.
Let $A' \subset \A$ be such that every element $x \in A$ can be written as $x=F(y)$ for some $y \in A'$ where $F$ is a definable function.
Now, player II wins $GI(b, A', \A)$ in $n$ moves.
If there are some tuples $a', a'' \in \M^{eq}$ and some finite sets $C \subset B \subset \A^{eq}$ such that $t(a'/C)=t(a''/C)$ but $t(a'/B) \neq t(a''/B)$, then there are tuples $b', b'' \in \M$ and  a definable function $F$ so that $a' =F(b')$, $a'' =F(b'')$, and some $B' \subset \A$ and a definable function $H$ so that $B \subseteq H(B')$ and $t(b'/B') \neq t(b''/B')$.
Thus, the claim follows.  
\end{proof}  

\begin{lemma}
Suppose $\K$ is a FUR-class.
Then, axiom AV holds for $\K^{eq}$.
\end{lemma}

\begin{proof}
Suppose $\A^{eq}, \B^{eq} \in \K^{eq}$ are countable models, $\A^{eq} \subset \B^{eq}$, $a \in \M^{eq}$ and $\B^{eq} \da_{\A^{eq}}^{ns} a$.
We need to prove $a  \da_{\A^{eq}}^{ns} \B^{eq}$.
Suppose this does not hold.
Then, there is some $b \in \B^{eq}$ such that
$$a  \nda_{\A^{eq}}^{ns} b.$$

We note that we may assume $b \in \B$.
Indeed, there is some $b_0 \in \B$ and some definable function $F$ such that $b=F(b_0)$.
Then, $a  \nda_{\A^{eq}}^{ns} b_0$.
This follows from the fact that if  $A \subset \A^{eq}$ is finite and $t(b_0/A)=t(b_0'/A)$, then $t(b/A)=t(F(b_0)/A)=t(F(b_0')/A)$, and $t(ab_0/A)=t(ab_0'/A)$ implies $t(ab/A)=t(aF(b_0')/A)$.
 
Let now $a' \in \M$ be such that $a =F(a')$ for some definable function $F$.
We wish to apply Lemma \ref{vaplaajol} to find some $b' \in \M$ such that $t(b'/\A^{eq})=t(b/\A^{eq})$ and $b' \da_{\A^{eq}}^{ns} aa'$.
First, we note that we only needed axiom AI to prove Lemma \ref{vaplaajol}, so we can indeed apply it.
Secondly, Lemma \ref{vaplaajol} requires that there is some finite $A \subset \A^{eq}$ such that $t(a/\A^{eq})$ does not split over $A$.
But by Lemma \ref{ansA}, there is some finite $A \subset \A \subset \A^{eq}$ such that $t(a/\A)$ does not split over $A$.
By similar reasoning that was used above to prove that it suffices that $b \in \B$, one sees that then actually $t(a/\A^{eq})$ does not split over $A$.
Thus, we may apply Lemma \ref{vaplaajol} to find a suitable $b' \in \M$.
 
Since $\B^{eq} \da_{\A^{eq}}^{ns} a$, we have $b \da_{\A^{eq}}^{ns} a$, and thus 
$t(b'/\A^{eq}a)=t(b/\A^{eq}a)$ (note that also Lemma \ref{vaplaajykskas} requires only axiom AI).
Thus, to obtain a contradiction, it suffices to show that $a \da_{\A^{eq}}^{ns} b'$.
But $b' \da_{\A^{eq}}^{ns} aa'$ implies that $b' \da_\A^{ns} a'$ (in $\M$), and thus by Lemma \ref{symlemma}, 
$a' \da_\A^{ns} b'$.
It follows that $a' \da_{\A^{eq}}^{ns} b'$.  
\end{proof}  

We have now proved the following.

\begin{theorem}
If $(\K, \preccurlyeq)$ is a FUR-class, then also $(\K^{eq}, \preccurlyeq^{eq})$ is a FUR-class.
\end{theorem}

We wish to prove the existence of canonical bases.
With this in aim, we will now discuss global types. 
Let $\M' \in \K$ be a $\vert \M' \vert$ -model homogeneous and universal structure such that $\M \preccurlyeq \M'$ and $\vert \M' \vert > \vert \M \vert$.
We call  $\M'$ the \emph{supermonster}. 
Then, every $f \in \textrm{Aut}(\M)$ extends to some $f' \in \textrm{Aut}(\M')$.
In the following, we will abuse notation and write just $f$ for both maps. 

By a \emph{global type} $p$, we mean a maximal collection $\{p_A \, | \, A \subset \M \textrm{ finite }\}$ such that $p_A$ is a Galois type over $A$, and whenever $A \subseteq B$ and $b \in \M$ realizes $p_B$, then $b$ realizes also $p_A$.
We denote the collection of global types by $S(\M)$.
Moreover, we require that global types are consistent, i.e. that for each $p \in S(\M)$, there is some $b \in \M'$ such that $b$ realizes $p_A$ for every finite set $A \subset \M$ (note that the \emph{same} element $b \in \M'$ is required to realize $p_A$ for every $A$).
  
Let $f \in \textrm{Aut}(\M^{eq})$, $p \in S(\M)$.
We say that $f(p)=p$ if for all finite $A,B \subset \M$ such that $f(B)=A$ and all $b$ realizing $p_B$, it holds that $t(b/A)=p_A$.

\begin{definition}
Let $p \in S(\M)$.
We say that $\alpha \in \M^{eq}$ is a \emph{canonical base} for $p$ if it holds for every $f \in \textrm{Aut}(\M^{eq})$ that $f(p)=p$ if and only if $f(\alpha)=\alpha$.
\end{definition}

We will now prove that we may choose the collection $\mathcal{E}$ of equivalence relations in such a way that each global type will have a canonical base in $\M^{eq}$.
For this, we need to make sure that certain equivalence relations are included in $\mathcal{E}$.

Let $p \in S(\M)$ be a global type that does not split over $a \in \M$.
Suppose $b \in \M'$ realizes $p$.
Consider an arbitrary $c \in \M$ and let $q=t(b,c/\emptyset)$.
Since $t(b/\M)$ does not split over $a$, there are types $q_i$, $i<\o$, over $\emptyset$, so that for all $d \in \M$ the following holds: $bd$ realizes $q$ if and only if $ad$ realizes $q_i$ for some $i<\o$.
Indeed, there are only countably many types over the empty set, and from the non-splitting it follows that if $t(d_1 a/ \emptyset)=t(d_2 a/\emptyset)$, then $t(bd_1a/\emptyset)=t(bd_2a/\emptyset)$.
Thus, we may choose the types  $q_i$ as wanted.

For $c \in \M$, denote $q_c=t(b,c/\emptyset)$, and let $q_i^c$, $i<\o$ be such that for all $d \in \M$, $bd$ realizes $q_c$ if and only if $ad$ realizes $q_i^c$ for some $i<\o$.
Define the equivalence relation $E$ as follows: $(a_0, a_1) \in E$ if  the following holds for all $c,d \in \M$:  $a_0 d$ realizes $\bigvee_{i<\o}q_i^c$ if and only if $a_1 d$ realizes $\bigvee_{i<\o}q_i^c$.

We will later show that $a/E$ is a canonical base for $p$.
But first, we have to make sure that the equivalence relation $E$ can be included in $\mathcal{E}$ for every $p \in S(\M)$.
Namely, it needs to be verified that we can do with countably many such equivalence relations $E$.

Indeed, for each global type $p$ realized by an element $b \in \M'$, there is some tuple $a \in \M$ such that $t(b/\M)$ does not split over $a$.
Now the tuple $ba$ determines the equivalence relation $E$ described above, and we claim that $E$ depends only on $t(ba/\emptyset)$.
Indeed, let $b' \in \M'$, $a' \in \M$ be such that $b' \da_{a'} \M$ and $t(b'a'/ \emptyset)=t(ba/\emptyset)$.
Then, there is some $f \in \textrm{Aut}(\M/\emptyset)$ such that  $f(a)=a'$, and some $F \in \textrm{Aut}(\M'/\emptyset)$ extending $f$ (and in particular, $F(m) \in \M$ for every $m \in \M$).
Let $b''=F(b)$.
Then, since $b \da_a^{ns} \M$, we have $b'' \da_{a'}^{ns} \M$.
On the other hand, we have 
$$t(b'a'/\emptyset)=t(ba/\emptyset)=t(b''a'/\emptyset),$$
so $t(b'/a')=t(b''/a')$ and by Lemma \ref{statilemma}, $t(b'/\M)=t(b/\M)$.
Thus, there is an automorphism $G$ of $\M'$ such that $G(ab)=a'b'$ and $G(\M)=\M$ (we first take $ab \mapsto a'b''$ by $F$ and then fix $\M$ pointwise and take $b'' \mapsto b'$).
Then, the global type $p'=t(G(b)/\M)$ can be determined from $f(a)=a'$ in the same way as $p$ was determined from $a$, and the definition of $E$ stays the same.  

From now on, we will assume that all these equivalence relations $E$ are indeed included in $\mathcal{E}$.
This allows us to construct $\M^{eq}$ so that every global type has a canonical base.

\begin{lemma}\label{cbexists}
Suppose $p \in S(\M)$. 
Then, there is some  $\alpha \in \M^{eq}$ so that $\alpha$ is a canonical base for $p$.
\end{lemma}
 
\begin{proof}
Let $p \in S(\M)$ be a global type that does not split over $a \in \M$, and suppose
$b \in \M'$ realizes $p$.
Let $E$ be for $p$ and $a$ as described above.
 
We claim that $a/E$ is a canonical base for $p$.
Suppose first that $f \in \textrm{Aut}(\M)$ is such that $f(p)=p$.
We will show that $(a,f(a)) \in E$, which implies that $f(a/E)=f(a)/E=a/E$.
 As $f$ is an automorphism, we have $t(a/\emptyset)=t(f(a)/\emptyset)$.
Let now $c \in \M$ be arbitrary. 
Since $f(p)=p$, we have $$q_c=t(bc/\emptyset)=t(bf(c)/\emptyset).$$ 
for every $c \in \M$
Thus, we may choose $q_i^c=q_i^{f(c)}$ for all $i<\o$, and moreover we have
\begin{eqnarray*}
f(a)f(d) \textrm{ realizes } \bigvee_{i<\o}q_i^c \iff ad \textrm{ realizes } \bigvee_{i<\o}q_i^c 
\iff af(d) \textrm{ realizes } \bigvee_{i<\o}q_i^c, 
\end{eqnarray*}
where the first equivalence follows from $f$ being an automorphism, and the second one from the fact that $t(bd/\emptyset)=t(bf(d)/\emptyset)$. 
Since this holds for every $d$ and automorphisms are surjective, we may, for arbitrary $d' \in \M$, choose $d \in \M$ so that $d'=f(d)$ to obtain
\begin{eqnarray*}
f(a)d' \textrm{ realizes } \bigvee_{i<\o}q_i^{c} \iff  ad' \textrm{ realizes } \bigvee_{i<\o}q_i^{c},
\end{eqnarray*}
so $(a, f(a)) \in E$.

Suppose now $f(a/E)=a/E$.
Then, $(a,f(a)) \in E$.
We will show that $t(bc/\emptyset)=t(b f(c)/ \emptyset)$ for all $c \in \M$.
Then, clearly $f(p)=p$.
Let $c$ be arbitrary.
Denote $q^c=t(bc/\emptyset)$.
Then, $ac$, and thus also $f(a)f(c)$, realizes $\bigvee_{i<\o}q_i^c$.
But now, since $(a,f(a)) \in E$, we have that also $af(c)$ realizes $\bigvee_{i<\o}q_i^c$.
From this it follows that $bf(c)$ realizes $q^c$, i.e.
$$t(bc/\emptyset)=q_c=t(b f(c)/ \emptyset).$$

So, we have shown that $a/E$ is a canonical base for $p$, as wanted.
\end{proof} 
  
\begin{definition}
Let $a \in \M$ and let $A \subset \M$.
By Theorem \ref{main}, (iv), there is some $b \in \M'$ such that $Lt(a/A)=Lt(b/A)$ and $b \da_A \M$.
Let $p=t(b/\M)$.
By a \emph{canonical base} for $a$ over $A$, we mean a canonical base of $p$.
We write $\alpha=Cb(a/A)$ to denote that $\alpha$ is a canonical base of $a$ over $A$.
\end{definition}

Next, we prove some important properties of canonical bases.

\begin{lemma}\label{cbbcl}
Let $a \in \M$ and let $A \subset \M$ be a finite set.
Then, $Cb(a/A) \in \textrm{bcl}(A)$.
\end{lemma}

\begin{proof}
Let $b \in \M'$ be such that $Lt(a/A)=Lt(b/A)$ and $b \da_A \M$,
and let $p=t(b/\M)$.
Denote $\alpha=Cb(a/A)$, and suppose $\alpha \notin \textrm{bcl}(A)$.
Then, $t(\alpha/A)$ has unboundedly many realizations.
Suppose $t(\beta/A)=t(\alpha/A)$.
Then, there is some $f \in \textrm{Aut}(\M/A)$ such that $f(\alpha)=\beta$, and there is some $f' \in \textrm{Aut}(\M'/A)$ such that $f \subset f'$.
Then, $\beta$ defines the global type $t(f'(b)/\M)$.
Similarly, each realization of $t(\alpha/A)$ defines a global type.
By the definition of a canonical base, the global types defined by these unboundedly many elements are pairwise distinct. 
Let $f \in \textrm{Aut}(\M/A)$ and let $\alpha'=f(\alpha)$. 
Then $f$ extends to an automorphism $g$ of $\M'$, and we have $g(b)=b'$ for some $b' \in \M' \setminus \M$.
Since $g(\M)=\M$, we have $b' \da_A \M$.
Let $\A \subset \M$ be a countable model such that $A \subset \A$.
Then, by (v) in Theorem \ref{main}, we have $b \da_\A \M$ and $b' \da_\A \M$.
Since $t(b/\M) \neq t(b'/\M)$, by Lemmas \ref{kolme} and \ref{vaplaajykskas}, we must have 
$t(b/\A)Ê\neq t(b'/\A)$.
This means that we have uncountably many distinct types over the countable model $\A$, a contradiction against
Lemma \ref{notypes}.
\end{proof}

\begin{remark}\label{cbremark}
Let $a \in \M$, and let $A$ and $B$ be sets such that $A \subsetneq B$, and let $\alpha \in \M^{eq}$
If $a \da_A B$, then $\alpha=Cb(a/A)$ if and only if $\alpha=Cb(a/B)$.
\end{remark}
 
\begin{lemma}
Let  $a \in \M$ and let $\alpha=Cb(a/A)$.
Then, $a \da_\alpha A$.
\end{lemma}

\begin{proof}
Let $b \in \M'$ be such that $Lt(a/A)=Lt(b/A)$ and $b \da_A \M$, and let $p=t(b/\M)$.
Then, $\alpha$ is a canonical base of $p$.

We note first that $b \da^{ns}_\alpha \M$.
Indeed, if there were some $c,d \in \M$ that would witness the splitting, then there would be some automorphism $f$ fixing $\alpha$ such that $f(d)=c$.
Let $b'=f(b)$.
Now, we have
$$t(bd/\alpha) \neq t(bc/\alpha)=t(b'd/\alpha),$$
so $t(b/d\alpha) \neq t(b'/d\alpha)$, which is a contradiction since $f$ fixes the type $p$ (since it fixes $\alpha$).

In particular, by Remark \ref{kolmeremark}, this implies $b \da_\alpha \M$ and thus $b \da_\alpha A$.
Since $\alpha \in \textrm{bcl}(A)$, we have $b \da_A \alpha$ and $a \da_A \alpha$, so $t(a/A\alpha)=t(b/A \alpha)$, and thus $a \da_\alpha A$. 
\end{proof}

\begin{definition}
Let $p=t(d/B)$ for some $d \in \M$ and $B \subset \M$.
Suppose $B \subset C$.
We say that $p'=t(d'/C)$ is a \emph{free extension} of $p$ into $C$ if $t(d/B)=t(d'/B)$ and $d' \da_B C$.
We call a type \emph{stationary} if it has a unique free extension to any set.
If $p$ is a stationary type, we will denote the free extension of $p$ into $C$ by $p \vert_C$. 
\end{definition}

\begin{lemma}
Let $\alpha=Cb(a/A)$.
Then, $t(a/\alpha)$ is stationary.
\end{lemma}

\begin{proof}
Let $b \in \M'$ be such that $Lt(a/A)=Lt(b/A)$ and $b \da_A \M$, and let $p=t(b/\M)$.
Then, $\alpha$ is a canonical base of $p$.
As in the proof of the previous Lemma, $b \da^{ns}_\alpha \M$, so $b \da^{ns}_\alpha \M^{eq}$ and thus $b \da_\alpha \M^{eq}$.
Also, we have (as seen in the proof of the previous lemma) $t(b/\alpha)=t(a/\alpha)$.
Suppose now $\alpha \in B \subset \M^{eq}$ and there is some $c \in \M$ such that $t(c/\alpha)=t(a/\alpha)=t(b/\alpha)$, $c \da_\alpha B$ but $t(c/B) \neq t(b/B)$.
Let $b' \in \M'$ be such that $Lt(b'/\alpha)=Lt(c/\alpha)$ and $b' \da_\alpha \M$.
Then, $b' \da^{ns} \M^{eq}$, so
by Lemma \ref{statilemma}, $t(b'/\M)=t(b/\M)$, which is a contradiction, since 
$$t(b/B)\neq t(c/B)=t(b'/B).$$
\end{proof}
 
\subsection{Quasiminimal classes}

We have given the model class $\K$ of Example \ref{equivalence} as an example of a FUR-class (see \cite{lisuri} for details).
In this section, we will show that something more general is true: If
 $\K$ is a quasiminimal class in the sense of \cite{monet} (which is the same as in \cite{kir}, but without finite-dimensional structures), then $\K$ is a FUR-class, given that $\K$ only contains infinite-dimensional models.
It will then follow from Theorem \ref{main} that $\K$ has a perfect theory of independence.
 
Quasiminimal classes are AECs that arise from a quasiminimal pregeometry structure. 
Quasiminimal pregeometry structures can be seen as an analogue to strongly minimal structures.
They are defined as structures equipped with a pregeometry that has similar properties as the pregeometry obtained from the algebraic closure operator in the strongly minimal case.
In fact, it turns out that this pregeometry is actually obtained from the bounded closure operator, and that $U$-ranks are given as pregeometry dimensions.
 
In \cite{monet}, a quasiminimal pregeometry structure and a quasiminimal class are defined as follows.  

\begin{definition}\label{quasiminclass}
Let $M$ be an $L$-structure for a countable language $L$, equipped with a pregeometry $cl$ (or $\textrm{cl}_M$ if it is necessary to specify $M$).
We say that $M$ is a \emph{quasiminimal pregeometry structure} if the following hold:
\begin{enumerate}
\item (QM1) The pregeometry  is determined by the language.
That is, if $a$ and $a'$ are singletons and
$\textrm{tp}(a, b)=\textrm{tp}(a', b')$, then $a \in \textrm{cl}(b)$ if and only if $a' \in \textrm{cl}(b')$.
\item (QM2) $M$ is infinite-dimensional with respect to cl.
\item (QM3) (Countable closure property) If $A \subseteq M$ is finite, then $\textrm{cl}(A)$ is countable.
\item (QM4) (Uniqueness of the generic type) Suppose that $H,H' \subseteq M$ are countable closed subsets, enumerated so that $\textrm{tp}(H)=\textrm{tp}(H')$.
If $a \in M \setminus H$ and $a' \in M \setminus H'$ are singletons, then $\textrm{tp}(H,a)=\textrm{tp}(H',a')$ (with respect to the same enumerations for $H$ and $H'$).
\item (QM5) ($\aleph_0$-homogeneity over closed sets and the empty set)
Let $H,H' \subseteq M$ be countable closed subsets or empty, enumerated so that $\textrm{tp}(H)=\textrm{tp}(H')$,
and let $b, b'$ be finite tuples from $M$ such that $\textrm{tp}(H, b)=\textrm{tp}(H',b')$, and let $a$ be a singleton such that $a \in \textrm{cl}(H, b)$.
Then there is some singleton $a' \in M$ such that $\textrm{tp}(H, b,a)=\textrm{tp}(H',b',a').$
\end{enumerate}

We say $M$ is a \emph{weakly quasiminimal pregeometry structure} if it satisfies all the above axioms except possibly QM2.
\end{definition}

It is easy to see that the class from Example \ref{equivalence} satisfies the axioms.
On the other hand, Example \ref{suorasumma} is an example of a non-elementary FUR-class that does not arise from a quasiminimal pregeometry structure.
Indeed, if you choose two distinct primes, $p$ and $q$, then two generic elements of order $p$ and $q$, respectively, will have different types over the closure of the empty set, violating (QM4).
 
\begin{definition}
Suppose $M_1$ and $M_2$ are weakly quasiminimal pregeometry  $L$-structures.
Let $\theta$ be an isomorphism from $M_1$ to  some substructure of $M_2$.
We say that  $\theta$ is a \emph{closed embedding} if $\theta(M_1)$ is closed in $M_2$ with respect to $\textrm{cl}_{M_2}$, and $\textrm{cl}_{M_1}$ is the restriction of $\textrm{cl}_{M_2}$ to $M_1$.
\end{definition}

Given a quasiminimal pregeometry structure $M$, let $\K^-(M)$ be the smallest class of $L$-structures which contains $M$ and all its closed substructures and is closed under isomorphisms, and let $\K(M)$ be the smallest class containing $\K^-(M)$ which is also closed under taking unions of chains of closed embeddings.

From now on, we suppose that $\K=\K(\M)$ for some  quasiminimal pregeometry structure $\M$, and that we have discarded all the finite-dimensional structures from $\K$.
We will call such a class a \emph{quasiminimal class}.
For $\A, \B \in \K$, we define $\A \preccurlyeq \B$ if $\A$ is a closed submodel of $\B$. 
It is well known that $(\K, \preccurlyeq)$ is an AEC with $LS(\K)=\o$.
We may without loss assume that $\M$ is a monster model for $\K$.
In \cite{monet}, it is shown that $\K$ is totally categorical and has arbitrarily large models (Theorem 2.2).
It is easy to see that $\K$ has AP and JEP.
We will show that it is in fact a FUR-class. 

We first note that we may reformulate the conditions QM4 and QM5 so that the concept of Galois type is used instead of the concept of quantifier-free type. This will be useful in the arguments we later present.

Indeed, for QM4, let $H, H' \subset \M$ be countable and closed, let $t^g(H)=t^g(H')$, and let $a, a'$ be singletons such that $a \notin \textrm{cl}(H)$ and $a' \notin \textrm{cl}(H')$.
As $H$ and $H'$ are closed, they are models.
Since $H$ and $H'$ are countable, there is some isomorphism $f: H \to H'$.
Using QM4, we may extend $f$ to a map $g_0: H a \to H' a'$ that preserves quantifier-free formulae.
Let $\A=\textrm{cl}(Ha)$ and $\B=\textrm{cl}(H'a')$.
We will extend $g_0$ to an isomorphism $g: \A \to \B$.
Indeed, if $b \in \A=\textrm{cl}(Ha)$, then by QM5 and QM1, there is some $b' \in \B= \textrm{cl}(H'a')$ such that 
$\textrm{tp}(H,a,b)=\textrm{tp}(H,a',b')$, so $f_0$ extends to a map $f_1: H,a,b \to H',a',b'$ preserving quantifier-free formulae.
Since both $\A$ and $\B$ are countable, we can do a back-and-forth construction to obtain an isomorphism $g:  \A \to  \B$.
Then, $g(H,a)=(H',a')$ and 
 $g$ extends to an automorphism of $\M$, so 
 $t^g(H,a)=t^g(H',a')$, as wanted.
 
For QM5,  suppose $H, H' \subset \M$ are either countable and closed or empty, let $t^g(H)=t^g(H')$, and let $b,b' \in \M$ be such that $t^g(H,b)=t^g(H',b')$ and let $a \in \textrm{cl}(H,b)$.
Again, there is a map $f$ such that $f(H)=H'$, $f(b)=b'$ and $f$ preserves quantifier-free formulae.
As in the case of QM4, we may extend $f$ to an isomorphism $g: \textrm{cl}(Hb) \to \textrm{cl}(H'b')$.
If $a \in \textrm{cl}(Hb)$, then $t^g(H, b, a)=t^g(H',b',g(a))$.
 
We will now show that $\K$ satisfies the axioms AI-AVI.
For this, we will need the following auxiliary result.
It was first presented in a draft for \cite{monet} but eventually left out.

\begin{lemma}\label{poisjatetty}
Let $a \in \M$ and let $\A$ be a model and $A$ a finite set such that $t(a/\A)$ does not split over $A$.
Then, $\textrm{dim}_{cl}(a/A)=\textrm{dim}_{cl}(a/\A).$
\end{lemma}

\begin{proof}
Assume towards a contradiction that $\textrm{dim}_{cl}(a/A)>\textrm{dim}_{cl}(a/\A)$.
Let $b$ be a pregeometry basis for $a$ over $A$.
By the counter-assumption, there exists $c \in \A$ such that $\textrm{dim}_{cl}(b/Ad)<\textrm{dim}(b/A).$
We can choose $c$ to be independent over $A$.

Since $\textrm{dim}_{cl}(\A)$ is infinite, there exists $c' \in \A$ with the same length as $c$ such that $\textrm{dim}(b/Ac')=\textrm{dim}(b/A)$ and that $c'$ is independent over $A$.

By (QM4), $t(c/A)=t(c'/A)$, but clearly $t(c/Aa) \neq t(c'/Aa)$, so $t(a/\A)$ splits over $A$, a contradiction.
\end{proof}

\begin{lemma}\label{qmfur}
If $\K$ is a quasiminimal class, then $\K$ is a FUR-class. 
\end{lemma}

\begin{proof}
By \cite{monet}, $\K$ is totally categorical and has arbitrarily large models.
It is also easy to see that $\K$ has AP and JEP (for AP, use Theorem 3.3 in \cite{kir}).
It is quite straightforward to prove that the axioms AI-AVI hold.
As an example, we present the proof for AIV.
 
On move $n$ in $GI(a, A, \A)$,
let player II choose the set $A_{n+1}$ so that $t(a_{n+1}/\A)$ does not split over $A_{n+1}$ (such a set exists by Proposition 4.2 in \cite{monet}).
Then, player $I$ plays some tuple $a_{n+2}$ and some set $A_{n+2}'=A_{n+1}b$ so that $t(a_{n+2}/A_{n+1})=t(a_{n+1}/A_{n+1})$ but  $t(a_{n+2}/A_{n+2}')\neq t(a_{n+1}/A_{n+2}')$.
We claim that $\textrm{dim}_{cl}(a_{n+2}/\A)<\textrm{dim}_{cl}(a_{n+1}/\A)$, which means that player I can only move $\textrm{dim}_{cl}(a/\A)$ many times.

Let $m=\textrm{dim}_{cl}(a_{n+1}/\A)$.
By Lemma \ref{poisjatetty}, $\textrm{dim}_{cl}(a_{n+1}/A_{n+1})=m$, and thus  $\textrm{dim}_{cl}(a_{n+2}/A_{n+1})=m$, so $\textrm{dim}_{cl}(a_{n+2}/\A) \le m$.
Suppose $\textrm{dim}_{cl}(a_{n+2}/\A) = m$.
Then, in particular, $\textrm{dim}_{cl}(a_{n+2}/A_{n+1}b)=\textrm{dim}_{cl}(a_{n+1}/A_{n+1}b)=m$.
For $i=1,2$, write $a_{n+i}=a_{n+i}' a_{n+i}''$, where $a_{n+i}'$ is an $m$-tuple free over $A_{n+1}b$ and $a_{n+i}'' \in \textrm{cl}(A_{n+1} b a_{n+i}')$.
By QM4, there is some automorphism $\sigma$ fixing $A_{n+1} b$ pointwise so that $\sigma(a_{n+2}')=a_{n+1}'$.
Since $t(a_{n+2}/A_{n+1})=t(a_{n+1}/A_{n+1})$, we have
$$t(\sigma(a_{n+2}'')/A_{n+1} a_{n+1}')=t(a_{n+1}''/A_{n+1} a_{n+1}').$$
On the other hand, $t(a_{n+2}/A_{n+2}')\neq t(a_{n+1}/A_{n+2}')$, so
$$t(a_{n+1}''/A a_{n+1}' b)\neq t(\sigma(a_{n+2}'')/A a_{n+1}' b),$$
which contradicts the fact that $t(a_{n+1}''/A_{n+1} a_{n+1}')$ determines $t(a_{n+1}''/\A a_{n+1}')$ by the proof of Proposition 5.2 in \cite{monet}.
 \end{proof}

\begin{remark}
We note that in a quasiminimal pregeometry structure, $\textrm{cl}(A)=\textrm{bcl}(A)$.
Indeed, suppose $a \in \textrm{cl}(A)$.
Then, $a \in \textrm{cl}(A_0)$ for some finite $A_0 \subseteq A$.
Let $a'$ be such that $t(a/A)=t(a'/A)$.
By QM1, $a' \in \textrm{cl}(A_0)$. 
By QM3, $\textrm{cl}(A_0)$ is countable, so $t(a/A)$ only has countably many realizations.
Thus, $\textrm{cl}(A) \subseteq \textrm{bcl}(A)$.
On the other hand, suppose $a \notin \textrm{cl}(A)$.
By QM4, $t(a/\textrm{cl}(A))$ has uncountably many realizations, and thus $a \notin \textrm{bcl}(\textrm{cl}(A))$, and hence $a \notin \textrm{bcl}(A)$.
Thus, $\textrm{cl}(A)=\textrm{bcl}(A)$.
From now on we will use bcl for cl.

It is then easy to prove that in a quasiminimal pregeometry structure, $U(a/A)=\textrm{dim}_{\textrm{bcl}}(a/A)$.
\end{remark}
 
\section{The Group Configuration}
  
In this chapter, we adapt Hrushovski's group configuration theorem for the setting of quasiminimal classes.
We assume that $\K$ is a quasiminimal class, i.e. $\K=\K(\M)$ for some quasiminimal pregeometry structure, as in the last section of Chapter 2.
We may without loss of generality assume that $\M$ is a monster model for the class $\K$.
The proof in the elementary case can be found in e.g. \cite{pi}.  
Hrushovski's theorem states that if there are six tuples in the model that form a strict algebraic partial guadrangle (see \cite{pi}), then a group can be interpreted.
We will adapt the configuration to our setting by essentially replacing algebraic closures with bounded closures.
  
We will be working in $\mathbb{M}^{eq}$ and occasionally in $(\mathbb{M}^{eq})^{eq}$.
To avoid confusion, we will write $\textrm{bcl}^{eq}(A)$ for the bounded closure of $A$ in $\mathbb{M}^{eq}$.
In this case, $A$ might contain some element $a \in \mathbb{M}^{eq} \setminus \mathbb{M}$.
  
We will say a set $A$ is independent over $B$ if $a \downarrow_B  (A \setminus \{a\})$ for each $a \in A$. 

In analogue of interalgebraicity, we define interboundedness. 

\begin{definition}
We say $x$ and $y$ are \emph{interbounded} over a set $A$ if $x \in \textrm{bcl}(Ay)$ and $y \in \textrm{bcl}(Ax)$.
\end{definition}
 
We are now ready to present the configuration that will yield a group.
  
\begin{definition}\label{conf}
By a \emph{strict bounded partial quadrangle} over  a finite set $A$ we mean a $6$-tuple of elements $(a,b,c,x,y,z)$ in $\M^{eq}$, each of $U$-rank $1$ over $A$, such that 
\begin{enumerate}[(i)]
\item any triple of non-collinear points is independent over $A$ (see the picture below), i.e. has $U$-rank $3$ over $A$;
\item every line has $U$-rank $2$ over $A$ (see the picture).
\end{enumerate}
\end{definition}

\begin{figure}[htbp] \label{fig:multiscale}
\begin{center}
   \includegraphics[scale=0.7]{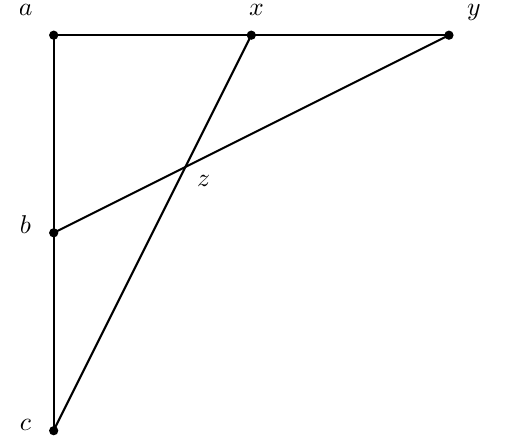}
\end{center}
\end{figure}

We will show that if there is a strict bounded partial quadrangle in $\M$, then a group can be interpreted (Theorem \ref{groupmain}).
The idea in Hrushovski's proof (see e.g. \cite{pi}) is that you (broadly speaking) view the tuples $x$, $y$, and $z$ as points on the plane 
and the tuples $a$, $b$, and $c$ as transformations that move the points.
He proves that $y$ and $z$ are interdefinable over $b$ (after making some small modifications), 
and thus $b$ can (broadly speaking) be seen as a function taking $y \mapsto z$. 
Eventually, he constructs a group that consists of this kind of functions. 

The small modifications that he makes mostly consist of replacing tuples with interalgebraic ones,
and the following remark ensures that we can follow the same strategy, in our case replacing tuples with interbounded ones.
The remark after it states that the tuples $a$, $b$ and $c$ can be replaced with certain canonical bases (this will be essential when constructing the functions).

\begin{remark}
If each of $a$,$b$,$c$,$x$,$y$,$z$ is replaced by an element interbounded with it over $A$, then  it is easy to see that the new $6$-tuple $(a',b',c',x',y',z')$ is also a strict bounded partial quadrangle over $A$. 
 
We say that this new partial quadrangle is \emph{boundedly equivalent} to the first one.
\end{remark}
   
\begin{remark}\label{cbinterbd}
If we have a strict bounded partial quadrangle, as in Definition \ref{conf}, then
$a$ is interbounded with $\textrm{Cb}(xy/Aa)$, $b$ is interbounded with $\textrm{Cb}(yz/Ab)$, and $c$ is interbounded with $\textrm{Cb}(zx/Ac)$.
\end{remark}

As a notion of definability, we use Galois definability, given in the definition below.
However, compared to the first order case, this is more similar to type definability than to actual first order definability.

\begin{definition}
We say that a tuple $a=(a_1, \ldots, a_n)$ is \emph{Galois definable} from a set $A$, if it holds for every $f \in \textrm{Aut}(\mathbb{M}/A)$
that $f(a)=(f(a_1), \ldots, f(a_n))=a$.
We write $a \in \textrm{dcl}(A)$, and say that $a$ is in the \emph{definable closure} of $A$.

We say that $a$ and $b$ are \emph{interdefinable} if $a \in \textrm{dcl}(b)$ and $b \in \textrm{dcl}(a)$.
We say that they are interdefinable over $A$ if  $a \in \textrm{dcl}(Ab)$ and $b \in \textrm{dcl}(Aa)$.
\end{definition}

\begin{definition}
We say that a set $B$ is \emph{Galois definable} over a set $A$, if every $f \in \textrm{Aut}(\M/A)$ fixes $B$ setwise.
\end{definition}
 
In the case of quasiminimal pregeometry structures, a set that is Galois definable over some set $A$ can be given as an infinite disjunction of Galois types over $A$.
 
\begin{definition}
We say that a group $G$ is \emph{Galois definable} over $A$ if $G$ and the group operation on $G$ are both Galois definable over $A$ as sets.
\end{definition}
 
For example, any group that is $\infty$-definable in first order over some set $A$ is Galois definable over $A$.
 
 \begin{definition}\label{genericeka}
 Let $B \subset \M$.
 We say an element $b \in B$ is \emph{generic} over some set $A$ if $U(b/A)$ is maximal (among the elements of $B$). 
 The set $A$ is not mentioned if it is clear from the context.
 For instance, if $B$ is assumed to be Galois definable over some set $D$, then we usually assume $A=D$. 
  
Let $p=t(a/A)$ for some $a \in \M$ and $A \subset A'$.
We say $b \in \M$ is a \emph{generic} realization of $p$ (over $A'$) if $U(b/A')$ is maximal among the realizations of $p$.
\end{definition}

We note that in the quasiminimal case, the length of a tuple gives an upper bound for its $U$-rank over any set, so generic elements always exist.

We will show that if there is a strict bounded partial quadrangle in a quasiminimal pregeometry structure $\M$, then a Galois definable group can be interpreted in $\M$.
In the first order setting, Hrushovski constructs a group from a strict algebraic partial quadrangle (see e.g. \cite{pi}) by viewing the tuple $b$ as a germ of a function taking $y$ to $z$.
However, we cannot straightforwardly follow Hrushovski's approach of showing that $y$ and $z$ are interdefinable over $b$.
Instead, we will find new tuples $a', x', c'$ and $a'', x'', c''$ so that also $(a', b, c', x', y,z)$ and $(a'', b, c'', x'', y,z)$ are strict, bounded partial quadrangles.
Then, we will show that $yx'$ and $zx''$ are interdefinable over $a'bc''$.
Another difference is that in the beginning of his proof, Hrushovski replaces the tuples $x,y,z$ with their finite sets of conjugates.
In our context, the conjugates are defined in terms of bounded closure, and the set of conjugates can be (countably) infinite.
Such a set cannot be taken as an element in our $\M^{eq}$, and thus we cannot use Hrushovski's trick.
Instead, we will, by defining suitable equivalence relations, move from the pregometry to the canonical geometry associated to it and work there. 
 
We will add the tuples $a'$ and $c''$ as parameters to our language and then view $b$ as a germ of function taking $yx'$ to $zx''$.
To be able to make this more precise, we need to introduce the concept of germs of definable functions.
We define them analoguously to the first order case.
 
\begin{definition}  
Suppose $p$ and $q$ are stationary types over some set $B$.
By a \emph{germ of an invertible definable function} from $p$ to $q$, we mean a Lascar type $r(u,v)$ over some finite set $C$, such that 
\begin{itemize}
\item $Lt(u,v)=r$ implies $t(u/B)=p$ and $t(v/B)=q$;
\item Suppose $Lt(u,v)=r$ and $D \subset \M$ is such that $C \subseteq D$. 
If $(u',v')$ realizes $r \vert_D$, then $u' \da_B D$, $v' \da_B D$, $v' \in \textrm{dcl}(u',D)$ and $u' \in \textrm{dcl}(v',D)$.
\end{itemize}
\end{definition}

We will denote germs of functions by the Greek letters $\sigma, \tau$, etc.
We note that the germs can be represented by elements in $\M^{eq}$. Just represent the germ determined by some Lascar type $r$, as above, by some canonical base of $r$.
If $\sigma$ is this germ and $u$ realizes $p \vert_\sigma$, then $\sigma(u)$ is the unique element $v$ such that $(u,v)$ realizes $r \vert_\sigma$.
Note that if $a$ realizes $p\vert_B$ and $\sigma \in B$, then $\sigma(a)$ realizes $q \vert_\sigma$, and as $\sigma(a) \da_\sigma B$, the element $\sigma(a)$ realizes $q \vert_B$.
 
We note that the germs can be composed.
Suppose $q'$ is another stationary type over $B$, $\sigma$ is a germ from $p$ to $q$ and $\tau$ is a germ from $q$ to $q'$.
Then, by $\tau.\sigma$ we denote a germ from $p$ to $q'$ determined as follows.
Let $u$ realize $p \vert_{\sigma, \tau}$.
Then, we may think of $\tau.\sigma$ as some canonical base of $Lt((u, \tau(\sigma(u)))/\sigma, \tau)$.
We note that $t(u, \tau(\sigma(a)))/\sigma, \tau)$ is stationary since $t(u/B \sigma \tau)$ is stationary as a free extension of a stationary type and since $\tau(\sigma(u))$ is definable from $u$, $\sigma$ and $\tau$.
Thus, $\tau.\sigma \in \textrm{dcl}(\sigma, \tau)$ (see the proof of Lemma \ref{cbexists}), and the notation is meaningful.

We are now ready to state the main theorem of this chapter.
We will prove it as a series of lemmas.

\begin{theorem}\label{groupmain}
Suppose $\K$ is a quasiminimal class, $\M$ is a monster model for $\K$, $A \subset \M$ is a finite set, $(a,b,c,x,y,z) \in \M$ is a strict bounded partial quadrangle over $A$ and $t(a,b,c,x,y,z/A)$ is stationary.
Then, there is a group $G$ in $(\M^{eq})^{eq}$, Galois definable over some finite set $A' \subset \M$.
Moreover, a generic element of $G$ has $U$-rank $1$.
\end{theorem}
   
\begin{proof} 
We note first that if we replace the closure operator bcl with the closure operator $\textrm{bcl}_A$ defined by $\textrm{bcl}_A(B)=\textrm{bcl}(A\cup B)$, we get from $\M$ a new quasiminimal class that is closed under isomorphisms and consists of models containing the set $A$.
We may think of this new class as obtained by adding the elements of $A$ as parameters to our language.
Then, $A \subseteq \textrm{cl}(\emptyset)$.
Thus, to simplify notation, we assume from now on that $A=\emptyset$.
When using the independence calculus developed in Chapter 2, we will write $A \da B$ for $A \da_\emptyset B$.

We begin our proof by replacing 
the tuple $(a,b,c,x,y,z)$ with one boundedly equivalent with it so that $z$ and $y$ become interdefinable over $b$.
For each $n$ we first define an equivalence relation $E^n$ on $\M^n$ so that $xE^ny$ if and only if $\textrm{bcl}(x)=\textrm{bcl}(y)$.
Similarly, define an equivalence relation $E^{*}$ on $\M^{eq}$ so that $xE^{*}y$ if and only if $\textrm{bcl}^{eq}(x)=\textrm{bcl}^{eq}(y)$.

\begin{lemma}\label{luokat}
For each $u \in \M^n$, the element $u/E^n$ is interdefinable with $(u/E^n)/E^*$ in $(\mathbb{M}^{eq})^{eq}.$
\end{lemma}

\begin{proof}
Clearly $(u/E^n)/E^* \in \textrm{dcl}(u/E^n)$.
Suppose now $v \in \M^n$ is such that $(v/E^n)/E^*=(u/E^n)/E^*$.
We note that for each $w \in \M^n$, $(\textrm{bcl}^{eq}(w/E)) \cap \M=\textrm{bcl}(w)$ and $((\textrm{bcl}^{eq})^{eq}((w/E)/E^*)) \cap \M^{eq}=\textrm{bcl}^{eq}(w/E)$.
Thus, $\textrm({bcl}^{eq})^{eq}((w/E)/E^*) \cap \M=\textrm{bcl}(w)$.
Hence, $\textrm{bcl}(v)=\textrm{bcl}(u)$, so $a/E^n=u/E^n$.
We have thus seen that $(u/E^n)/E^{*}$ determines $u/E^n$, so $u/E^n \in \textrm{dcl}((u/E^n)/E^{*})$. 
\end{proof}

We also note that if $U(u)=1$, then $U(u/E^n)=U((u/E^n)/E^*)=1$.
Indeed,  $u/E^n$ is interbounded with $u$ and thus has $U$-rank $1$. 
As $((u/E^n)/E^*)$ is interbounded with $u/E^n$, it also has $U$-rank $1$.

\begin{claim}
Without loss of generality, each of the tuples $x$, $y$, $z$ is of the form $u/E^n$ for some $n$ and some $u \in \M^n$.
\end{claim}

\begin{proof}
If we replace $x$ with $x/E^n$, $y$ with $y/E^n$ and $z$ with $z/E^n$, then these elements are interbounded
with the old ones, so we still have a strict bounded partial quadrangle over $A$. 
\end{proof}

  
Let $a' \in \M$ be such that $Lt(a'/b,z,y)=Lt(a/b,z,y)$
and $a' \da_{bzy}  abcxyz$.
Since we have a bounded partial quadrangle, $a' \da bzy$ and thus by transitivity $a' \da abcxyz$. 
Then, there are tuples $c',x'$ such that $Lt(a',c',x'/ b,z,y)=Lt(a,c,x/b,z,y)$ and in particular, $t(a',b,c',x',y,z/\emptyset)=t(a,b,c,x,y,z/\emptyset)$.
Thus, $(a',b,c',x',y,z)$ is a strict bounded partial quadrangle over $\emptyset$.
Similarly, we find an element $c'' \in \M$ such that $Lt(c''/bzy)=Lt(c/bzy)$ and $c'' \da abcxyza'c'x'$,  and  there are elements $a'', x''$ so that $(a'',b,c'',x'',y,z)$ is a strict bounded partial quadrangle over $\emptyset$.
The below picture may help the reader.
\begin{figure}[htbp] \label{fig:multiscale}
\begin{center}
   \includegraphics[scale=0.9]{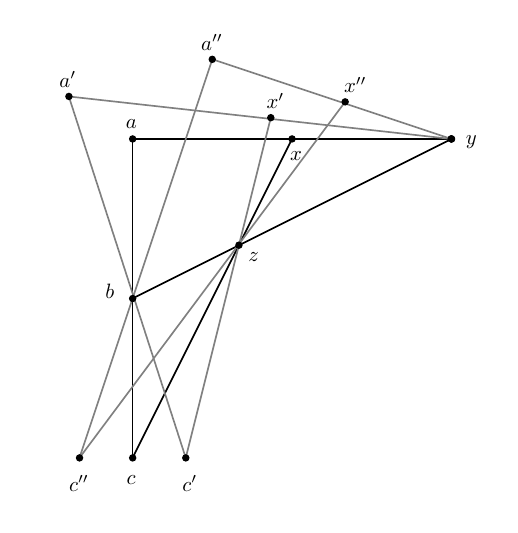}
\end{center}
\end{figure}

We will add the elements $a'$ and $c''$ as parameters in our language, but this will affect the closure operator and the independence notion.
In our arguments, we will be doing calculations both in the set-up we have before adding these parameters and the one after adding them.
We will use the notation cl and $\da$ for the setup before adding the parameters, and $\textrm{cl}^*$ and $\da^*$ for the setup after adding the parameters, i.e. for any sets $B,C,D$, $\textrm{cl}^*(B)=\textrm{cl}(B,a',c'')$ and 
$B \da^{*}_C D$ if and only if $B \da_{Ca' c''} D$.
Similary, we write $u \in \textrm{dcl}^*(B)$ if and only if $u \in \textrm{dcl}(Ba'c'')$ and use the notation $\textrm{Cb}^*(u/B)$ for $\textrm{Cb}(u/Ba'c'')$.
 
\begin{lemma}\label{interdefinable}
The tuples $yx'$ and $zx''$ are interdefinable over $a''bc'$ in $\mathbb{M}^{eq}$ after adding the parameters $a'$ and $c''$ to the language.
\end{lemma}
 
\begin{proof} 
We first prove an auxiliary claim.
 
\begin{claim}\label{tyyppi}
If $t(z'/byc'x')=t(z/byc'x')$, then $\textrm{bcl}^{eq}(z)=\textrm{bcl}^{eq}(z')$ in $\M^{eq}$.
\end{claim}

\begin{proof}
We will show that both $z$ and $z'$ are interbounded (with respect to bcl) with $\textrm{Cb}(b,y/c',x')$ and thus interbounded with each other.
Denote $\alpha=\textrm{Cb}(b,y/c',x')$.
The set $\{b,c',z\}$ is independent. 
In particular, $b \downarrow_z c'$.
But $y \in \textrm{bcl}(b,z)$ and $x' \in \textrm{bcl}(c',z)$.
Hence, $by \downarrow_z c' x'$.
Since $z \in \textrm{bcl}(c'x')$, we have $\alpha=Cb(by/c'x'z)$ by Remark \ref{cbremark}, and since $by \da_z c'x'$, we have $\alpha=Cb(by/z)$.
Thus, $\alpha \in \textrm{bcl}(z)$ by Lemma \ref{cbbcl}.
We also have $by \da_\alpha c'x'$, so
$$U(by/\alpha)=U(by/\alpha c'x')=U(by/c'x')=1,$$
where the second equality follows from the fact that $\alpha \in \textrm{bcl}(z)$ and $z \in \textrm{bcl}(c',x')$.
Now $\alpha \notin \textrm{bcl}(\emptyset)$, since then we would have $U(by/\emptyset)=1$, contradicting our assumptions.
Thus, $\alpha \nda z$, and hence $z \nda \alpha$, so $z \in \textrm{bcl}(\alpha)$.
Hence we have seen that  $z$ is interbounded with $\alpha=\textrm{Cb}(b,y/c',x')$.
Since $t(z'/byc'x')=t(z/byc'x')$, the same holds for $z'$.
Thus, $z$ and $z'$ are interbounded.
\end{proof}

By Claim \ref{tyyppi}, 
$u=z/E^*$ if and only if there is some $w$ such that $t(w/byc'x')=t(z/byc'x')$ and $w/E^*=u$.
From this, it follows that $z/E^* \in \textrm{dcl}(byc'x')$ in $(\M^{eq})^{eq}$.  
Thus, by Lemma \ref{luokat}, $z \in \textrm{dcl}(byc'x') \subseteq \textrm{dcl}(a''bc'yx')$.  

For $zx'' \in \textrm{dcl}^{*}(a''bc'yx')$, it suffices to show that $x'' \in \textrm{dcl}^{*}(a''bc'yx'z)$.
If $t(x^*/a''yzc'')=t(x''/a''yzc'')$, then $\textrm{bcl}(x^*)=\textrm{bcl}(x'')$ (this is proved like Claim \ref{tyyppi}), and thus $x''/E^* \in \textrm{dcl}(a''c''yz) \subseteq \textrm{dcl}^{*}(a''bc'x'yz)$ (note that $\textrm{dcl}^*$ is defined with $c''$ as a parameter).
By Lemma \ref{luokat}, $x'' \in \textrm{dcl}^*(a''bc'x'yz)$.

Similarly, one proves that $yx' \in \textrm{dcl}^*(a''bc'zx'')$.
\end{proof}

Let $q_1=t(yx'/a'c'')$, $q_2=t(zx''/a'c'')$.
We will consider $Cb(yx',zx''/a''bc')$ as a germ of an invertible definable function from $q_1$ to $q_2$.
But for this, we need the types $q_1$ and $q_2$ to be stationary.
This can be obtained by doing a small trick.

\begin{claim}
We may without loss assume that the types $q_1$ and $q_2$ are stationary.
\end{claim}
 
\begin{proof}
We will show that the types become stationary after adding certain parameters to the language.

In order to simplify notation, denote for a while $d=(a,b,c,x,y,z,c',x',a'', x'')$.
Choose now a tuple $d' \in \M^{eq}$ such that $Lt(d'/a'c'')=Lt(d/a'c'')$ and $d' \da_{a'c''} d$.
Now, there is some tuple $d'' \in \M$ and a definable function $F$ such that $d'=F(d)$ (choose $F$ to be a 
cartesian product of the functions $F_E$ for the relevant equivalence relations $E$ (see \ref{meqsect})).
Let now $d^* \in \M$ be such that $Lt(d^*/a'c''d')=Lt(d''/a'c''d')$ and $d^* \da_{a'c''d'} d$.
Then, $F(d^*)=F(d'')=d'$, and by the choice of $d'$, symmetry and transitivity, we get $d^* \da_{a'c''} d$.

We claim that for any subsequence $e \subseteq d$, the type $t(e/a'c''d^*)$ is stationary.
Indeed, there is some subsequence $e^* \subset d^*$ such that $Lt(F'(e^*)/a'c'')=Lt(e/a'c'')$ for some definable function $F'$ (here, $F'$ is some subsequence of the functions $F_E$ that form $F$).
Thus, $t(e/a'c''e')$ (and hence $t(e/a'c''d'')$) determines $Lt(e/a'c'')$.
Let $a'c''d^* \subseteq B$ and  $f_1, f_2$ are such that $t(f_1/a'c''d^*)=t(f_2/a'c''d^*)=t(e/a'c''d^*)$ and $f_i \da_{a'c''d^*} B$ for $i=1,2$. 
Then, $Lt(f_1/a'c'')=Lt(f_2/a'c'')$.
Since $f_i \da_{a'c''} d^*$ for $i=1,2$, we have by transitivity $f_i \da_{a'c''} B$, and thus $Lt(f_1/B)=Lt(f_2/B)$.
So the type is indeed stationary.

Now, we add the tuple $d^*$ as parameters to our language.
Since it is independent over $a'c''$ from everything that we will need in the independence calculations that will follow, the calculations won't depend on whether we have added $d^*$ or not.
Thus, we may without loss assume $d^*=\emptyset$.
\end{proof}
  
We will show that we can consider $Cb(yx',zx''/a''bc')$ as a germ of an invertible definable function from $q_1$ to $q_2$, and that we may without loss suppose that $b=Cb(yx',zx''/a''bc')$.
Then, we will prove that for independent  $b_1, b_2$ realizing $\textrm{tp}(b/a'c'')$, $b_1^{-1}.b_2$ is a germ of an invertible definable function from $q_1$ to $q_1$.  

Note first that as $a'' \in \textrm{bcl}(bc'') \subseteq \textrm{bcl}^{*}(b)$ and $c' \in \textrm{bcl}(a'b) \subseteq \textrm{bcl}^{*}(b)$,
 we have $\textrm{Cb}(yx',zx''/a''bc')=\textrm{Cb}^*(yx',zx''/b)$.
Thus, from Lemma \ref{interdefinable}, it follows that the tuples $yx'$ and $zx''$ are interdefinable over $\textrm{Cb}^*(yx',zx''/b)$ after adding the parameters.
Hence, we may view $\textrm{Cb}^*(yx',zx''/b)$ as a germ of a function taking $yx' \mapsto zx''$.

We claim that after adding the parameters, $b$ is interbounded with $\textrm{Cb}^*(yx',zx''/b)$.
Denote now $\alpha=\textrm{Cb}^*(yx',zx''/b)$.
Clearly, $\alpha \in \textrm{bcl}^*(b)$.
We have $yx' zx'' \da_\alpha^* b$ and thus $b \da_\alpha^* yx'zx''$.
Since $b \in \textrm{bcl}^*(y,z)$, we have $b \in \textrm{bcl}^*(\alpha)$ by Theorem \ref{main}, (xii).
Thus, we may without loss assume that $b=\textrm{Cb}^*(yx',zx''/b)$.

Let $r=t(b/a',c'')$.
If $b_1, b_2$ realize $r$, then by $b_1^{-1}.b_2$ we mean the germ of the invertible definable function from $q_1$ to $q_1$ obtained by first applying $b_2$, then $b_1^{-1}$.
In other words, let $y_1 x_1'$ realize $q_1 |_{b_1b_2}$, and let $z_1x_1''=b_2.(y_1 x_1')$.
So $z_1x_1''$ realizes $q_2|_{ b_1b_2}$.
Let $y_2 x_2'=b_1^{-1}.(z_1x_1'')$ (i.e. $z_1x_1''=b_1.(y_2 x_2')$).
We may code the germ $b_1^{-1}.b_2$ by some canonical base of the type $p=t(y_1 x_1',y_2 x_2'/b_1,b_2,a',c'')$, i.e.
we will have $b_1^{-1}.b_2=\textrm{Cb}^*(y_1 x_1',y_2 x_2'/b_1,b_2)$.
A type can have different canonical bases that don't necessarily have the same Galois type.
However, we now fix a Galois type for this canonical base, and whenever we consider a canonical

This type can have different canonical bases, and all of them might not have the same Galois type.
Thus we now choose one canonical base and fix its Galois type; whenever we will consider a canonical base for $p$, we will assume it to have this particular Galois type.
 
As noted before, we have
$b_1^{-1}.b_2 \in \textrm{dcl}^*(b_1,b_2)$.

\begin{lemma}\label{vapaus}
Let $b_1$, $b_2$ realize $r$ ($=t(b/a'c'')$), and let $b_1 \da^* b_2$. 
Then, $b_1^{-1}.b_2 \da^* b_i$ for $i=1,2$.
In particular, $U(b_1^{-1}.b_2/a'c'')=1.$
\end{lemma}

\begin{proof}
Without loss of generality, $b_2=b$ and $b_1 \da^* a,b,c,x,y,z,c',x',a'',x''$.
  
We have $a' \da bzx$, and thus $b \da_{zx} a'$.
Since $b \da zx$, we get $b \da a'zx$.
On the other hand, $c'' \da a'bzx$, and thus (since $b \da a'zx$) $b \da a'c''zx$.
This implies $b \da^* zx$.
Since, $x'' \in \textrm{bcl}^*(z)$ and $c \in \textrm{bcl}^*(zx)$, we have $b \downarrow^* cxzx''$.
 Hence, $t(b/a'c''cxzx'')=t(b_1/a'c''cxzx'')$ (remember that $r$ is stationary due to the trick we did earlier), and there are elements $a_1, y_1, c_1',x_1',a_1''$ so that
$$t(a_1,b_1,c,x,y_1,z,c_1',x_1',a_1'',x''/a'c'')=t(a,b,c,x,y,z,c',x',a'',x''/a'c'').$$
To visualize this, think of the picture just before Lemma \ref{interdefinable}.
In the picture, keep the lines $(c,x,z)$ and $(c'',z,x'')$ fixed pointwise and move $b$ to $b_1$ by an automorphism fixing $a'c''$.
As a result, we get another similar picture drawn on top of the first one, with new elements $a_1, y_1, c_1'$ and $a_1''$ in the same configuration with respect to the fixed points as $a,y,c$ and $a''$ in the original picture.
 
\begin{claim}\label{i}
$a a_1 b b_1 \downarrow^* yx'.$ 
\end{claim}
\begin{proof}
By similar arguments as before, one sees that $y \da_{a'c''abc} b_1$ and $y \da a'c''abc$, so $y \da_{a' c''} abcb_1$
by transitivity.
As $a_1 \in \textrm{bcl}^*(b_1, c)$, we have  (by symmertry) $aa_1b b_1c \downarrow^* y$ and thus $a a_1 b b_1 \downarrow^* y.$
As $x' \in \textrm{bcl}^*(y)$, we have $a a_1 b b_1 \downarrow^* yx'.$
\end{proof}

\begin{claim}\label{ii}
$y_1 x_1' \in \textrm{bcl}^*(a,a_1,y)$
\end{claim}
\begin{proof}
$x \in \textrm{bcl}^*(a,y)$, $y_1 \in \textrm{bcl}^*(a_1,x)$ and $x_1' \in \textrm{bcl}^*(y_1)$.
\end{proof}

\begin{claim}\label{iii}
$y_1x_1'=(b_1^{-1}.b)(yx')$. 
\end{claim}
\begin{proof}
By Claim \ref{i}, $yx' \downarrow^* b b_1$, so it realizes $q_1 \vert_{bb_1}$.
On the other hand, $t(b_1y_1x_1'/a'c'')=t(byx'/a'c'')$ so $y_1 x_1' \da^* b_1$.
By similar arguments that were used to show that we may assume $b=Cb^*(yx',zx''/b)$, we also
see that we may assume $b_1=Cb^*(y_1x_1',zx''/b_1)$.
Thus, $b: yx' \mapsto zx''$ and $b_1:y_1x_1' \mapsto zx''$.
\end{proof}

\begin{claim}\label{iv}
$a a_1 \downarrow b.$
\end{claim}
\begin{proof}
$abc \downarrow^* b_1$, and thus $ab \downarrow_c^* b_1$.
As $a_1 \in \textrm{bcl}^*(b_1,c)$, we have $ab \downarrow_c^* a_1$.
By similar type of calculations that we have done before, we see that $a \da a' c'' c$.
Since $t(a_1/a'c''c)=t(a/a'c''c)$, we have that $a_1 \da a' c'' c$.
Together with $ab \downarrow_c^* a_1$, this implies $ab \da^* a_1$.
Using the independence calculus, one can verify that $b \da a'c''$ and $a \da^* b$, and thus 
$$U(aa_1b/a'c'')=U(b/a'c'')+U(a/ba'c'')+U(a_1/aba'c'')=1+1+1=3,$$
so $aa_1 \da^* b$, as wanted.
 \end{proof}

\begin{claim}\label{v}
$aa_1 \downarrow b_1.$
\end{claim}
\begin{proof}
Like Claim \ref{iv}.
\end{proof}

Denote $\sigma=b_1^{-1}.b$.
Now by Claim \ref{i}, $yx' \downarrow_{aa_1}^* a a_1 b b_1$.
Thus, by Claim \ref{ii}, $yx'y_1x_1' \downarrow _{aa_1}^* a a_1 b b_1$.
On the other hand, by Claim \ref{i}, $yx' \downarrow_{bb_1}^* a a_1 b b_1$.
By Claim \ref{iii}, $y_1x_1' \in \textrm{bcl}^*(yx',b,b_1)$, so $yx'y_1x_1' \downarrow_{bb_1}^* a a_1bb_1$.
Recall that $\sigma=\textrm{Cb}^*(yx',y_1x_1'/b,b_1)$, and thus 
$$\sigma=\textrm{Cb}^*(yx',y_1x_1'/a,a_1,b,b_1).$$
So, $\sigma \in \textrm{bcl}^*(a,a_1)$ since $yx'y_1x_1' \downarrow _{aa_1}^* a a_1 b b_1$.
By Claims \ref{iv} and \ref{v}, $\sigma \downarrow^* b$ and $\sigma \downarrow^* b_1$.

Since $\sigma \in \textrm{dcl}^*(b_1,b_2)$ and $\sigma \downarrow^* b_i$ for $i=1,2$, we have $U(\sigma/a'c'') \le 1$ (remember that $dcl^*$ and $\da^*$ were defined using $a'c''$ as parameters).
Since $b_1$ and $b_2$ realize $r=t(b/a'c'')$, $U(b/a'c'')=1$, $b_1 \da^*b_2$, and $b_1 \in dcl^*(\sigma, b_2)$, we have $U(\sigma/a'c'') \ge 1$.

\end{proof}
 
Denote now $\sigma=b_1^{-1}.b_2$ (from Lemma \ref{vapaus}) and let $s=t(\sigma/a'c'')$ (note that $t(\sigma^{-1}/a'c'')=s$ also).

\begin{lemma}\label{ykkonen}
Let $\sigma_1, \sigma_2$ be realizations of $s$ such that $\sigma_1 \da^* \sigma_2$.
Then, $\sigma_1.\sigma_2$ realizes $s|_{\sigma_i}$ for $i=1,2$.
\end{lemma}

\begin{proof}
Similar as in \cite{pi} (see also \cite{lisuri}). 
\end{proof}

Let $G$ be the group of germs of functions from $q_1$ to $q_1$ generated by $\{\sigma \, | \, \sigma \textrm{ realizes } s\}$ (note that this set is closed under inverses and thus indeed a group). 

\begin{lemma}
For any $\tau \in G$, there are $\sigma_1, \sigma_2$ realizing $s$ such that $\tau=\sigma_1.\sigma_2$. 
\end{lemma}

\begin{proof}
Similar as in \cite{pi} (see also \cite{lisuri}). 
\end{proof}

Consider the set 
$$G'=\{(\sigma_1, \sigma_2) \, | \, \textrm{$\sigma_1, \sigma_2$ are realizations of $s$}\}.$$
It is clearly Galois definable over $a'c''$.
Let $E$ be the equivalence relation such that for $\gamma_1, \gamma_2 \in G'$, $(\gamma_1, \gamma_2) \in E$ if and only if $\gamma_1(u)=\gamma_2(u)$ for all $u$ realizing $q_1 \vert_{\gamma_1 \gamma_2}$.
Then, $G=G'/E$, and $G$ is  Galois definable over $a'c''$.

It remains to prove that  for a generic $\sigma_1. \sigma_2$ it holds that $U(\sigma_1.\sigma_2/a'c'')=1$.
We note first that for $\sigma=b_1^{-1}.b_2$, we have $U(\sigma/a'c'')=1$.
Indeed, since $\sigma \da^* b_1$, we have
$$U(\sigma/a'c'')=U(\sigma/a'c''b_1) \le U(b_2/a'c''b_1)=1,$$
where the inequality follows from the fact that $\sigma \in \textrm{dcl}^*(b_1, b_2)$, and the last equality from the fact that $b_1 \da^* b_2$.
On the other hand, we cannot have $\sigma \in \textrm{bcl}(a'c'')$, since $b_2 \in \textrm{dcl}(b_1, \sigma)$.
Thus, $U(\sigma/a'c'')=1$.
If $\sigma_1 \da^* \sigma_2$, then by Lemma \ref{ykkonen}, $\sigma_1.\sigma_2$ realizes $s$, and thus $U(\sigma_1.\sigma_2/a'c'')=1$.
If $\sigma_1 \nda^* \sigma_2$, then $U(\sigma_1.\sigma_2/a'c'') \le 1$.
This proves the theorem.
\end{proof}

\begin{remark}\label{locmodgrouprem}
Using the group configuration theorem, it is easy to see that if $(\M, bcl)$ is a non-trivial, locally modular pregeometry, then there exists a group (see \cite{lisuri} for details).
\end{remark}
 
\section{Groups in Zariski-like structures}

In this chapter we suppose that $\M$ is a monster model for a quasiminimal class as introduced in Chapter 2.
As an attempt to generalize Zariski geometries to this context, we will present axioms for a Zariski-like structure.
These axioms capture some of the properties of the irreducible closed sets in Zariski geometries that are needed for finding a group in that context.
We then apply the group configuration theorem from Chapter 3 to show that if $\M$ satisfies these axioms and the pregeometry obtained from the bounded closure operator is non-trivial, then a $1$-dimensional group can be found in $(\M^{eq})^{eq}$.
The argument is a modification of the one presented for Zariski geometries in \cite{HrZi}.

To simplify notation, we often write $a \da b$ for $a \da_\emptyset b$ and $U(a)$ for $U(a/\emptyset)$.
In the following definition, when speaking about indiscernible sequences, we don't assume that they are non-trivial.
Recall that by generic elements of a set, we mean the elements with maximal $U$-rank (see Definition \ref{genericeka}).
  
Before we can list the axioms for a Zariski-like structure, we need some auxiliary definitions.
First, we generalize the notion of specialization from \cite{HrZi} to our context.

\begin{definition}\label{spesialmaar}
Let $\M$ be a monster model for a quasiminimal class, $A \subset \M$, and let $\mathcal{C}$ be a collection of subsets of $\M^n$, $n=1,2,\ldots$.
We say that a function $f: A \to \M$ is a \emph{specialization} (with respect to $\mathcal{C}$) if for any $a_1, \ldots, a_n \in A$ and for any $C \in \mathcal{C}$, it holds that if $(a_1, \ldots, a_n) \in C$, then $(f(a_1), \ldots, f(a_n)) \in C$.
If $A=(a_i : i \in I)$, $B=(b_i:i\in I)$ and the indexing is clear from the context, we write $A \to B$ if the map $a_i \mapsto b_i$, $i \in I$, is a specialization.

We say that the specialization $f$ is an \emph{isomorphism} if also the converse holds, that is if $C \in \mathcal{C}$ and $(f(a_1), \ldots, f(a_n)) \in C$, then $(a_1, \ldots, a_n) \in C$.

If $a$ and $b$ are finite tuples and $a \to b$, we denote $\textrm{rk}(a \to b)=\textrm{dim}(a/\emptyset)-\textrm{dim}(b/\emptyset)$.
\end{definition}
 
The specializations in the context of Zariski geometries in \cite{HrZi} are specializations in the sense of our definition if we take $\mathcal{C}$ to be the collection of closed sets (Zariski geometries are quasiminimal since they are strongly minimal).
 
Next, we generalize the definition of regular specializations from \cite{HrZi}.  

\begin{definition}
Let $\M$ be a monster model for a quasiminimal class.
We define a \emph{strongly regular} specialization recursively as follows:
\begin{itemize}
\item Isomorphisms (in the sense of Definition \ref{spesialmaar}) are strongly regular;
\item If $a \to a'$ is a specialization and $a \in \M$ is generic over $\emptyset$, then $a \to a'$ is strongly regular;
\item $aa' \to bb'$ is strongly regular if $a \downarrow_\emptyset a'$ and the specializations $a \to b$ and $a' \to b'$ are strongly regular.
\end{itemize}
\end{definition}

\begin{remark}
It follows from Assumptions 6.6 (7) in \cite{HrZi} (for a more detailed discussion on why these properties hold in a Zariski geometry, see \cite{lisuri}, Chapter 1.1.) that if a specialization on a Zariski geometry is strongly regular in the sense of our definition, then it is regular in the sense of \cite{HrZi} (definition on p. 25).
\end{remark}
  
The following generalizes the definition of good specializations from \cite{HrZi}.

\begin{definition}\label{good}
We define a \emph{strongly good} specialization recursively as follows.
Any strongly regular specialization is strongly good.
Let $a=(a_1, a_2, a_3)$, $a'=(a_1', a_2', a_3')$, and $a \to a'$.
Suppose:
\begin{enumerate}[(i)]
\item $(a_1, a_2) \to (a_1', a_2')$ is strongly good.
\item $a_1 \to a_1'$ is an isomorphism.
\item $a_3 \in \textrm{cl}(a_1)$.
\end{enumerate}
Then, $a \to a'$ is strongly good.
\end{definition}

A Zariski-like structure is defined by nine axioms as follows.  
Axioms (ZL1)-(ZL6) are some basic properties that hold for Zariski geometries.
Axioms (ZL7) and (ZL8) are based on some results proved for Zariski geometries in \cite{HrZi} 
that are needed for finding the group.
Axiom (ZL9) is designed to bring back some traces of Compactness to the setting, since that is needed at the
very end of the argument for finding the group (this will be discussed in more detail below).

\begin{definition}\label{zariskilike}
We say that a quasiminimal pregeometry structure $\M$ is \emph{Zariski-like} if there is a countable collection $\mathcal{C}$ of subsets of $\M^n$ ($n=1, 2, \ldots$), which we call the \emph{irreducible} sets, satisfying the following axioms (all specializations are with respect to $\mathcal{C}$).
\vspace{0.1cm}
  
\noindent  
(ZL1) Each $C \in \mathcal{C}$ is Galois definable over $\emptyset$. \\
(ZL2) For each $a \in \M$, there is some $C \in \mathcal{C}$ such that $a$ is generic in $C$. \\
(ZL3) If $C \in \mathcal{C}$, then the generic elements of $C$ have the same Galois type. \\
(ZL4) If $C,D \in \mathcal{C}$, $a \in C$ is generic and $a \in D$, then $C \subseteq D.$\\
(ZL5) If $C_1, C_2 \in \mathcal{C}$, $(a,b) \in C_1$ is generic, $a$ is a generic element of $C_2$ and $(a',b') \in C_1$, then $a' \in C_2$.\\
(ZL6) If $C \in \mathcal{C}$, $C \subset \M^n$, and $f$ is a coordinate permutation on $\M^n$, then $f(C) \in \mathcal{C}$.\\
(ZL7) Let $a \to a'$ be a strongly good specialization  and let $rk(a \to a') \le 1$.\\
Then any specializations $ab \to a'b'$, $ac \to a'c'$  can be amalgamated: there exists $b^{*}$, independent from $c$ over $a$ such that $\textrm{t}^g(b^{*}/a)=\textrm{t}^g(b/a)$, and $ab^{*}c \to a'b'c'$.\\
(ZL8) Let $(a_i:i\in I)$ be independent and indiscernible over $b$.\\
Suppose $(a_i':i\in I)$ is indiscernible over $b'$, and $a_ib \to a_i'b'$ for each $i \in I$.
Further suppose $(b \to b')$ is a strongly good specialization  and $rk(b \to b') \le 1$.
Then, \mbox{$(ba_i:i \in I) \to (b'a_i':i\in I)$.}\\
(ZL9) Let $\kappa$ be a (possibly finite) cardinal and let  $a_i, b_i \in \M$ with $i < \kappa$, such that $a_0 \neq a_1$ and $b_0=b_1$. 
Suppose $(a_i)_{i<\kappa} \to (b_i)_{i <\kappa}$ is a specialization.
Assume there is some unbounded and directed $S \subset \mathcal{P}_{<\omega}(\kappa)$ satisfying the following conditions:
\begin{enumerate}[(i)]
\item  $0,1 \in X$ for all $X \in S$;
\item For all $X,Y \in S$ such that $X \subseteq Y$,  and for all sequences $(c_i)_{i \in Y}$ from $V$, the following holds: 
If $c_0=c_1$, $ (a_i)_{i \in Y} \to (c_i)_{i \in Y} \to (b_i)_{i \in Y},$
and $\textrm{rk}((a_i)_{i \in Y} \to (c_i)_{i \in Y}) \le 1$, then $\textrm{rk}((a_i)_{i \in X} \to (c_i)_{i \in X}) \le 1$.
\end{enumerate}
 
Then, there are $(c_i)_{i<\kappa}$ such that   
$$(a_i)_{i \in \kappa} \to (c_i)_{i \in \kappa} \to (b_i)_{i \in \kappa},$$
$c_0=c_1$ and $\textrm{rk}((a_i)_{i \in X} \to (c_i)_{i \in X}) \le 1$ for all $X \in S$.
  
\end{definition}
 
Let $a \in \M$.
Then, by (ZL2) and (ZL3), there is a unique $C \in \mathcal{C}$ such that $a$ is generic on $C$.
This is also the smallest $C \in \mathcal{C}$ such that $a \in C$.
This set determines the Galois type of $a$, and we call it the \emph{locus} of $a$.
 
\begin{remark}\label{ZL9implies}
We note that (ZL9) implies the dimension theorem of Zariski geometry (see \cite{HrZi}).
Indeed, suppose $\kappa=n$, a finite cardinal, and $a=(a_0, \ldots, a_{n-1})$ and $b=(b_0, \ldots, b_{n-1})$ are such that
$a \to b$, $a_0 \neq a_1$ and $b_0=b_1$.
Let $S=\{n\}$.
Then, conditions (i) and (ii) in (ZL9) hold, so we find an $n$-tuple $c$ such that
$a \to c \to b$, $U(a)-U(c) \le 1$ and $c_0=c_1$.
\end{remark}

In the following, we note that Zariski-like structures are indeed generalizations of Zariski geometries.
 
\begin{example}\label{ZariskiZariskiLike} 
Let $D$ be a Zariski geometry.
Since $D$ is strongly minimal, it is also quasiminimal.
Consider the collection of closed sets in the language.
Then, the irreducible (in the topological sense) ones among them satisfy the axioms (ZL1)-(ZL9).
Indeed, the axioms (ZL1)-(ZL6) are clearly satisfied.
It is well known that on a strongly minimal structure, $U$-ranks and Morley ranks coincide.
On a Zariski geometry, first order types imply Galois types.
Moreover, every strongly regular specialization is regular, and every strongly good specialization is good.
Hence, (ZL7) is Lemma 5.14 in \cite{HrZi} and (ZL8) is Lemma 5.15 in \cite{HrZi}.
(ZL9) holds by Compactness.
 \end{example}

\begin{example}
Consider the model class from Example \ref{equivalence}.
For each $n$, define the irreducible sets of $\M^n$ to be those definable with finite conjunctions of formulae of the form $x_i=x_j$ or $E(x_i, x_j) \wedge x_i  \neq x_j$.
In addition, we require that   
if $E(x_i, x_j) \wedge  x_i  \neq x_j$ belongs to the conjunction, then also $E(x_j, x_i) \wedge  x_i \neq x_j$ belongs there, that if 
both $E(x_i, x_j)  \wedge  x_i \neq x_j$ and $E(x_j, x_k) \wedge   x_j \neq x_k$ belong to the conjunction, then either $x_i=x_k$ or $E(x_i,x_k) \wedge  x_i \neq x_k$ belongs there,
and that if both $E(x_i, x_j) \wedge x_i \neq x_j$ and $x_i=x_k$ belong to the conjunction, then also $E(x_k, x_j) \wedge  x_k \neq x_j$ belongs there.

Now, it is easy to verify that the class satisfies the axioms for a Zariski-like structure.
\end{example} 
 
\begin{example}
Let $V$ be a vector space over $\mathbb{Q}$, and let $F$ be an algebraically closed field of characteristic $0$.
A \emph{cover} of the multiplicative group of $F$ is a structure represented by the exact sequence
$$0 \to K \to V \to F^* \to 1,$$
where the map $V \to F^*$ is a surjective group homomorphism from $(V,+)$ onto $(F^*, \cdot)$ with kernel $K$.
This is an example of a Zariski-like structure if we take the collection $\mathcal{C}$ to consist of the 
irreducible $\emptyset$-closed sets in the PQF-topology (see \cite{Lucy}, \cite{cover}).  
The proof can be found in \cite{lisuri} or \cite{cover}.
\end{example}
 
\begin{example}
Hrushovski \cite{newsm} has constructed a non locally modular strongly minimal set that does not interpret any group.
Since strongly minimal structures are quasiminimal, this is an example of a quasiminimal pregeometry structure.
However, it is not Zariski-like since our Theorem \ref{grouptheorem} implies that if Zariski-like, it would interpret a group. 
\end{example} 
 
We will show that if the pregeometry obtained from the bounded closure operator is non-trivial, then a group can be interpreted in a Zariski-like structure.
The locally modular case will follow from Remark \ref{locmodgrouprem}, so the non locally modular case is the more interesting one. 
In \cite{HrZi}, Hrushovski's group configuration theorem is used to show that there is a group in a
non locally modular Zariski geometry.
There authors introduce the concept of an indiscernible array and reformulate the group configuration theorem using it: if there is a certain kind of an array in the model, then there is a group configuration.
Then, they use properties of families of plane curves to construct a suitable array.
The crucial fact here is that if $D$ is a non locally modular Zariski geometry, then there is a family of plane curves that cannot be parametrized with a set of dimension less than $2$. 

Reformulating the group configuration in terms of indiscernible arrays generalizes rather straightforwardly to our setting and is done in section 4.2.
The main difference is that in the first order context it suffices that the arrays are infinite, while we need them to be uncountable.
To be able to follow the method from \cite{HrZi} to construct an array that yields a group, we also
need a suitable generalization for the notion of a family of plane curves presented in \cite{HrZi}.
We introduce it in section 4.1. 
Then, in section 4.3, we will follow the argument from \cite{HrZi} to construct an indiscernible array that yields a group.
Here, the main difference is that we don't have Compactness.
In \cite{HrZi}, the array is built by first constructing $n \times n$ -arrays for each natural number $n$, 
and then using Compactness to obtain an infinite array. 
Since we can't do this, we start from a very large array, use some combinatorial geometric tricks to make it indiscernible, and finally employ (ZL9) in place of Compactness.
  
\subsection{Families of plane curves} 
  

In \cite{HrZi}, a plane curve is defined simply as a $1$-dimensional closed and irreducible set.
This might seem intuitive also in our context, but it would make it difficult to define families of curves in a meaningful way.
Following \cite{HrZi}, we wish to define a family of plane curves as an irreducible set $C$ where the curves are parametrized using another irreducible set $E$.
Then, for each generic $e \in E$, we would expect the set $C(e)=\{x \, | \, (e,x) \in C\}$ to be a plane curve.
This makes sense in the context of Zariski geometries since there the irreducible sets are determined by a topology
which makes the notion of irreducibility in a sense more flexible than in our setting.
The trouble is that our collection $\mathcal{C}$ captures only the properties of the irreducible, $\emptyset$-closed sets in a Zariski geometry, not the properties of all irreducible closed sets.
Axiom (ZL1) requires the sets in $\mathcal{C}$ to be Galois definable over $\emptyset$, but arbitrary sets of the form $C(e)$ need not satisfy this.
However, if for a generic $(x,y) \in C(e)$, the tuple $(e,x,y)$ is generic in $C$, then (ZL3) guarantees that the generic elements of $C(e)$ have the same Galois type over $e$.
Thus, we introduce the following auxiliary notion and define families of plane curves in terms of it.

\begin{definition}\label{goodfor}
Let $C \subset \M^{n+m}$ be an irreducible set.
We say an element $a \in \M^n$ is \emph{good} for $C$ if there is some $b \in \M^m$ so that $(a,b)$ is a generic element of $C$.
\end{definition}

\begin{definition}
Let $\M$ be a Zariski-like structure, let $E \subseteq \M^n$ be irreducible, and let $C \subseteq \M^2 \times E$ be an irreducible set.
For each $e \in E$, denote $C(e)=\{(x,y) \in \M^2 \, | \, (x,y,e) \in C\}$.
Suppose now $e \in E$ is a generic point.
If $e$ is good for $C$ and the generic point of $C(e)$ has $U$-rank $1$ over $e$, then we say that $C(e)$ is a \emph{plane curve}.
We say $C$ is a \emph{family of plane curves} parametrized by $E$.  

Assume $e$ is good for $C$ and $(x,y,e)$ is generic on $C$.
We say that $\alpha$ is the \emph{canonical parameter} of the plane curve $C(e)$ if $\alpha=\textrm{Cb}(x,y/e)$ for a generic element $(x,y) \in C(e)$.
We define the $\emph{rank}$ of the family to be the $U$-rank of $\textrm{Cb}(x,y/e)$ over $\emptyset$, where $e \in E$ is generic, and $(x,y)$ is a generic point of $C(e)$.
\end{definition}

\begin{definition}
We say a family of plane curves $C \subset \M^2 \times E$ is \emph{relevant} if for a generic  $e \in E$ and a generic point $(x,y) \in C(e)$ it holds that $x,y \notin \textrm{bcl}(e)$.
\end{definition}

When proving that a one-dimensional group can be found from a Zariski-like structure, the non locally modular case will be the difficult one.
In this case, finding the group configuration will lean heavily on the fact that not being locally modular implies the existence of a relevant family of plane curves of rank at least $2$.
  
\begin{lemma}\label{relevantfamily}
Suppose $\M$ is a Zariski-like structure, and every relevant family of plane curves on $\M$ has rank $1$.
Then, $\M$ is locally modular.
\end{lemma}

\begin{proof}
This is similar to (ii) $\Rightarrow$ (iii) of Lemma 3.4. in \cite{Hy}.
See \cite{lisuri} for details.
\end{proof}
 
\subsection{Groups from indiscernible arrays} 
 
In the non locally modular case, we are going to use a relevant family of plane curves of rank at least $2$ to build the group configuration from Chapter 3.
As in \cite{HrZi}, we reformulate this configuration in terms of indiscernible arrays.

 \begin{definition}
We say that $A=(A_{ij} : i \in I, j \in J)$, where $I$ and $J$ are ordered sets, is an \emph{indiscernible array} over $B$ if whenever $i_1, \ldots, i_n \in I$, $j_1, \ldots, j_m \in J$, $i_1 < \ldots < i_n$, $j_1 <\ldots < j_m$, then $t((A_{i_{\nu} j_{\mu}}: 1 \le \nu \le n, 1 \le \mu, \le m)/B)$ depends only on the numbers $n$ and $m$.

If at least the $U$-rank of the above sequence depends only on $m,n$, and $U((A_{i_{\nu} j_{\mu}}: 1 \le \nu \le n, 1 \le \mu, \le m)/B)= \alpha(m,n)$, where $\alpha$ is some polynomial of $m$ and $n$, we say that $A$ is \emph{rank-indiscernible}  over $B$, of type $\alpha$, and write $U(A;n,m/B)=\alpha(n,m)$.
\end{definition}  
 
If $(c_{ij}: i \in I, j \in J)$ is an array and $I' \subseteq I$, $J' \subseteq J$ , we write $c_{I'J'}$ for $(c_{ij}: i \in I', j \in J'$).
If $|I'|=m$ and $|J'|=n$, we call $c_{I'J'}$ an $m \times n$ -\emph{rectangle} from $c_{ij}$.
 
\begin{lemma}
Let $f=(f_{ij} : i, j \in \kappa)$ be an indiscernible array over $A$, and let $\kappa \ge \o_1$.
Then, for all $m,n$, all the $m \times n$ rectangles of $f$ have the same Lascar type over $A$. 
\end{lemma}

\begin{proof}
Suppose not.
Let $m, n$ be such that all the $m \times n$ -rectangles don't have the same Lascar type over $A$.
Let $(B_k)_{k< \kappa}$ be a sequence of disjoint $m \times n$ -rectangles such that if $f_{ij} \in B_{k_1}$ and $f_{i'j'} \in B_{k_2}$, where $k_1 <k_2$, then $i<i'$ and $j<j'$.
There is some $I \subset \kappa$, $\vert I \vert =\kappa$ such that $(B_k)_{k \in I}$ is Morley over some model $\A \supset A$.
Relabel the indices so that $I=\kappa$.
By the counterassumption, there is some $m \times n$ rectangle $B$ such that $Lt(B/A) \neq Lt(B_0/A)$.
Let $0<\lambda < \kappa$ be such that whenever $f_{ij} \in B$ and $f_{i'j'} \in B_\lambda$, then $i<i'$ and $j<j'$.
Now, $B_0 B$ and $B_0 B_\lambda$ are both $2m \times 2n$ -rectangles, so $t(B_0 B/A)=t(B_0 B_\lambda /A)$.
This is a contradiction, since $Lt(B_0/A) \neq Lt(B/A)$, $Lt(B_0/A)=Lt(B_\lambda/A)$ and automorphisms preserve the equality of Lascar types.
\end{proof} 
 
The following lemma will yield the connection between the indiscernible arrays and the group configuration from Chapter 3.

\begin{lemma}\label{groupexists}
 Let $(f_{ij} : i,j< \omega_1)$  be an indiscernible array of elements of $\M$, of type  $\alpha(m,n)=m+n-1$ over some finite parameter set $B$.
Then there exists a Galois definable $1$-dimensional group in $(\M^{eq})^{eq}$. 
 \end{lemma}

\begin{proof}
We will show that there is in $\M$ a group configuration as in Definition \ref{conf}, and thus a Galois definable $1$-dimensional group by Theorem \ref{groupmain}. 
Let $\A$ be a countable model such that $B \subset \A$ and $f \da_B \A$
(note that we can find such a model by constructing a sequence $(a_i)_{i<\o}$ independent from $f$ over $B$, and then taking $\A=\textrm{bcl}(B, (a_i)_{i<\o})$).
We write $\textrm{bcl}_\A(X)$ for $\textrm{bcl}(\A \cup X)$.
To simplify notation, we assume that $B=\emptyset$.

We prove first an auxiliary claim. 
 
\begin{claim}\label{claimura}
Suppose $U(c/ d_1 d_2 \A)=U(c/d_1 \A)=U(c/d_2 \A)$.
Then there exists $e \in \textrm{bcl}_\A(d_1) \cap \textrm{bcl}_\A(d_2)$ such that $U(c/e \A)=U(c/d_1 d_2 \A)$.
\end{claim}

\begin{proof}
 
Let $E= \textrm{bcl}_\A(d_1) \cap \textrm{bcl}_\A(d_2)$. 
Let $c=(c_1, \ldots, c_m)$ and suppose $c_1, \ldots, c_k$ are independent over $\A$ from $d_1 d_2$ while $c \in \textrm{bcl}_\A(c_1, \ldots, c_k, d_1, d_2)$.
Then
\begin{displaymath}
c \in \textrm{bcl}_\A(c_1, \ldots, c_k, d_i, E)
\end{displaymath}
for $i=1,2$.
We will show that $c \in \textrm{bcl}_\A(c_1, \ldots, c_k, E)$.
Let $e \in E$ be a finite tuple such that $E=\textrm{bcl}_\A(e)$.
Suppose $U(d_2/\A d_1)=r$ and $U(d_2/\A \cup E)=r+l$.
We may assume without loss of generality that $d_2=e \cup \{d_{2,1}, \ldots, d_{2,r}, d_{2,r+1}, \ldots, d_{2,r+l}\}$ where $d_{2,1}, \ldots, d_{2,r}$ are independent over $\A d_1$ and $d_{2,r+1}, \ldots, d_{2,r+l} \in \textrm{bcl}_\A(d_1, d_{2,1}, \ldots, d_{2,r})$.
Now  
\begin{displaymath}
c \in \textrm{bcl}(c_1, \ldots, c_k, d_{2,1}, \ldots, d_{2,r}, \ldots, d_{2,r+l}, e,a)
\end{displaymath}
for some $a \in \A$ such that $d_{2,r+1}, \ldots, d_{2,r+l} \in \textrm{bcl}(d_1, d_{2,1}, \ldots, d_{2,r},a)$.
We will show that we can move the parameters $d_{2,1}, \ldots, d_{2,r}, \ldots, d_{2,r+l}$ one by one to $E$ using automorphisms.
We do this first for $d_{2,1}$.

We note first that $c \da_{\A d_1, d_{2,2}, \ldots, d_{2,r}} d_{2,1}$.
Indeed,
\begin{displaymath}
U (c/ \A d_1) \ge U(c/ \A, d_1, d_{2,2}, \ldots, d_{2,r}) \ge U(c/ \A, d_1, d_{2,1}, d_{2,2}, \ldots, d_{2,r})=U(c/\A, d_1, d_2)=U(c/\A d_1),
\end{displaymath}
so $U(c/ \A, d_1, d_{2,2}, \ldots, d_{2,r})=U(c/ \A, d_1, d_{2,1}, d_{2,2}, \ldots, d_{2,r})$.

We have $d_{2,1} \notin \textrm{bcl}_\A(d_1, d_{2,2}, \ldots, d_{2,r})$, and thus, $d_{2,1} \da_\A d_1 d_{2,2},\ldots d_{2,r}$.
Hence, by transitivity,
$$d_{2,1} \da_\A d_1 d_{2,2} \ldots, d_{2,r} c.$$
Then, there is some finite set $A \subset \A$ such that $a \in A$, $d_{2,1} \da_A d_1 d_{2,2} \ldots, d_{2,r} c$ and $d_1 d_{2,2} \ldots, d_{2,r} c \da_A \A$.
Let $d_{2,1}' \in \A$ be such that $Lt(d_{2,1}'/A)=Lt(d_{2,1}/A)$.
Now, there is some $f \in \textrm{Aut}(\M/A d_1 d_{2,2} \ldots, d_{2,r} c)$ such that $f(d_{2,1})=d_{2,1}'$.
For $1 \le i \le l$, denote $d_{2,r+i}'=f(d_{2,r+i})$.
Then, we have
\begin{displaymath}
c \in \textrm{bcl} (c_1, \ldots, c_k, d_{2,2}, \ldots, d_{2,r}, d_{2,r+1}' \ldots, d_{2,r+l}', e, d_{2,1}', a).
\end{displaymath}

We now repeat the above argument with $d_{2,2}$ in place of $d_{2,1}$.
When choosing a finite set $A' \subset \A$ such that $d_{2,2} \da_{A'} d_1 d_{2,3} \ldots, d_{2,r} c$ and $d_1 d_{2,3} \ldots, d_{2,r} c \da_{A'} \A$, we take care that $a,d_{2,1}' \in A$.
After doing the argument $r$ times, we have obtained elements $d_{2,r+1}^* \ldots, d_{2,r+l}'^* \in \textrm{bcl}_\A(d_1)$ such that 
\begin{displaymath}
c \in \textrm{bcl}_\A(c_1, \ldots, c_k, d_{2,r+1}^* \ldots, d_{2,r+l}^*, e).
\end{displaymath} 
If $d_{2,r+1}^* \ldots, d_{2,r+l}^* \in E$, we are done.

If not, there are some numbers $0<n \le m \le l$ such that (after renaming the elements in $\{d_{2, r+1}^{*}, \ldots, d_{2, r+l}^{*} \} \setminus (E \cup \A)$)  we have $U(d_{2,1}^{*}, \ldots, d_{2,n}^{*}, d_{2,n+1}^{*}, \ldots, d_{2,m}^{*}/E \cup \A)=m$, $U(d_{2,1}^{*}, \ldots, d_{2,n}^{*}/\A d_2)=n$, and $d_{2,n+1}^{*}, \ldots, d_{2,m}^{*} \in \textrm{bcl}_\A(d_2, d_{2,1}^{*}, \ldots, d_{2,n}^{*})$.
As $d_{2,1}^{*}, \ldots, d_{2,n}^{*} \in \textrm{bcl}_\A(d_1)$, we have $n \le U(d_1/E \cup \A)$.
Thus
\begin{displaymath}
U(c/\A d_2) \ge U(c/\A,d_2, d_{2,1}^{*}, \ldots, d_{2,n}^{*}) \ge U(c/ \A, d_1,d_2)=U(c/ \A d_2),
\end{displaymath}
so 
$$c \da_{ \A d_2} d_{2,1}^{*} \ldots d_{2,n}^{*},$$
and thus e.g. $d_{2,1}^{*} \da_{ \A, d_2, d_{2,2}^{*}, \ldots, d_{2,n}^{*}} c.$
Hence we may move $d_{2,1}^{*}, \ldots, d_{2,n}^{*}$ to $E$ with the same process as before with $d_2$ in place of $d_1$.
We keep repeating the process, and as at every step we move one element to $E$, we will eventually have moved them all, so we get $c \in \textrm{bcl}_\A (c_1, \ldots, c_k, E)$ as wanted. 

\end{proof}

From now on, we will simplify the notation by assuming that the elements of $\A$ are symbols in our language.

Let $a=f_{1,2}$, $c=f_{2, 2}$, $y=f_{1,3}$, $z=f_{2,3}$.
We will find elements $x$ and $b$ so that $\{a,b,c,x,y,z\}$ will form a group configuration.

Let $d=(f_{3, 2}, f_{3,3})$.
One can compute using the type $\alpha$ of the array that 
\begin{displaymath}
U(d/a y)=U(d/c z)=U(d/a c z y)=1.
\end{displaymath}
Thus, by Claim \ref{claimura}, there exists $x \in \textrm{bcl}(a y) \cap \textrm{bcl}(c z)$ such that $U(d/x)=1$.
We prove that $U(x)=1$.
We have $U(x) \geq 1$ since $U(d)=2$.
Now
\begin{displaymath}
3-U(x)=U(a y c z/x) \leq U(a y/x)+ U(c z/ x)=U(a y)+U(c z) -2U(x)=4- 2 U(x),
\end{displaymath}
where we use the type of the array and the fact that $x \in \textrm{bcl}(a y) \cap \textrm{bcl}(c z)$.
Thus, $U(x) \leq 1$.

Let $a'=f_{1, 1}$, $c'=f_{2,1}$.
By the type of the array, 
\begin{displaymath}
U(yz/a c)= U(yz /a' c')=U(yz /a c a' c') =1.
\end{displaymath}
By Claim \ref{claimura}, there exists $b \in \textrm{bcl}(a c) \cap \textrm{bcl}(a' c')$ such that $U(yz/b)=1$.
Similarly as before, $U(b)=1$.

It is clear from the type of the array that 
\begin{displaymath}
U(z)=U(y)=U(c)=U(a)=1,
\end{displaymath}
and
\begin{displaymath}
U(z, y)=U(a, c) =U(a, y) =U(c, z)=2.
\end{displaymath}
Also, 
\begin{displaymath}
U(a, b, c)=U(a, y, x)=U(z, y, b)=U(z, c, x)=2,
\end{displaymath}
and
\begin{displaymath}
U(z,x,y,a,b,c)=U(z,y, a,c)=3
\end{displaymath}
by the type of the array and the choice of $x$ and $y$.
Thus, we are left to prove that the rest of the pairs have $U$-rank $2$ and that the rest of the triples have $U$-rank $3$.

We prove first that $U(a, c, y)=U(a, c, z)=3$ (and it of course follows that $U(y,c)=U(z, a)=2$).
Suppose that $y \in \textrm{bcl}(a, c)$.
Consider the concatenated sequence $(f_{i,2}f_{i,3})_{i < \omega_1}$.
Now, there is some stationary set $S \subseteq \omega_1$ and some model $\B$ such that the sequence $(f_{i,2}f_{i,3})_{i \in S}$ is Morley over $\B$.
Let $j, k \in S$ be such that $j<k$.
Since the sequence $(f_{i,2}f_{i,3})_{i < \omega_1}$ is order indiscernible, there is some automorphism $g$ of $\M$ such that $g(f_{1,2} f_{1,3})=f_{j,2} f_{j,3}$ and $g(f_{2,2} f_{2,3})=f_{k,2} f_{k,3}$.
Since $(f_{i,2}f_{i,3})_{i \in S}$ is Morley over $\B$, there is an automorphism $\pi \in \textrm{Aut}(\M/\B)$ such that
$\pi(f_{j,2} f_{j,3})=f_{k,2} f_{k,3}$ and $\pi(f_{k,2} f_{k,3})=f_{j,2} f_{j,3}$.
The map $g^{-1} \circ \pi \circ g$ is an automorphism taking  $f_{1,2} f_{1,3} \mapsto f_{2,2} f_{2,3}$,  and $f_{2,2} f_{2,3} \mapsto f_{1,2} f_{1,3}$.
Hence
\begin{displaymath}
t(f_{1,3}f_{1,2} f_{2,2}/\emptyset)=t(f_{2,3}f_{2,2} f_{1,2}/\emptyset).
\end{displaymath}
So, $z \in \textrm{bcl}(a, c)$ and $U( a, c,y,z)=2$ which is a contradiction (by the type of the array it should be $3$).
One proves similarly that $z \notin \textrm{bcl}(a, c)$.

Similarly one shows that $U(c, y, z)=3$.
 
Now we prove $U(x, z)=2$.
Suppose not.
Then, $x \in \textrm{bcl}(z).$
We chose $x$ so that $U(d/x)=1$.
As $U(d)=2$, we have $U(d/z)=1$ and thus $U(d, z)=2$.
So, $z \in \textrm{bcl}(d)= \textrm{bcl}(f_{3,2}, f_{3,3})$.
By the indiscernibility of the array, $y \in \textrm{bcl}(c, z)$, and we already proved this is not the case.
Similarly, $U(x, y)=2$.

For the rest of the conditions needed for $\{a,b,c,x,y,z\}$ to be a group configuration, one uses properties of pregeometries and the conditions that we already proved.
Eventually, we will have obtained a group configuration over $\A$,
and there is some finite $A \subset \A$ so that the configuration is over $A$.
Hence, we may apply Theorem \ref{groupmain} to see that the group exists.
\end{proof}

\subsection{Finding the group}

In this section we will prove that the group exists.
We will need the technical lemma that follows.
Recall that the closure operator $\textrm{bcl}_X$ is defined so that $bcl_X(Y)=bcl(XY)$, and write $dim_{bcl_X}$ for the dimension in the pregeometry given by $bcl_X$.

\begin{lemma}\label{tekn1}
Let $(A_{ij}: 1 \le i \le M, 1 \le j \le N)$, $M, N \ge 2$, be a subarray of an indiscernible array of size $\o_1 \times \o_1$ over some finite tuple $b$.  
Assume $U(A;m,n/b)=m+n$ for any $m \le M$, $n \le N$, and that $\textrm{dim}_{\textrm{bcl}_b}(\textrm{dcl}(A_{12} A_{22}b) \cap \textrm{dcl}(A_{11} A_{21}b))=2$.
 
Let $b(A_{ij}) \to b(a_{ij})$ be a rank-$1$ specialization.
Suppose $Lt(a_{ij}/ b)$ is constant with $i$, $j$,   $U(a_{ij}/b)=1$ for each pair $i$, $j$, and $U(a;2,1/b)=2$.
Also assume $bA_{ij} A_{i'j} \to ba_{ij} a_{i'j}$ is strongly good for any $i$, $i'$, $j$.
Then $a$ is a rank-indiscernible array of type $m+n-1$ over $b$.
\end{lemma}

\begin{proof} 
Similar to the proof of Lemma 6.8 in \cite{HrZi} (see \cite{lisuri} for details).
\end{proof}

Now we are ready to prove the existence of a group.

\begin{theorem}\label{grouptheorem}
Let $\M$ be a Zariski-like structure with a non-trivial pregeometry.
Then, there exists a Galois definable one-dimensional group in $(\M^{eq})^{eq}$.
\end{theorem}

\begin{proof}  
If $\M$ is locally modular, then there is a group by Remark \ref{locmodgrouprem}.
 
So suppose $\M$ is non locally modular.
By Lemma \ref{relevantfamily}, there exists a relevant family of plane curves that has rank $r \ge 2$.
Let $\alpha$ be the canonical parameter for one of the curves in this family, and suppose $U(\alpha)=r$.
Let $(x,y)$ be a generic point on this curve, i.e. $\alpha=\textrm{Cb}(x,y/\alpha)$. 
Since the family is relevant, we have $x \da \alpha$, $y \da \alpha$.
We also have $x \da y$, because otherwise $1=U(xy/\emptyset)=U(xy/\alpha)$, so $xy \da_\emptyset \alpha$, which would imply $\alpha \in \textrm{bcl}(\emptyset)$ since $\alpha$ is a canonical parameter.
This is a contradiction, since the family has rank $r \ge 2$.
 
Let $c_1, \ldots, c_r, d_1, \ldots, d_r$  be such that $t(c_i,d_i/\alpha)=t(x,y/\alpha)$ for $1 \le i \le r$,
and the sequence $x, c_1, \ldots, c_r$ is independent over $\alpha$.
We claim that $U(c_1,d_1, \ldots, c_r,d_r)=2r$.
For this, we first show that $\alpha \in \textrm{bcl}(c_1,d_1, \ldots, c_r,d_r)$.
 We have $c_1 \in \textrm{bcl}(d_1, \alpha)$, so $c_1 \nda_\alpha d_1$, and thus
$$U(c_1/d_1)=U(c_1)=U(c_1/\alpha)>U(c_1/d_1\alpha),$$
so $c_1 \nda_{d_1} \alpha$.
Hence, $U(\alpha/d_1)>U(\alpha/c_1d_1)$.
Thus, $U(\alpha/ c_1 d_1) \le r-1$.
 
If $U(\alpha/c_1 d_1, \ldots, c_k d_k)=0$ for some $k \le r$, then $\alpha \in \textrm{bcl}(c_1 d_1, \ldots, c_k d_k)$ so we are done. 
So suppose $0<U(\alpha/c_1 d_1, \ldots, c_k d_k) \le r-k$ for some $k<r$.
We claim that $U(\alpha/c_1 d_1, \ldots, c_k d_k, c_{k+1} d_{k+1}) \le r-k-1$.
Suppose towards a contradiction that
$$U(\alpha/c_1 d_1, \ldots, c_k d_k)=U(\alpha/c_1 d_1, \ldots, c_k d_k, c_{k+1} d_{k+1}).$$
Then, 
\begin{eqnarray}\label{rairai}
c_{k+1} d_{k+1} \da_{c_1 d_1, \ldots, c_k d_k}  \alpha.
\end{eqnarray}
We have
\begin{eqnarray}\label{rairai2}
\alpha=\textrm{Cb}(c_{k+1} d_{k+1}/\alpha)=\textrm{Cb}(c_{k+1} d_{k+1}/\alpha, c_1 d_1, \ldots, c_k d_k),
\end{eqnarray}
where the second equality follows from the fact that $c_{k+1} d_{k+1}Ê\da_\alpha c_1 d_1, \ldots, c_k d_k.$
From (\ref{rairai}) and and (\ref{rairai2}), it follows that $\alpha \in \textrm{bcl}( c_1 d_1, \ldots, c_k d_k),$ a contradiction.
 
Thus, $\alpha \in \textrm{bcl}(c_1, d_1, \ldots, c_r,d_r)$.
Since the sequence $(c_id_i)_{1 \le i \le r}$ was chosen to be independent over $\alpha$, Theorem \ref{main} (x) gives
$$U(\alpha)+U(c_1 d_1/\alpha)+ \ldots + U(c_r d_r/\alpha)=r+r=2r,$$
so $U(c_1, d_1, \ldots, c_r,d_r)=2r$. 

Next, we show that for $1 \le k \le r$, $U(\alpha/c_1, \ldots, c_r, d_1, \ldots, d_k)=r-k$.
Indeed, since $d_1, \ldots, d_r \in \textrm{bcl}(\alpha, c_1, \ldots, c_r)$, we have
\begin{eqnarray*}
2r&=&U(\alpha, c_1, \ldots, c_r, d_1, \ldots, d_k)=U(c_1, \ldots, c_r, d_1, \ldots, d_k)+U(\alpha/c_1, \ldots, c_r, d_1, \ldots, d_k) \\ &=&(r+k)+U(\alpha/c_1, \ldots, c_r, d_1, \ldots, d_k). 
\end{eqnarray*} 

Let now $C$ be the locus of $(x,y,c_1, \ldots, c_r, d_1, \ldots, d_r)$ and $E$ the locus of $(c_1, \ldots, c_r, d_1, \ldots, d_r)$.
Then, $C$ is a family of plane curves parametrized by $E$, and 
 $C(c_1, \ldots, c_r, d_1, \ldots, d_r)$ is a curve in this family.
Denote $d=(c_1, \ldots, c_r, d_1, \ldots, d_{r-2})$ and $e_0=(d_{r-1}, d_r)$.
Since $xy \da_\alpha c_i d_i$ for each $i$, we have $\alpha=\textrm{Cb}(x,y/c_1, \ldots, c_r, d_1, \ldots, d_r)$. 
It is interbounded with $e_0$ over $d$.
Since the sequence $de_0=(c_1, \ldots, c_r, d_1, \ldots, d_r)$ was seen to be independent, we have $U(e_0/d)=2$.

Let $e \in E(d)$ be a generic element.
We now write  $C(e; a,b)$ for "$(a,b)$ is a generic point of $C(ed)$". 
We write $C^2(e; ab, a'b')$ if the following hold:
\begin{enumerate}
\item $C(e; a, b)$ and $C(e; a', b')$; 
\item $ab \downarrow_{de} a' b'$; 
\item $Lt(ab/de)=Lt(a'b'/de).$  
\end{enumerate}

\begin{claim}\label{aclspec}
Suppose $a \neq a'$, $b \neq b'$, and $C^2(e; ab, a'b')$.
Then,
\begin{enumerate}[(i)]
\item $e \in \textrm{bcl}(d, a,b,a',b')$.
\item $U(aba'b'/d)=4$.
\item  $deaba' b' \to deabab$. 
\end{enumerate}
\end{claim}

\begin{proof}
Since $e$ is interbounded with $\textrm{Cb}(a,b/de)$, (i) and (ii) hold by similar arguments that were used above to calculate that
$\alpha \in \textrm{bcl}(c_1, \ldots, c_r, d_1, \ldots, d_r)$ and
$U(c_1, \ldots, c_r, d_1, \ldots, d_r)=2r$.

 we see using similar arguments as above, that (i) holds.
 
 For (ii), we will apply (ZL8).
Let $\A$ be a model such that $de \in \A$ and $aba'b' \da_{de} \A$.
Then, $ab \da_{dea'b'} \A$, and since $ab \da_{de} a'b'$, we get by transitivity that $ab \da_{de} a'b' \A$, which implies $ab \da_\A a'b'$.
On the other hand, we have $Lt(ab/de)=Lt(a'b'/de)$, $ab \da_{de} \A$ and $a'b' \da_{de} \A$, so $Lt(ab/\A)=Lt(a'b'/\A)$.
Thus, we may extend $(ab, a'b')$ to a Morley sequence over $\A$.
It follows that $a'b'$ and $ab$ are strongly indiscernible over $de$.
Of course also $ab$ and $ab$ are strongly indiscernible over $de$ (just repeat $ab$ arbitrarily many times to extend the sequence).
Clearly $dea'b' \to deab$ and $deab \to deab$, $\textrm{rk}(de \to de)=0 \le 1$, and $de \to de$ is strongly good.   
Hence, we may apply (ZL8) to get $deaba' b' \to deabab$.
\end{proof}
 
Pick some generic point $e \in E(d)$, and independent generics $a_0, b_0, a, b \in \M$ such that $C^{2}(e; ab, a_0 b_0)$ (after we have $a$ and $b$, we can use Theorem \ref{main} (iv) to find $a_0$ and $b_0$). 
Let $\lambda$ be a cardinal such that $\lambda \ge \beth_{(2^{\omega_1})^+}$, and let $\kappa$
be a cardinal such that $\kappa \ge \beth_{(2^\lambda)^+}$.
Let $a_i, b_j$, $0<i,j <\kappa$ be a sequence of  generic elements of $\M$ independent over $d$ such that $Lt(a_i b_j /d a_0 b_0)=Lt(ab/ d a_0 b_0)$ for all $i,j$.  
For each pair $i,j$, let $f_{ij}$ be an automorphism fixing $a_0, b_0, d$ such that $f_{ij}(a,b)=(a_i,b_j)$.
Denote $e_{ij}=f_{ij}(e)$.
Then, $C^2(e_{ij}; a_i b_j, a_0 b_0)$ holds  for each pair $i,j$.
Let $A_{ij}=(a_i, b_j, e_{ij})$, $A=(A_{ij})_{i,j \ge 1}$.
We will next show that we can find an indiscernible array of size $\o_1 \times \o_1$ such that each one of its finite subarrays is isomorphic to some finite subarray of $A$.
 
For each $i< \kappa$, denote $A_{i, <\lambda}=(A_{ij} | j< \lambda)$.
Using Erd\"os-Rado and an Ehrenfeucht-Mostowski construction, one finds a  sequence $(A'_{i, < \lambda})_{i<\o_1}$ such that every finite permutation of the sequence preserving the order of the indices $i$ extends to 
some $f \in \textrm{Aut}(\M/d a_0 b_0)$.
Moreover, an isomorphic copy of every finite subsequence can be found in the original sequence $(A_{i, < \lambda})_{i<\kappa}$.
This construction is due to Shelah and can be found in e.g. \cite{CatTran}, Proposition 2.13. (for a proof, see \cite{meeri}, paper II, Proposition 2.13).  
There it is done for a sequence of finite tuples (whereas we have a sequence of sequences of length $\lambda$), but the proof is similar in our case.
 
We may now without loss  assume that $(A_{i,< \lambda}')_{i<\o_1}$ are the $\o_1$ first elements in the sequence 
$(A_{i,< \lambda})_{i<\kappa}$.
We may now apply the same argument to $(A_{<\o_1, j}')_{j<\lambda}$ to obtain an array $(A_{<\o_1, j}'')_{j<\o_1}$. 
This is an array of size $\o_1 \times \o_1$, indiscernible over $d a_0 b_0$, and we may assume it is a subarray of the original array $A$.
From now on, we will use $A$ to denote this indiscernible array of size $\o_1 \times \o_1$.
 
We write $x \to*y$ for $(x, d, a_0, b_0) \to (y, d, a_0, b_0)$.
Let $A'_{ij}=A_{i1}$ for $j \ge 1$, and write $A'=(A'_{ij})_{0<i,j<\omega_1}$.

\begin{claim}\label{ekac}
$A \to^*A'$.
\end{claim} 
\begin{proof}
For each $i<\o_1$, consider the the sequence $(A_{ij})_{j<\o_1}$.
By Lemma \ref{fodor}, there is some cofinal set $X_i \subset \o_1$ such that $(A_{ij})_{j \in X_i}$ is Morley, and thus strongly indiscernible, over $d a_0 b_0$.
For each $j$, we have $A_{ij} d a_0 b_0 \to A_{i1} d a_0 b_0$.
Moreover, $\textrm{rk}(d a_0 b_0 \to d a_0 b_0)=0 \le 1$ and $d a_0 b_0 \to d a_0 b_0$ is a strongly good specialization.
Also, $(A'_{ij})_{j \in X_i}$ is strongly indiscernible (since it just repeats the same entry).
Thus, by (ZL8), there is, for each $i$, a specialization 
$(A_{ij})_{j \in X_i} \to^* (A'_{ij})_{j \in X_i}.$
If we enumerate the set $X_i$ again, using the order type of $\o_1$, then we get (we still use the notation with the index set $X_i$ to denote that the sequence so indexed is the Morley one)
$$(A_{ij})_{j \in \o_1}\to^*(A_{ij})_{j \in X_i} \to^*(A'_{ij})_{j \in X_i}\to^*(A_{ij}')_{j \in \o_1},$$
so in particular $(A_{ij})_{j \in \o_1} \to^*(A'_{ij})_{j \in \o_1}$.

To prove that $A \to^*A'$, it suffices to show $(A_{ij})_{i<\o_1, j \in J} \to^*(A_{ij}')_{i<\o_1, j \in J}$ for all finite $J \subset \o_1$.
So, let $J \subset \o_1$ be finite.
Since $(A_{ij})_{j \in \o_1} \to^* (A'_{ij})_{j \in \o_1}$ holds for every $i$, we have  $(A_{ij})_{ j \in J} \to^*(A_{ij}')_{ j \in J}$ for every $i \in \o_1$.
Thus, applying (ZL8) similarly as we did above, we obtain $(A_{ij})_{i<\o_1, j \in J} \to^*(A_{ij}')_{i<\o_1, j \in J}$, as wanted.
It then follows that $A \to^*A'$.  
\end{proof} 

Let $A''_{ij}=(a_0, b_0, e_{i1})$, and write $A''=(A''_{ij})_{0<i,j<\o_1}$.

\begin{claim}\label{tokac}
$A' \to^*A''$.
\end{claim}
\begin{proof} 
As $A'_{ij}$ and $A''_{ij}$ do not depend on $j$ and as specializations respect repeated entries, it suffices to show that $(A'_{i1}:i) \to^*(A''_{i1}:i)$.
By Claim \ref{aclspec} (iii), $d a_0 b_0 b_1 \to d a_0 b_0 b_0$.
By Claim \ref{aclspec} (ii), we have $U(a_0 b_0 a_1 b_1/d)=4$, so $(d, a_0, b_0, b_1)$ is a generic point of $\M^{2r+1}$ (remember that $d$ is an independent tuple), and this is a strongly good specialization.
It is also clearly of rank $1$.
By Claim \ref{aclspec} (iii), $(a_i, b_1, e_{i1}, d, a_0, b_0) \to (a_0, b_0, e_{i1}, d, a_0, b_0)$ for every given $i$.
Thus, we may apply (ZL8) similarly as in the proof of the previous claim.
\end{proof}

\begin{claim}\label{aputulos}
If $(i,j) \neq (i',j')$, then $e_{ij} \downarrow_{d a_0 b_0} e_{i'j'}$.
\end{claim}
\begin{proof}
Suppose not.
By the same arguments that we used to prove Claim \ref{aclspec} (i), $U(e_{ij}/d a_0 b_0)=1$, and $e_{ij} \in \textrm{bcl}(d a_0 b_0 a_i b_j)$.
From the first of these statements it follows that $e_{ij} \in \textrm{bcl}(d a_0 b_0 e_{i'j'})$, since the counterassumption gives us $U(e_{ij}/da_0b_0e_{i'j'})<U(e_{ij}/da_0b_0)$.
From the second statement it follows that
$a_i b_j$ dominates $e_{ij}$ over $d a_0 b_0$.
Similarly, $a_{i'} b_{j'}$ dominates $e_{i'j'}$ over $d a_0 b_0$.
Suppose first $i \neq i'$ and $j \neq j'$.
Then, $a_i b_j \downarrow_{d a_0 b_0} a_{i'} b_{j'}$ (the sequence was chosen to consist of elements independent over $d$),  and by domination $e_{ij} \downarrow_{d a_0 b_0} e_{i'j'}$, a contradiction.
 
Suppose now $i=i'$ and $j \neq j'$ (the other case is symmetric).
Similarly as before, we get that $b_j$ dominates $e_{ij}$ over $d a_0 b_0 a_i$ and $b_{j'}$ dominates $e_{ij'}$ over $d a_0 b_0 a_i$.
As $b_j \downarrow_{d a_0 b_0 a_i} b_{j'}$, we get that $e_{ij} \downarrow_{d a_0 b_0 a_i} e_{ij'}$.
Thus, to get a contradiction it suffices to show that $a_i \downarrow_{a_0 b_0} e_{ij}$, since $e_{ij} \downarrow_{d a_0 b_0} e_{i'j'}$ then follows by transitivity.  
Suppose not.
As $U(e_{ij}/d a_0 b_0)=1$, we must now have $e_{ij} \in \textrm{bcl}(d a_0 b_0 a_i)$.
But then we have $b_j \in \textrm{bcl}(d a_i e_{ij}) \subseteq \textrm{bcl}(d a_0 b_0 a_i)$ which is a contradiction since the sequence $a_0, b_0, a_i, b_j$ was chosen to be independent over $d$.
\end{proof}
 
By claims (\ref{ekac}) and (\ref{tokac}), $A \to^*A''$.
We will apply (ZL9) to this specialization and eventually obtain an infinite rank-indiscernible array $A^*$ such that
$A \to A^* \to A''$.
The array $A^*$ will be of type $m+n-1$ over the parameters $d a_0 b_0$, as desired.

Let now $A^0$ be a finite subarray of $A$ containing the entry $A_{11}$, and let ${A^0}''$ be the corresponding finite subarray of $A''$.
The specialization $A \to^* A''$ induces a specialization $A^0 \to^* {A^0}''$.
After suitably rearranging the indices, we may view these finite arrays as tuples and
assume that the tuple on the left begins with ``$a_0 a_1 \ldots$", whereas the tuple on the right begins with ``$a_0 a_0 \ldots$".
By Remark \ref{ZL9implies}, the dimension theorem holds, and thus there is a finite array ${A^0}^*$
such that $d a_0 b_0 A^0 \to  d' a_0' b_0'  {A^0}^*  \to d a_0 b_0 {A^0}''$ for some $d', a_0', b_0'$, ${A^0}^*_{11}=a_0'b_1^*e_{11}^*$ for some $b_1^*$, $e_{11}^*$, and $U(A^0)-U({A^0}^*) \le 1$.
In particular, we have $d a_0 b_0 \to  d' a_0' b_0' \to d a_0 b_0$.
By (ZL3) this implies that $t^g (d a_0 b_0/\emptyset)=t^g (d' a_0' b_0'/\emptyset)$.
Thus, we may assume that $d' a_0' b_0'=d a_0 b_0$ (if it is not, then just apply to the array ${A^0}^*$ an automorphism taking $d' a_0' b_0' \mapsto d a_0 b_0$).
In particular, we may assume $A^0 \to^* {A^0}^*  \to^*{A^0}''$ and ${A^0}^*_{11}=a_0b_1^*e_{11}^*$ for some $b_1^*$ and $e_{11}^*$.

\begin{claim}
The array ${A^0}^*$ is of type $m+n-1$ over $da_0b_0$.
\end{claim}

\begin{proof}
We will show that the assumptions posed for $A$ and $a$ in Lemma \ref{tekn1} hold for $A^0$ and ${A^0}^*$, respectively, over the parameters $d a_0 b_0$.
The claim then follows from the lemma.

By  Claim \ref{aclspec} (i), $e_{ij} \in \textrm{bcl}(d a_0 b_0 a_i b_j)$.
Thus, as the elements $a_i, b_j$ were chosen to be independent over $d$ for $i,j \ge 0$, we have $U(A; 1, 1/d a_0 b_0)=2$, and it is easy to show by induction that $U(A; m, n/ d a_0  b_0)=m+n$.
Write $C=A_{11} A_{21}$ and $C'=A_{12} A_{22}$.
Now $U(C/d a_0 b_0)=U(C'/d a_0 b_0)=3$, and $U(C \cap C'/d a_0 b_0)=2$.
Thus, 
$$2 \le \textrm{dim}_{\textrm{bcl}}(\textrm{dcl}(Cda_0b_0) \cap \textrm{dcl}(C'da_0b_0)/d a_0 b_0) \le 3.$$ 
Denote $X=\textrm{dcl}(Cda_0b_0) \cap \textrm{dcl}(C'da_0b_0)$, and suppose $\textrm{dim}_{\textrm{bcl}}(X/d a_0 b_0)=3$.
Since $X \subseteq \textrm{bcl}(Cda_0b_0)$,  we must have $\textrm{bcl}(X)=\textrm{bcl}(Cda_0b_0)$.
But this is impossible since $b_1 \in \textrm{bcl}(Cda_0b_0) \setminus \textrm{bcl}(X)$.
Thus, $\textrm{dim}_{\textrm{bcl}}(X/d a_0 b_0)=2$.

Similarly as in the proof of Lemma 6.10 in \cite{HrZi} (see \cite{lisuri} for details), one shows that ${A^0}^*$ satisfies the assumptions posed for $a$  in Lemma \ref{tekn1} over the parameters.
By Lemma \ref{tekn1}, ${A^0}^*$ is of type $m+n-1$ over $da_0b_0$.

\end{proof}

Next, we apply (ZL9) to the specialization $A \to^*A''$ to eventually obtain an infinite indiscernible array of type $m+n-1$ over $da_0b_0$.
Enumerate the elements on the left side of the specialization so that $a_0$ is the element enumerated by $0$ and $a_1$ the element enumerated by $1$, and use a corresponding enumeration on the right side (there, both the element enumerated by $0$ and the element enumerated by $1$ will be $a_0$).
Let  $S$ be a collection of index sets corresponding to all $m \times n$ subarrays of $A$ containing the entry $A_{11}$ for all natural numbers $m,n$.
Moreover, we add $0$ to every $X \in S$.
The set $S$ is unbounded and directed, and by what we just proved,  every $X \in S$ corresponds to an array $A^*_X$ of type $m+n-1$ over $da_0b_0$ (we get the correspondence by removing the element indexed by $0$ from each $X$).
Thus, the conditions of (ZL9) hold for the set $S$, and hence we 
obtain an infinite array $A^*$ where each $m \times n$ -subarray containing the entry $A_{11}^*$ has $U$-rank $m+n-1$ over $da_0 b_0$ (note that each $m \times n$ -subarray of $A$ has rank $m+n$).

We claim that $A^*$ is actually of type $m+n-1$ over $d a_0 b_0$.
To prove this, let ${A^0}^*$ be an arbitrary $m_0 \times n_0$ subarray of $A^*$.
Then, there is some $(m_0+1) \times (n_0+1)$ subarray ${A^1}^{*}$ of $A^*$ such that $A_1^*$ contains the entry $A_{11}^*$ and ${A^0}^*$ is a subarray of ${A^1}^*$.
We have already shown that ${A^1}^*$ is of type $m+n-1$ over $d a_0 b_0$.
Hence, $U(A_0^*/d a_0 b_0)=m_0+n_0-1$, as wanted.

We wish to apply Lemma  \ref{groupexists} to show that there is a $1$-dimensional Galois-definable group in $(\M^{eq})^{eq}$.
For this, we need $A^*$ to be indiscernible.
However, if in the beginning of this proof, when we started to construct the array $A$, we have chosen the cardinals 
$\kappa$ and $\lambda$ to be large enough, then we may assume that $A$ and thus $A^*$ is big enough that we may apply the Shelah trick again.
Thus, we may without loss suppose that $A^*$ is indiscernible.
By Lemma \ref{groupexists}, there is a $1$-dimensional Galois-definable group in $(\M^{eq})^{eq}$.
\end{proof}

\end{document}